\documentclass[12pt]{amsart}
\usepackage[a4paper, left=3cm, right=1cm]{geometry}
\usepackage[normalem]{ulem}

\usepackage{stackengine,scalerel}
\usepackage[utf8]{inputenc}
\usepackage[sans]{dsfont}
\usepackage[T1]{fontenc}
\usepackage{color,colortbl, xcolor}
\usepackage{amsmath,amssymb,amsfonts,amsthm,amscd, bm, leftidx,mathtools}
\usepackage{enumerate}
\usepackage{hyperref}

\usepackage{tikz}
\usepackage{tikz-cd}
\usepackage{xypic}

\usepackage{centernot}

\def\tang{\ThisStyle{\abovebaseline[0pt]{\scalebox{-1}{$\SavedStyle\perp$}}}}
\tikzset{
	labl/.style={anchor=south, rotate=90, inner sep=10mm}
}
\usetikzlibrary{arrows,positioning,chains,matrix,scopes,babel}

\makeatletter
\newcommand*{\rom}[1]{\expandafter\@slowromancap\romannumeral #1@}
\makeatother

\newcommand\msmall[1]{\mbox{\small\ensuremath{#1}}}

\newtheorem{thm}{Theorem}[section]
\newtheorem{prob}[thm]{Problem}
\newtheorem{prop}[thm]{Proposition}
\newtheorem{lemma}[thm]{Lemma}
\newtheorem{cor}[thm]{Corollary}
\newtheorem{ex}[thm]{Example}
\newtheorem{remark}[thm]{Remark}

\theoremstyle{definition}
\newtheorem{defn}[thm]{Definition}

\newenvironment{myproof}[2] {\paragraph{\it Proof of {#1} {\ref{#2}}:}}{\hfill$\square$}

\newcommand{\comment}[1]{}

\newcommand{\set}[1]{\left\{\, {#1} \,\right\}}
\newcommand{\seq}[1]{{\left( {#1} \right)}}
\newcommand{\W}{{\Omega}}

\newcommand{\g}{\mathsf{G}}
\newcommand{\ali}{\textup{\textmd{\textsc{Alice}}}}
\newcommand{\bob}{\textup{\textmd{\textsc{Bob}}}}
\newcommand{\A}{\textup{\textmd{\textsc{A}}}}
\newcommand{\B}{\textup{\textmd{\textsc{B}}}}
\newcommand{\Run}{{\mathrm{Run}}}
\newcommand{\M}{\mathrm{M}}
\newcommand{\K}{\mathsf{K}}
\newcommand{\D}{\mathrm{D}}
\newcommand{\Cp}{\mathrm{C}_{\mathrm {p}}}

\DeclareMathOperator{\dom}{dom}

\DeclareMathOperator{\id}{id}
\DeclareMathOperator{\restrict}{\upharpoonright}
\DeclareMathOperator{\level}{l}
\DeclareMathOperator{\embed}{\hookrightarrow}

\newcommand{\mG}{{\mathcal{G}}}

\newcommand{\mU}{{\mathcal{U}}}

\newcommand{\mW}{{\mathcal{W}}}
\newcommand{\NN}{\mathbb{N}}
\newcommand{\RR}{\mathbb{R}}

\newcommand{\Obj}[1]{\mathsf{Obj}\left({#1}\right)}

\newcommand{\Sets}{\mathbf{Set}}
\newcommand{\CUMet}{\mathbf{CUltMet}_1}

\newcommand{\MetGame}{\mathbf{MetGame}}

\newcommand{\CTop}{\mathbf{Top}}
\newcommand{\Tree}{\mathbf{Tree}}
\newcommand{\PrTree}{\mathbf{PrTree}}
\newcommand{\TrGame}{\mathbf{ArbGame}}
\newcommand{\CTopp}{\mathbf{Top_\ast}}
\newcommand{\BM}{\mathsf{B}\mathsf{M}}
\newcommand{\Games}{\mathbf{Game}}
\newcommand{\Gmes}{\mathbf{Gme}}
\newcommand{\FunGame}{\mathbf{FunGame}}
\newcommand{\SetNop}{\Sets^{\omega^{\op}}}
\newcommand{\Sub}[1]{\mathsf{Sub}\left({#1}\right)}
\newcommand{\Multi}{\mathsf{Multi}}
\newcommand{\op}{\mathrm{op}}

\newcommand{\Tight}{\mathrm{Tight}}
\newcommand{\Cover}{\mathrm{Cover}}
\newcommand{\Br}{\mathrm{Branch}}

\tikzset{
	invisible/.style={opacity=0},
	visible on/.style={alt=#1{}{invisible}},
	alt/.code args={<#1>#2#3}{
		\alt<#1>{\pgfkeysalso{#2}}{\pgfkeysalso{#3}} 
	},
}

\tikzset{
	block/.style = {
		rectangle,
		thick,
		text width=6em,
		align=center,
		rounded corners,
		draw=cyan!40!black,
		fill=cyan!20,
		inner ysep=10pt
	}
}

\begin{document}
	\title{Infinitely ludic categories}
	\author[M. Duzi]{Matheus Duzi$^1$}
	\address{Instituto de Ci\^encias Matem\'aticas e de Computa\c c\~ao, Universidade de S\~ao Paulo\\
		Avenida Trabalhador s\~ao-carlense, 400,  S\~ao Carlos, SP, 13566-590, Brazil}
	\email{matheus.duzi.costa@usp.br}
	\thanks{$^1$Supported by FAPESP, grants no 2019/16357-1 and 2021/13427-9}
	\author[P. Szeptycki]{Paul Szeptycki$^2$}
	\address{Department of Mathematics and Statistics, York University\\
		Toronto, Ontario, Canada M3J 1P3}
	\email{szeptyck@yorku.ca}
	\thanks{$^2$Supported by NSERC Discovery Program, grant no 504276}
	\author[W. Tholen]{Walter Tholen$^3$}
	\address{Department of Mathematics and Statistics, York University\\
		Toronto, Ontario, Canada M3J 1P3}
	\email{tholen@yorku.ca}
	\thanks{$^3$Supported by NSERC Discovery Program, grant no 501260}
	
	\begin{abstract}
		Pursuing a new approach to the study of infinite games in combinatorics, we introduce the categories $\mathbf{Game}_{\textsc{A}}$ and $\mathbf{Game}_{\textsc{B}}$ and improve some classical results concerning topological games related to the duality between covering properties of $X$ and convergence properties of $\Cp(X)$ by establishing the existence and key role of certain natural transformations. We describe these ludic categories in various equivalent forms, viewing their objects as certain structured trees, presheaves, or metric spaces, and we thereby obtain their arboreal, functorial and metrical appearances. We use their metrical disguise to demonstrate a universality property of the Banach-Mazur game. The various equivalent descriptions come with underlying functors to more familiar categories which help establishing some important properties of the game categories: they are complete, cocomplete, extensive, cartesian closed, and coregular, but neither regular nor locally cartesian closed. We prove that their classes of strong epimorphisms, of regular epimorphisms, and of descent morphisms, are all distinct, and we show that these categories have weak classifiers for strong partial maps. Some of the categorical constructions have interesting game-theoretic interpretations.
	\end{abstract}
	\maketitle		
	\pagenumbering{arabic}
	
{\sc Keywords:} infinite game, game category, game tree, Banach-Mazur game, tightness game, covering game, selection principles, arboreal game, metric game, functorial game, multiboard game, exponential game, chronological map, game morphism, locally surjective map, (strict) quotient map, weak classifier.

{\sc Subject Classification:} 91A44, 18B99; 18A20, 18B50, 18D15, 54C35, 54D20.

\section{Introduction}
An abstract infinite game with $\omega$-many innings is a turn-based game with no draws in which two players, $\ali$ and $\bob$, compete and all moves by the players are fully known to both parties (that is, a game of \emph{perfect information}, in which no choice is hidden). One example of such is the Banach-Mazur game, considered to be the first infinite game (for a thorough description of the history of the study of infinite games, we refer to \cite{Telgarsky1987}):

\begin{ex}[Banach-Mazur game]\label{EX_BMGame}
	Given a non-empty topological space $X$, consider the following game, denoted by $\BM X$:
	\begin{itemize}
		\item In the first inning, $\ali$ chooses a non-empty open set $U_0$ and $\bob$ responds with a non-empty open set $V_0\subseteq U_0$.
		
		\item In the n-th subsequent inning, $\ali$ chooses a non-empty open set $U_n \subseteq V_{n-1}$ and $\bob$ responds with a non-empty open set $V_n\subseteq U_n$.
	\end{itemize} 
	We say that $\ali$ wins the run $\seq{U_0, V_0, \dotsc U_n, V_n, \dotsc}$ if $\bigcap_{n<\omega} V_n= \emptyset$, and that $\bob$ wins otherwise.	
\end{ex}

One often finds in many areas of mathematics the problem of whether a sequence $\seq{x_n:n<\omega}$ can be constructed with some desired properties. Usually infinite turn-based games played between two players are then defined as a way to strengthen the potential of such constructions. Indeed, through the existence of winning strategies we can then pose the question of whether it is possible to construct a sequence $\seq{x_n:n<\omega}$ with the desired properties \textit{even with an opposing force (the other player) attempting to hinder this process}.

The Banach-Mazur game, for instance, was introduced in the study of Baire subspaces of the real line and, indeed, in \cite{Oxtoby1957}, Oxtoby showed that a space is Baire if, and only if, $\ali$ has no winning strategy in $\BM X$. 

In addition to topological applications, games have been introduced and studied to classify Banach spaces in functional analysis (as seen in, e.g., \cite{Ferenczi2009}), to explore properties of algebraic structures (as seen in \cite{Brandenburg2018}), and notably in mathematical logic with the Axiom of Determinacy.

A categorical framework for a wide range of mathematical games has already been proposed in the literature (see \cite{Streufert2018, Streufert2020, Streufert2021}), but here we focus on two-player and perfect-information games of length $\omega$. The categories we introduce are significantly different from the one presented in \cite{Streufert2021}, even when restricted to those games of present interest (as we clarify further in Section \ref{SEC_Games}). Such differences play an important role in the proofs of the Theorems \ref{THM_ClassOmega}, \ref{THM_ClassGamma} and \ref{THM_BMUni_Fun}, as well as in the other categorical characterizations shown in Sections \ref{SEC_SubTree} and \ref{SEC_MetGames}. 

Within this categorical framework we are able to shed new light onto the functorial depth of some results about topological games. Namely, we show that two natural transformations between game functors entail the following classical result of Scheepers \cite{Scheepers1997, Scheepers2014} arising from the duality between covering properties of a space $X$ and convergence properties of $\mathrm{C_p}(X)$, the space of real valued functions over $X$ with the topology of pointwise convergence. We use the standard notation for the games involved; for their precise definition, see Examples \ref{EX_Tightness} and \ref{EX_Covering}.

\begin{thm}[\cite{Scheepers1997}, Theorem 13, and \cite{Scheepers2014}, Theorem 29]\label{THM_ClassOmega}
	Let $X$ be a $T_{3\frac{1}{2}}$-space\footnote{$T_{3\frac{1}{2}}$ is not assumed to entail $T_1$; see the end of the Introduction.}. Then:
	\begin{itemize}
		\item $\ali$ has a winning strategy in $\g_1(\W,\W)$ over $X$ if, and only if, $\ali$ has a winning strategy in $\g_1(\W_{\bar{0}},\W_{\bar{0}})$ over $\Cp(X)$.
		\item $\bob$ has a winning strategy in $\g_1(\W,\W)$ over $X$ if, and only if, $\bob$ has a winning strategy in $ \g_1(\W_{\bar{0}},\W_{\bar{0}})$ over $\Cp(X)$. 
	\end{itemize}
\end{thm}

We are also able to provide the same categorical perspective for the following result, which can be obtained by adapting the original proof of Theorem \ref{THM_ClassOmega}. 
\begin{thm}\label{THM_ClassGamma}	Let $X$ be a $T_{3\frac{1}{2}}$-space. Then:
	\begin{itemize}
		\item $\ali$ has a winning strategy in $\g_1(\W,\Gamma)$ over $X$ if, and only if,  $\ali$ has a winning strategy in $\g_1(\W_{\bar{0}},\Gamma_{\bar{0}})$ over $\Cp(X)$.
		\item $\bob$ has a winning strategy in $\g_1(\W,\Gamma)$ over $X$ if, and only if, $\bob$ has a winning strategy in $ \g_1(\W_{\bar{0}},\Gamma_{\bar{0}})$ over $\Cp(X)$.
	\end{itemize}
\end{thm}

In Section \ref{SEC_Trees} we first define game trees, which are the underlying objects responsible for dictating which moves the players can make in the infinite games of interest in this paper. We then establish morphisms between game trees, thus determining a category which we denote by $\Gmes$. We show that $\Gmes$ can be coreflectively embedded into both the category $\Tree$ of (graph-theoretic) trees and the presheaf category $\SetNop$ (with $\omega$ denoting the first infinite ordinal, also considered as a small category). Hence, inside this Grothendieck topos, which plays a prominent role in categorical logic (see in particular \cite{Birkedal2012}), there is a coreflective copy of the category $\Gmes$ of ``games without payoff sets''.
The topos $\SetNop$ has been presented in the computer science literature as a model of guarded recursion, and its objects were related via adjunction to {\em certain} ultrametric spaces; see in particular \cite{Birkedal2012}. It follows that the full subcategory of these ultrametric spaces and their non-expanding maps must be equivalent to our category $\Gmes$. 

By adding a payoff set to a tree, we define infinite games in Section \ref{SEC_Games}, which then serve as the objects of our principal game categories. We give some notable examples, and also define the concept of a strategy.
Two different candidates for morphisms between games are defined, leading us to the categories $\Games_\A$ and $\Games_\B$, with the same objects. Despite being isomorphic, distinguishing between the two categories turns out to be convenient. In addition to discussing some elementary properties of morphisms of games, we introduce the important novel notion of \textit{local surjectivity}, which we then characterize in strictly categorical terms.

We show that introducing the payoff sets to the objects in $\Gmes$ leads us to a network of categories, each of which
is topological \cite{Adamek1990} over its poorer cousin as considered in Section  \ref{SEC_Trees}. As a result, adjunctions existing at the ground level (where payoff sets are disregarded) can be ``lifted'' to the level of game-theoretic interest. Consequently, we present our two game categories equivalently in their arboreal disguise, $\mathbf{ArbGame}$, and in their functorial disguise, $\mathbf{FunGame}$, the latter category being a full coreflective subcategory of a topological category over a presheaf topos. 

The well-known categorical techniques used in Subsection \ref{SEC_SubTree} have multiple roots and modern ramifications. In our context, it suffices to exploit the fact that every functor $K:\mathbf C\to\Sets$ gives rise to a new category of $\mathbf C$-objects $X$ equipped with a subset of $KX$ that is topological over $\mathbf C$, and that there is an easy way of lifting adjunctions along topological functors, as first explored by Wyler \cite{Wyler1971a, Wyler1971b}. As the new category is simply a full subcategory of the comma category $(\Sets\downarrow K)$, this process is nowadays often referred to as the {\em sconing} (formery {\em Artin glueing}) or {\em subsconing} of the functor $K$ (whose codomain may be taken to be a general category, rather than $\Sets$). Sconing has become a basic, but important tool in the computer science literature on the categorical semantics for logical relations; see, for example, \cite{Goubault2008, Lucatelli2023}.

Applying the same general categorical techniques as in Subsection  \ref{SEC_SubTree}, we show that there is a full coreflective subcategory $\mathbf{MetGame}$ of the category $\CUMet$ of complete ultrametric spaces of diameter at most $1$ that is equivalent to $\Games_\A$. Hence, in addition to their arboreal and functorial descriptions, we also have an equivalent {\em metric} presentation of games which turns out to be particularly useful in the remainder of the paper.

We dedicate Section \ref{SEC_TopGames} to establishing rather naturally the functors $\Cover^\W_1$, $\Tight^\W_1$, $\Cover^\Gamma_1$ and $\Tight^\Gamma_1$. They take as arguments topological spaces and send them respectively to the games  $\g_1(\W,\W)$, $\g_1(\W_x,\W_x)$, $\g_1(\W,\Gamma)$ and $\g_1(\W_x,\Gamma_x)$. Our point then is to establish a couple of natural transformations linking these functors, which then enable us to present a largely categorified proof of Theorem \ref{THM_ClassOmega}. Likewise for Theorem \ref{THM_ClassGamma}. We believe that our use of these categorical elements adds clarity to the understanding of the theorems in question.

With the help of the metric presentation of games, in Section \ref{SEC_BMUniversal}, we show that every game may be embedded into a Banach-Mazur game. While this result may be considered as an analogue of Cayley's Theorem in group theory, thus putting Banach-Mazur games at par with symmetric groups, we show that our embedding result actually goes beyond this predecessor: under a restriction to so-called locally surjective game morphisms, the emdedding defines a natural transformation. 

Aided by the various equivalent presentations of games, in Section \ref{SEC_CatProp} we finally explore some relevant categorical properties of $\Games_\A\cong\Games_\B$. At times, this becomes challenging from both, the game-theoretic and the categorical perspective. Of course, being equivalent to a full coreflective subcategory of a topological category over a presheaf category, our category of games is complete and cocomplete. Our point, however, is to give concrete descriptions of limits and colimits from a game-theoretic perspective, and we often proceed likewise when we prove that the category $\Games_\A$ is cartesian closed, infinitely extensive \cite{Carboni1993, Centazzo2004}, dual to a regular category \cite{Bourn2004}, but not regular itself, and that it has further interesting orthogonal factorization systems. But the category fails to be locally cartesian closed or to possess a classifier of strong subobjects and, hence, it fails to be a quasi-topos, although it does possess \emph{weak} classifiers for strong partial maps \cite{Adamek1990}. 

\medskip
Some notational conventions follow. Given a finite sequence $t = \seq{x_0, \dotsc, x_{n-1}}$ of $n<\omega=\{0,1,2,,\dotsc\}$ elements, we denote by $|t|=n$ the \textit{length} of $t$ and, for each $k\le |t|$, by 
\[t\upharpoonright k=\seq{x_0,\cdots,x_{k-1}}
\]
its initial segment of length $k$, also called the \textit{truncation} of $t$ of its first $k$ elements. For $n=0$, the sequence $t=\seq{\,}$ is empty. We use the truncation notation also for infinite sequences $R=\seq{x_i:i<\omega}$, so that $R\upharpoonright k$ denotes the finite sequence of the first $k$ elements of $R$.

Given two finite sequences $t=\seq{x_i: i< n}$ and $s=\seq{y_i:i< m}$, 
\[
t^\smallfrown s = \seq{x_0, \dotsc, x_{n-1}, y_0, \dotsc, y_{m-1}}.
\]
is the \textit{concatenation} of $t$ with $s$.	
Furthermore, for a single element $x$ we have the constant sequence $\seq{x: i< n}$ of $n\ge 1$ copies of $x$, and in case $n=1$ we write $t^\smallfrown x = t^\smallfrown \seq{x}$ for the concatenation of $t$ with $\seq{x}$.

Unless stated otherwise, by a \textit{space} we always mean a topological space. A space $X$ is $T_{3\frac{1}{2}}$ if, for all closed sets $Y\subseteq X$ and $x\in X\setminus Y$, there is a continuous function $f\colon X\to [0,1]$ with $f(x)=0$ and $f[Y]\subseteq\{1\}$.

\medskip
{\sc \textbf{\textit{Acknowledgement.}}} We are grateful for helpful comments and pointers received after a presentation by the first-named author at the conference Category Theory 2023, held in July 2023 in Louvain-la-Neuve, especially by Marcelo Fiore and Rui Prezado.

\section{Categories of trees}\label{SEC_Trees} 
We establish our notion of infinite game by first formalizing its underlying arboreal component.

\begin{defn}\label{DEF_GameTree}
	For a set $M$, a set $T\subseteq M^{<\omega}=\bigcup_{n<\omega}{M^n}$ of finite sequences in $M$ is a \textit{game tree over $M$} if 
	\begin{itemize}
		\item[(\rom{1})] For every $t\in T$ and $k\le|t|$, one has $t\restrict k\in T$;
		\item[(\rom{2})] For every $t\in T$, there is an element $x$ such that $t^\smallfrown x\in T$.
	\end{itemize}	
	We say that $T$ is a \textit{game tree} if $T$ is a game tree over some set $M$.
	
	For a game tree $T$, we use the following suggestive language and notations:
	
	\begin{itemize}
		\item An infinite sequence 
		\[
		R\in \Run(T) = \set{R\in M^{\omega}: R\restrict n\in T \text{ for every $n<\omega$}}.
		\]
		is a \textit{run} of $T$.
		
		\item A finite sequence $t\in T$ is a \textit{moment} of $T$. 
		
		\item The minimal set $M$ with $T\subseteq \bigcup_{n<\omega}{M^n}$, denoted by $\M(T)$, is the set of \textit{moves} of $T$.
		
		\item If $|t|$  is even, then $t\in T$ is $\ali$'s \textit{turn}, and $\set{x\in \M(T): t^\smallfrown x\in T}$ is the set of all possible moves $\ali$ can make at the moment $t$.
		
		\item If $|t|$ is odd, then $t\in T$ is $\bob$'s \textit{turn}, and $\set{x\in \M(T): t^\smallfrown x\in T}$ is the set of all possible moves $\bob$ can make at the moment $t$.
		
		\item If $t\in T$  and $|t|=2n$ or $|t|=2n+1$, then $t$ is in the $n$th \textit{inning} of $T$.
		
		\item For every $n<\omega$, we write 
		\begin{gather*}
			T\restrict n = \set{t\in T: |t|\le n},\\
			{T(n)} = \set{t\in T: |t| = n}.
		\end{gather*}
	\end{itemize}	
\end{defn}

As the vocabulary suggests, game trees play the role of containing every possible moment of a given game and, thus, they determine every possible choice that a player can make at a given moment of said game.

With respect to some examples we note that condition (\rom{2}) may seem to be rather restrictive at first glance, as there are interesting games in which some runs (or even all of them) are finite. But in this case we can artificially extend every finite run to a single infinite branch preserving the winning criteria. In this way the ``essence'' of the game with finite runs is preserved within the setting of our infinite games. 

We now define the morphisms of our category of game trees.
\begin{defn}[Chronological map]\label{DEF_fChronological}
	For game trees $T_1$ and $T_2$, a mapping $f\colon T_1\to T_2$ is \textit{chronological} (in which case we write $T_1\stackrel{f}{\to}T_2$) if $f$ preserves length and truncation of moments; that is: if for every $t\in T_1$, $|f(t)|=|t|$ and $f(t\restrict k)=f(t)\restrict k$ for all $k\le |t|$.
\end{defn}

One trivially has:
\begin{prop}\label{PROP_morphcat_gme}
	Given game trees $T_1$, $T_2$ and $T_3$, if $T_1\stackrel{f}{\to}T_2$ and $T_2\stackrel{g}{\to}T_3$, then $T_1\stackrel{g\circ f}{\to}T_3$.
	
	Moreover, the identity map over a game tree $T$ is chronological.
\end{prop}

Thus, we denote by $\Gmes$ the category whose objects are game trees and morphisms are chronological maps (with usual composition of maps).

In this paper a \textit{graph} is a pair $(V,E)$, with $V$ the set of its vertices and $E\subseteq[V]^2$ the set of its edges, where $[V]^2$ is the set of two-element subsets of $V$.
A \textit{path} in the graph $(V,E)$ is a finite sequence $\seq{x_0, \dotsc, x_n}$ of distinct vertices such that $\{x_i,x_{i+1}\}\in E$ for every $i\le n$. We call a graph $(V,E)$ a \textit{tree} if, for all distinct vertices $x,y\in V$, there is a unique path $\seq{x_0, \dotsc, x_n}$ with $x_0 = x$ and $x_n = y$.

A \textit{rooted tree} is a tree with a distinguished vertex $r$, that is, a triple $(V,E,r)$, with $(V,E)$ a tree and $r\in V$. Then, given $x\in V\setminus \{r\}$, we know that there is a unique path from $r$ to $x$, so we can naturally make the edges of the tree \textit{directed} by stating that an edge $\{x,y\}\in E$ goes from $x$ to $y$ if the path from $r$ to $y$ contains $x$ (the direction \textit{upwards from the root}); such edges may now be written ordered pairs $(x,y)$. 

If $T\neq \emptyset$ is a game tree, then the empty sequence $\seq{\,}$ lies in $T$. Hence, such a game tree $T$ has a natural structure of a rooted directed tree, with root $\seq{\,}$ and edges of the kind $(t,t^\smallfrown x)$. We will thus use the standard graph-theoretic vocabulary (as it can be seen in, e.g., \cite{Pitz2023}): for instance, condition (\rom{2}) in Definition \ref{DEF_GameTree} can be restated as $T$ being \textit{pruned}, and an element $R\in \Run(T)$ may also be called a \textit{branch} of $T$. We will further explore this relation with graph theory in Sections \ref{SEC_TreeTopos} and \ref{SEC_SubTree}.

Let $\Tree$ denote the category of directed rooted trees (with the direction defined in terms of the root) whose morphisms are the homomorphisms of directed rooted trees; that is: $(V,E,r)\stackrel{f}{\to}(V',E',r')$ is a mapping $f\colon V\to V'$ with $(f\times f)(E)\subseteq E'$ and $f(r)=r'$. 
We call
\[
\Br(T) = \set{R\in V^\omega: \seq{r}^\smallfrown R\restrict n  \text{ is a path for every } n<\omega},
\]
the set of \textit{branches} of the directed rooted tree $T=(V,E,r)$.

The chronological mappings of game trees are then precisely the homomorphisms of the rooted trees arising from the games. Hence, we have a full functorial embedding 
\[
\mathrm{I}\colon\Gmes\hookrightarrow\Tree. 
\]
Let us determine the essential image  $\PrTree$ of $\mathrm{I}$, \textit{i.e.}, describe those trees which, up to isomorphism, arise as game trees. These are easily seen to be  precisely the direted rooted trees $T=(V,E,r)$ that are pruned and, hence, distinguished by the property that, for every $x\in V$, there are $R\in \Br(T)$ and $n<\omega$ with $R(n)=x$. Indeed, for a given directed rooted tree $(V,E,r)$ one just considers the 1-1 association 
\[x\in V \,\longleftrightarrow\, t_x\in \bigcup_{n<\omega}V^n, \]
where $t_x$ is the unique path from $r$ to $x$.

\begin{thm}\label{THM_GmeTree}
	The category $\Gmes$ is equivalent to the full coreflective subcategory $\PrTree$ of $\Tree$ given by the pruned rooted trees.
\end{thm}
\begin{proof}
	It only remains to be shown that the inclusion of $\PrTree$ into $\Tree$ has a right adjoint. To this end, 
	given a directed rooted tree $T=(V,E,r)$, we construct a tree $\mathrm{Pr}T$ by \textit{pruning} its finite branches, so that $\mathrm{Pr}T$ becomes a pruned subtree of $T$; that is,		
	we define $\mathrm{Pr}T=(V^*,E^*,r)$, where		
	\begin{gather*}
		V^* = \{r\}\cup\set{x\in V: \exists R\in \Br(T) \exists n<\omega (R(n) = x)},\\
		E^* = \set{(x,y)\in E: x,y\in V^*}.
	\end{gather*}
	Now, letting \begin{tikzcd} \mathrm{Pr}T \arrow[r, "\varepsilon_{T}"{auto}] &  T \end{tikzcd} be  the inclusion map of $\mathrm{Pr}T$ into $T$, we just need to show that every homomorphism $T'\stackrel{f}{\to} T$ of directed rooted trees, where $ T'=(V',E',r')$ is pruned, factorizes uniquely through	 $\varepsilon_T$, \textit{i.e.,} the  image of $f$ must lie in $\mathrm{Pr}T$. 
	Indeed, necessarily $f(x)\in V^*$ whenever $x\in V'$, since by hypothesis on $T'$ there is $\seq{x_n: n<\omega}\in \Br(T')$ with $x = x_n$ for some $n<\omega$, so that $\seq{f(x_n):n<\omega}\in \Br(T)$ and $f(x) = f(x_n)$. In other words: the codomain restriction $g$ of $f$ to $\mathrm{Pr}T$
	is the unique morphism making the following diagram commute:
	\begin{center}
		$\xymatrix{\mathrm{Pr}T\ar[rr]^{\varepsilon_T}&& T\\
			T'\ar[u]^g\ar[rru]_f &&\\
		}$
	\end{center}
\end{proof}

\subsection{The category \texorpdfstring{$\Gmes$}{TEXT} as a subcategory of a presheaf topos}\label{SEC_TreeTopos}
In order to confirm the claim suggested by this subsection's header, we consider the well-studied (see \cite{Birkedal2012}) presheaf category $\SetNop$ whose
\begin{itemize}
	\item objects are inverse systems of sets $T(n)$ with connecting maps $\Gamma^n_m: T(n)\to T(m)$ for all $n\geq m$ in $\omega$, {\textit  i.e.}, functors $T\colon \omega^{\op} \to \Sets$, 
	and whose
	\item morphisms 
	are natural transformations $T\stackrel{f}{\to}T'$, \textit{i.e.}, families of maps $f_n\colon T(n)\to T'(n)$ commuting with the connecting maps of $T$ and $T'$. 
\end{itemize}
Of course, like every category of $\Sets$-valued presheaves, the category $\SetNop$ is a standard example of a  (Grothendieck) topos (see\cite{MacLane1992}). Thus, it now suffices for us to sketch the proof of the following well-known fact:

\begin{prop}\label{THM_Tree_SetNop}
	The categories $\Tree$ and $\SetNop$ are equivalent.
\end{prop}
\begin{proof}
	One defines a functor $\mathrm{Fun}\colon \Tree \to \SetNop$, as follows. Given a directed rooted tree $T=(V,E,r)$, for every $x\in V\setminus\{r\}$ we denote by $\level(x) = \mathrm{length}(p)-2<\omega$ the \textit{level of $x$ in $T$}, where $p$ is the unique path from $r$ to $x$. 
	Now one defines $\mathrm{Fun}T\in \SetNop$  by assigning to $n<\omega$ the set
	\[		
	V(n) = \set{x\in V: \level(x)=n}
	\]
	and  letting $(\mathrm{Fun}T)(n\ge m)\colon V(n) \to V(m)$ send $x\in V(n)$ to the unique $y\in V(m)$ in the path from $r$ to $x$. 
	Furthermore, given a homomorphism of directed rooted trees $(V,E,r)\stackrel{f}{\to}(V',E',r')$, since $\level(f(x))=\level(x)$, we may let $\mathrm{Fun}f = (f_n)_{n<\omega}$ be defined by $f_n=f|_{V(n)}$ for all $n<\omega$.
	
	A pseudo-inverse of $\mathrm{Fun}$ is given by the functor
	$\mathrm{Tr}\colon \SetNop \to \Tree$, which assigns
	\begin{itemize}
		\item to an inverse system $T\colon \omega^{\op} \to \Sets$ the tree $(V_T,E_T,r_T)$, where 
		\begin{gather*}
			V_T=\{r_T\} \sqcup\bigsqcup_{n<\omega}T(n) \text{ with any } r_T \text{ not in } \bigcup_{n<\omega}T(n),\\  
			E_T=\set{(x,y)\in V_T^2: \exists n<\omega (y\in T(n+1),  x\in T(n) \text{ and } \Gamma^{n+1}_n(y)=x)},
		\end{gather*}
		\item to a natural transformation $T\stackrel{f}{\to} T'$ the mapping $V_T\stackrel{\mathrm{Tr}f}{\to} V_{T'}$ with
		\[
		(\mathrm{Tr}f)(x) = \begin{cases}
			r_{T'}, \text{ if $x = r_T$,}\\
			f(x), \text{ otherwise.}
		\end{cases}
		\]
	\end{itemize}
	We omit the straightforward verification that the composite functors $\mathrm{Tr}\circ\mathrm{Fun}$ and $\mathrm{Fun}\circ\mathrm{Tr}$ are isomorphic to identity functors.
\end{proof}

Further to the proof of Proposition  \ref{THM_Tree_SetNop}, we note that it may seem strange to have the functor $\mathrm{Fun}$ completely forget the root of a tree. 
But intuitively, the root, as a single vertex, is \textit{redundant}, as a unique vertex in a rooted tree, preserved by every homomorphism. Its only role is to signal the \textit{upward direction} of the tree (parallel to the notion of \textit{chronology of moments} in games), which is preserved by homomorphisms. For  $T\in\SetNop$, we do not need such distinguished element (that is, we don't need to have $|T(0)|=1$) to play that role, since the order of the natural numbers already accomplishes the chronological effect.

\medskip	

We also note that the essential image of $\PrTree$ under the functor $\mathrm{Fun}\colon \Tree \to \SetNop$ has an easy description. Its objects are precisely the systems $T\colon \omega^{\op}\to \Sets$ for which all connecting maps are surjective, \textit{i.e.}, systems $T$ which take values in the category $\mathsf{Epi}(\Sets)$ of sets and surjective maps; we have the resulting full subcategory  $\mathsf{Epi}(\Sets)^{\omega^{\op}}$ of $\SetNop$.
Indeed, the connecting maps correspond precisely to the truncation maps $\mathrm{Tr}T(n)\to \mathrm{Tr}T(m),\, t\mapsto t\upharpoonright m,$ for all $n\geq m$, whose surjectivity describes $\mathrm{Tr}T$ as pruned. 

In this functorial environment then, ``pruning'' means ``enforcing surjectivity'' of all connecting maps. Consequently, the full inclusion functor $\mathsf{Epi}(\Sets)^{\omega^{\op}}	\hookrightarrow{\SetNop}$ has an easily described right adjoint:  it sends $T\in \Obj{{\SetNop}}$ to the system $T^*$ with
\[
T^\ast(n) =\pi_n[\mathrm{Lim}T],
\]
where $\mathrm{Lim}T$ is $T$'s (projective) limit in $\Sets$,
\[
\mathrm{Lim}T =\set{(t_n:n<\omega):t_n\in T(n),\, \Gamma^n_m(t_n)=t_m \;(n\geq m<\omega)}\subseteq \prod_{n<\omega}T(n),
\] 
and $\pi_n\colon \mathrm{Lim}T\to T(n)$ is its $n$th projection.

Let us summarize these observations:

\begin{cor}\label{COR_Tree_SetNop}
	The category $\Gmes$ is equivalent to the full coreflective subcategory  $\mathsf{Epi}(\Sets)^{\omega^{\op}}$ of the topos		$\SetNop$. The diagram 
	\begin{center}
		$\xymatrix{\Gmes\ar[r]^{\simeq\;}\ar[rd]_{\mathrm{I}} & \PrTree\ar@/_0.5pc/[d]^{\dashv}_{\mathrm{inc}}\ar@/^0.4pc/[rr]^{\mathrm{Fun}} && \mathsf{Epi}(\Sets)^{\omega^{\op}}\ar@/_0.5pc/[d]^{\dashv}_{\mathrm{inc}}\ar@/^0.4pc/[ll]_{\simeq}^{\mathrm{Tr}} & T^*\\
			& \Tree\ar@/_0.5pc/[u]_{\mathrm{Pr}}\ar@/^0.4pc/[rr]^{\mathrm{Fun}} && \SetNop\ar@/_0.5pc/[u]\ar@/^0.4pc/[ll]_{\simeq}^{\mathrm{Tr}} & T\ar@{|->}[u] \\	
		}$
	\end{center}	
	of categories and functors commutes in an obvious sense, up to natural isomorphisms.
\end{cor}

The bottom equivalence of categories in the above diagram may be seen as ``sitting over $\Sets$'' (but not faithfully so). Indeed, the functor $\Delta: \Sets \to \SetNop$ assigning to every set $X$ the system with constant value $X$ provides a full coreflective embedding, with its standard right adjoint $\mathrm{Lim}\colon\SetNop\to\Sets$ assigning to the system $T$ its (projective) limit. 
Considering $\mathrm{Lim}T$ as ``the underlying set'' of the system $T$ and therefore treating $\mathrm{Lim}$ as ``forgetful'' (even though it is not faithful), we will then treat $\Delta X$ as the ``free system'' over $X$; briefly,
\[
\Delta\dashv\mathrm{Lim}\colon\SetNop\to\Sets\,.
\]		
Of course, the functor $\Delta$ actually takes values in $\mathsf{Epi}(\Sets)^{\omega^\op}$ (since for any set $X$, all connecting maps of the system $\Delta X$ are identity maps), so that (without name changes) we also have the restricted adjunction
\[
\Delta\dashv\mathrm{Lim}\colon\mathsf{Epi}(\Sets)^{\omega^\op}\to\Sets\,.
\]
Let us also describe the functor $\mathrm{Lim}\colon\SetNop\to\Sets$ after being precomposed with the equivalence $\mathrm{Fun}$. It then becomes (isomorphic to) the functor
$\mathrm{Branch}\colon \Tree\to \Sets$
whose effect on a rooted directed tree we have already described at the beginning of this section. It sends a homomorphism $T\stackrel{f}{\to}T'$ of rooted directed trees to the mapping $\overline{f}$ that assigns to $\seq{x_n:n<\omega}\in \Br(T)$ the sequence $\seq{f(x_n):n<\omega}\in \Br(T')$.

\begin{cor}\label{COR_Branch}
	The functor $\Br$ has a left adjoint, given by the full embedding $\mathrm{Free}:=\Pr\circ\Delta\colon\Sets\to\Tree$. It assigns to a set $X$ the tree of all finite sequences of the form $(x,\dots, x)$ $(x\in X)$. Furthermore, in the (up to isomorphism) commutative diagram
	\begin{center}
		$\xymatrix{\Tree\ar@/^0.4pc/[rr]^{\mathrm{Fun}}_{\simeq}\ar@/_0.5pc/[ddr]_{\mathrm{Branch}}^{\bot} && \SetNop\ar@/^0.4pc/[ll]^{\mathrm{Pr}}\ar@/^0.5pc/[ldd]^{\mathrm{Lim}}_{\bot}\\
			&&\\
			&\Sets\ar@/_0.5pc/[luu]_{\mathrm{Free}}\ar@/^0.5pc/[uur]^{\Delta}&\\
		}$
	\end{center}
	the two top categories may respectively be replaced by $\PrTree$ and $\mathsf{Epi}(\Sets)^{\omega^\op}$ and then give  the ``pruned version'' of the diagram.	
\end{cor}

\subsection{Metric game trees}\label{SEC_MetGameTrees}
There is a well-known adjunction between presheaves on $\omega$ and so-called {\em bisected} ultrametric spaces, {\em i.e.}, ultrametric spaces in which all non-zero distances are of the form $2^{-n}$ for some $n<\omega$, and this adjunction restricts to an equivalence between the category $\mathsf{Epi}(\Sets)^{\omega^{\op}}$ and the category of bisected complete
ultrametric spaces and their non-expanding maps: see Proposition 5.1 of \cite{Birkedal2012}. Replacing $\mathsf{Epi}(\Sets)^{\omega^{\op}}$ by the equivalent category $\Gmes$, in this section we want to describe the essence of this adjunction and the ensuing equivalence of $\Gmes$ with a category of certain ultrametric spaces in a more direct fashion. 

We denote by $\CUMet$ the category of complete ultrametric spaces of diameter at most $1$, with their non-expanding, or $1$-Lipschitz, mappings $f\colon (X,d)\to (Y,d')$ as morphisms: $d'(fx,fx')\leq d(x,x')$ for all $x,x'\in X$. As a minor technical difference with existing work, instead of metrics whose non-zero values are of the form $2^{-n}$, we will consider complete ultrametric spaces whose non-zero distances are of the form $\frac{1}{n+1}\; (n<\omega)$ and denote the resulting full subcategory of $\CUMet$  by $\mathbf{SeqSpa}$. Shortly we will see that its object behave indeed like ``spaces of sequences''.

Our goal now is to show:

\begin{thm}\label{THM_GmeUMet}
	There is an adjunction
	$$\xymatrix{\Gmes\ar@/^1.0pc/[rr]^{\mathrm{Run}\quad} &\bot& \CUMet\ar@/^1.0pc/[ll]^{\mathrm{Ball}\quad}
	}$$
	with a full and faithful functor $\mathrm{Run}$, which restricts to an equivalence of $\Gmes$ with the coreflective subcategory $\mathbf{SeqSpa}$ of $\CUMet$.
\end{thm}

In order to do so, we first note that any chronological map $f\colon T_1\to T_2$ extends to a mapping $\overline{f}\colon \Run(T_1)\to \Run(T_2)$, uniquely determined by  $\overline{f}(R)\restrict n=f(R\upharpoonright n)$ for every $n<\omega$; we may write
\[
\overline{f}=\lim_{\longrightarrow}f_n,
\]
with $f_n: T_1(n)\to T_2(n)$ denoting the indicated restriction of $f$.
It is then easy to check that we have established a functor $\mathrm{Run}:\Gmes\to\Sets$ such that
\begin{itemize}
	\item over objects, $\Run T = \Run(T)$;
	\item over morphisms, for a chronological $f\colon T\to T'$, $\Run f = \overline{f}\colon \Run(T)\to \Run(T')$. 
\end{itemize}

We now need to ``lift'' the functor $\mathrm{Run}:\Gmes\to\Sets$, as indicated in Theorem \ref{THM_GmeUMet}. Let us first state a technical lemma, whose proof can be left to the reader. For $R,R'\in T$ with $T$ in $\Gmes$ we use the notation
\[
\Delta(R,R') = \min\set{n<\omega: R(n)\neq R'(n)}
\]
when $R\neq R'$, and put $\Delta(R,R')=\infty$ otherwise.

\begin{lemma}\label{PROP_ChronDeltaChar}
	For $T_1,T_2\in \Gmes$, let $g\colon \Run(T_1)\to \Run(T_2)$ be any mapping. Then there is a chronological mapping $f\colon T_1\to T_2$ such that $g=\overline{f}$ if, and only if, $\Delta(g(R),g(R'))\ge \Delta(R,R')$ for all $t,t'\in \Run(T_1).$
\end{lemma}

\begin{myproof}{Theorem}{THM_GmeUMet}
	
	For an object $T\in\Gmes$ we obtain an ultrametric $d_T$ on $\Run(T)$ of diameter at most $1$ by setting
	\[
	d_T(R,R')=\begin{cases}
		\frac{1}{\Delta(R,R')+1}\text{ if $R\neq R'$,}\\
		0 \text{\qquad\quad\,\, otherwise.}
	\end{cases}
	\]
	It clearly makes $(\Run(T),d_T)$ complete, and 
	for a chronological map $T\stackrel{f}{\to} T'$, the map $\overline{f}\colon \Run(T)\to \Run(T')$ becomes $1$-Lipschitz.
	With Lemma  \ref{PROP_ChronDeltaChar} one easily sees that the now well-established functor $\Run:\Gmes\to\CUMet$ is full and faithful.
	
	Conversely, given  $(X,d)\in \CUMet$ with induced topology $\tau$, for every $x\in X$ let us denote the closed ball of radius $r>0$ centered at $x$ by $\overline{B}(x, r)$ and define the game tree $\mathrm{Ball}(X,d)$ over $\tau$ as the set of all finite sequences of the form
	\[
	\seq{\overline{B}(x_i,\frac{1}{i+2}):i\le n}\subseteq\tau^n
	\]
	for some $n<\omega$ and $\seq{x_i:i\le n}\in X^n$,
	such that $\overline{B}(x_i,\frac{1}{i+2})\supseteq \overline{B}(x_j,\frac{1}{j+2})$ for all $i\le j\le n$. 
	Clearly,  $\mathrm{Ball}(X,d)$ 
	satisfies conditions \rom{1} and \rom{2} of Definition \ref{DEF_GameTree}. 
	For $(X,d)\stackrel{f}{\to}(X',d')$ $1$-Lipschitz, let $\mathrm{Ball}$ assign to $f$ the well-defined chronological map $\tilde{f}$ that sends a sequence 
	$\left(\seq{\overline{B}\left(x_i,\frac{1}{i+2}\right):i\le n }\right)$ to $\seq{\overline{B}\left(f(x_i),\frac{1}{i+2}\right):i\le n }$, by just applying $f$ to the center of a ball.
	
	To complete the proof of the adjunction $\Run\dashv \mathrm{Ball}$, it suffices to establish a $\Run$-couniversal arrow 
	$$\varepsilon_{(X,d)}:\Run(\mathrm{Ball}(X,d))\to (X,d)$$
	for all $(X,d)\in \CUMet$. First we note that for every run $R$ in $\Run(\mathrm{Ball}(X,d))$ there is a sequence $\seq{x_n:n<\omega}\in X^\omega$ such that 
	$R=\seq{\overline{B}\left(x_n,\frac{1}{n+2}\right):n<\omega}$.
	Since $(X,d)$ is complete, there is an $x\in X$ such that $\bigcap_{n<\omega}\overline{B}\left(x_n,\frac{1}{n+2}\right)=\{x\}$. Moreover, because $d$ is an  ultrametric, $x\in \overline{B}\left(x_n,\frac{1}{n+2}\right)$ implies $\overline{B}\left(x,\frac{1}{n+2}\right)=\overline{B}\left(x_n,\frac{1}{n+2}\right)$ for every $n<\omega$, and
	\[
	R=\seq{\overline{B}\left(x,\frac{1}{n+2}\right):n<\omega}.
	\]
	follows. Hence, we can let $\varepsilon_{(X,d)}$ assign to $R$ the unique point $x\in X$ such that the above identity holds.
	Let us see that $\varepsilon_{(X,d)}$ is $1$-Lipschitz. 
	Indeed, given distinct runs $R_1= \seq{\overline{B}\left(x_1,\frac{1}{n+2}\right):n<\omega}$ and $R_2=\seq{\overline{B}\left(x_2,\frac{1}{n+2}\right):n<\omega}$, let
	\[
	M=\min\set{n<\omega: R_1(n)\neq R_2(n)}.
	\]
	Then, if $d(x_1,x_2)\le \frac{1}{M+1}$, with $d'$ denoting the metric of $\Run(\mathrm{Ball}(X,d))$ one has
	\[
	\frac{1}{M+1} = d'(R_1,R_2)\ge d(x_1,x_2)=d(\varepsilon_{(X,d)}(R_1),\varepsilon_{(X,d)}(R_2)).
	\]
	
	To see the couniversal property, we observe that $\varepsilon_{(X,d)}$ is bijective. Hence, given any $1$-Lipschitz map $h:(\Run(T),d_T)\to (X,d)$ with a game tree $T$, an application of Lemma \ref{PROP_ChronDeltaChar} gives us a unique chronological map $f:T\to \mathrm{Ball}(X,d)$ with $\varepsilon_{(X,d)}\circ \bar{f}=h$, as required.
	\begin{center}
		$\xymatrix{\mathrm{Ball}(X,d) & \Run(\mathrm{Ball}(X,d))\ar[rr]^{\varepsilon_{(X,d)}} && (X,d)\\
			T\ar@{-->}[u]^f & \Run(T)\ar[urr]_h\ar[u]^{\bar{f}}&& \\
		}$
		
	\end{center}
	
	For the completion of the proof it suffices to note that the map $\varepsilon_{(X,d)}$ becomes an isomorphism in $\CUMet$, {\em i.e.}, a bijective isometry, precisely when $(X,d)$ is an object of $\mathbf{SeqSpa}$. This, however is obvious, since the inequality in $(*)$ becomes an equality 
	when $d$ has its range restricted to $\{0\}\cup\set{\frac{1}{n+1}:n<\omega}$.
\end{myproof}

We conclude this section with the diagram of Figure \ref{diagram_Gme_tikzcd}, which summarize the results we showed in this section. 
In these diagrams, double-sided arrows ``$\leftrightarrow$''indicate an equivalence between categories (which, accompanied by the symbol ``$\cong$'', further indicate that such functors are isomorphisms). We emphasize that the diagram commute only up to isomorphism. But  the downward composition of functors from $\Gmes$ to $\Sets$ in \ref{diagram_Gme_tikzcd} on the left side commutes strictly with the functor $\Run$.

\begin{figure}
	\begin{tikzcd}[row sep=5em, column sep=6em]
		&  {\Gmes}  \arrow[d,leftrightarrow]\arrow[dr,leftrightarrow]\arrow[dl,leftrightarrow]	& 	 \\ 
		\mathbf{SeqSpa}\arrow[r,leftrightarrow] \arrow[d,hook,"\dashv"{auto,xshift=0.5ex}]\arrow[d, leftarrow, bend left, shift left=2, "\mathrm{Run}\circ \mathrm{Ball}"] & \PrTree \arrow[r, leftrightarrow] \arrow[d,hook,"\dashv"{auto,xshift=0.5ex}]\arrow[d, leftarrow, bend left, shift left=2, "\Pr"]	&    \mathsf{Epi}(\Sets)^{\omega^{\op}}\arrow[d,hook,"\dashv"{auto,xshift=0.5ex}]\arrow[d, leftarrow, bend left, shift left=2, "(-)^*"]\\
		\CUMet  \arrow[dr, "\mathrm{Forget}"description] & \Tree \arrow[r,leftrightarrow] \arrow[d,  "\Br"]&    \SetNop  \arrow[dl, "\mathrm{Lim}"description]\\
		&  \Sets \arrow[ul, bend left, "\tang"{sloped, auto,yshift=0.75ex}, "\msmall{\mathrm{Discrete}}"{sloped, auto, swap}] \arrow[ur, bend right, "\tang"{sloped, auto,yshift=0.75ex}, "{\Delta}"{sloped, auto, swap}] \arrow[u,  bend left, shift left=2, "\msmall{\mathrm{Free}}"{auto}, , "\dashv"{auto,swap,yshift=0.0ex, xshift=0.5ex}]
	\end{tikzcd}
	\caption{Equivalent categories of trees and their underlying categories}\label{diagram_Gme_tikzcd}	
\end{figure}
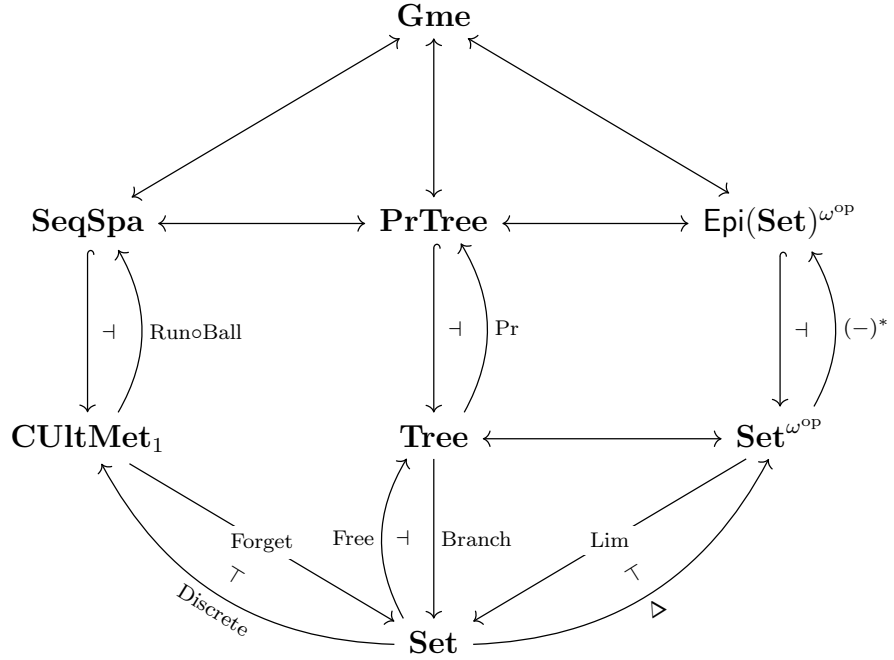

\section{Categories of games}\label{SEC_Games}
So far, we have only addressed the decision structure of games through the study of their game trees. We now cover the remaining defining feature of a game: its \emph{winning criteria}.

\begin{defn}\label{DEF_infinite_game}
	A pair $G = (T, A)$ is an \textit{infinite game} if $T$ is a game tree and $A\subseteq \Run(T)$.
	
	The set $A$ is called the \textit{payoff set} of $G$; a run $R$ is  \textit{won by} $\ali$ if $R\in A$; and $R$ is \textit{won by $\bob$} otherwise.
	We may use for infinite games the same vocabulary introduced for game trees in Definition \ref{DEF_GameTree} (e.g., we may also say that a $t\in T$ is a \emph{moment} of $G=(T,A)$).
\end{defn}

For specific games, such as the Banach-Mazur game of Example \ref{EX_BMGame}, the rules found there for the moves of the players should be used to recursively define the game tree of the game in terms of Definition \ref{DEF_infinite_game} and the the payoff set of the game is then determined by the winning criteria. Explicitly:

\begin{ex}\label{EX_BMInfGame}
	For a non-empty space $X = (X,\tau)$, let $\BM X = (T,A)$ be such that
	\begin{gather*}
		T = \set{t\in (\tau\setminus\{\emptyset\})^n:n<\omega, \forall i\le j<n\; (t(i)\supseteq t(j))},\\
		A = \set{R\in (\tau\setminus\{\emptyset\})^\omega: \bigcap_{n<\omega}R(2n+1)=\emptyset}.
	\end{gather*}
	In this case, $\M(\BM X) = \tau\setminus\{\emptyset\}$.
\end{ex}

We now present a few more examples of games that will come up in the subsequent sections. Despite their triviality, the first ones play essential roles in our categories of games, such as making sure that they are complete and cocomplete (see Section 9).

\begin{ex}[Empty game] \label{EX_trivial}
	By vacuity, the empty game $(\emptyset, \emptyset)$ satisfies the conditions of an infinite game, with $\M((\emptyset,\emptyset))=\emptyset$. In this case, one obtains the \emph{initial game} (see {\rm Section \ref{SUBSECT_coproducts}}).
\end{ex}

\begin{ex}[World's most boring games]\label{EX_boring}
	The games with a non-empty game tree over a one-element set $\{*\}$ are necessarily of the form $(\{*\}^{<\omega},A)$; for $A=\emptyset$, we call this game {\em generating} (a terminology to be explained in {\em Remark \ref{REM_generator}(1)}), and  for $A=\{*\}^{\omega}$ one obtains the {\em terminal game} (see {\rm Section \ref{SUBSECT_products}}). 
\end{ex}

\begin{ex}\label{EX_cogen}
	Also easily described is the game $(T,\Run(T))$ with decision tree
	\[
	T = \set{\seq{1:i<n}^\smallfrown \seq{0:j<k}:n,k\in\NN}
	\]
	whose finite sequences consist of a string of $1$s followed by a string of $0$s. 
	We call this game {\em cogenerating}, a terminology to be explained in {\em Remark \ref{REM_generator}(4)}.
\end{ex}

\begin{ex}[Tightness games]\label{EX_Tightness}
	For a space $X$ and a distinguished point $x\in X$, we let 
	\begin{gather*}
		\W_x=\set{A\subseteq X: x\in\overline{A}},\\
		\Gamma_x = \set{(x_n)_{n<\omega}\in X^\omega: x_n\stackrel{n\to\infty}{\longrightarrow}x}.
	\end{gather*}
	and consider the following games.
	\begin{description}
		\item[$\g_1(\W_x,\W_x)$] In each inning $n<\omega$,
		\begin{itemize}
			\item $\ali$ chooses a subset $A_n\subseteq X$ with $x\in \overline{A_n}$;
			\item $\bob$ responds by choosing $a_n\in A_n$.
		\end{itemize} 
		$\bob$ wins the run $\seq{A_0,a_0,\dotsc, A_n,a_n,\dotsc}$ if $x\in \overline{\set{a_n:n\ge k}}$ for all $k<\omega$, and $\ali$ wins otherwise. In other words, if $\g_1(\W_x,\W_x) = (T,A)$, then 
		\[
		A = \set{R\in \Run(T): \exists k<\omega \set{R(2n+1):n\ge k}\notin \W_x}.
		\]
		
		\item[$\g_1(\W_x,\Gamma_x)$] The game is played in the same way as $\g_1(\W_x,\W_x)$, but now\\
		$\bob$ wins the run $\seq{A_0,a_0,\dotsc, A_n,a_n,\dotsc}$ if $\lim_{n\to\infty}a_n=x$, and $\ali$ wins otherwise. That is, if $\g_1(\W_x,\Gamma_x) = (T,A)$, then 
		\[
		A = \set{R\in \Run(T): \exists k<\omega \left((R(2n+1))_{n\ge k}\notin \Gamma_x\right)}.
		\]
	\end{description}
	For these games, $\M(\g_1(\W_x,\W_x)) = \M(\g_1(\W_x,\Gamma_x))=\W_x\cup X$.	
\end{ex}

\begin{ex}[Covering games]\label{EX_Covering}
	Given a space $X=(X,\tau)$, we call $\mU\subseteq \tau$ an \emph{$\omega$-cover} if 
	\[
	\forall F\in [X]^{<\omega}(\exists U\in\mU(F\subseteq U)),
	\] 
	with $[X]^{<\omega}$ denoting the set of finite subsets of $X$.
	
	A sequence $(U_n)_{n<\omega}\in \tau^{\omega}$ is called a $\gamma$-cover if, for every infinite $S\subseteq \omega$, $\set{U_n:n\in S}$ is an $\omega$-cover.
	
	Letting
	\begin{gather*}
		\W=\set{\mU\subseteq \tau: \mU \text{ is an $\omega$-cover}},\\
		\Gamma = \set{(U_n)_{n<\omega}\in \tau^\omega: (U_n)_{n<\omega} \text{ is a $\gamma$-cover}}.
	\end{gather*}
	
	we consider the following games.
	\begin{description}
		\item[$\g_1(\W,\W)$] In each inning $n<\omega$,
		\begin{itemize}
			\item $\ali$ chooses an $\omega$-cover for $X$;
			\item $\bob$ responds by choosing an open set $U_n\in\mU_n$.
		\end{itemize} 
		$\bob$ wins the run $\seq{\mU_0,U_0,\dotsc, \mU_n,U_n,\dotsc}$ if $\set{U_n:n\ge k}$ is an $\omega$-cover for every $k<\omega$, and $\ali$ wins otherwise. In other words, if $\g_1(\W,\W) = (T,A)$, then 
		\[
		A = \set{R\in \Run(T): \exists k<\omega \set{R(2n+1):n\ge k}\notin \W}.
		\]
		
		\item[$\g_1(\W,\Gamma)$] The game is played in the same way as $\g_1(\W,\W)$, but now\\  $\bob$ wins the run $\seq{\mU_0,U_0,\dotsc, \mU_n,U_n,\dotsc}$ if, for every infinite set $S\subseteq\omega$, $\set{U_n:n\in S}$ is an $\omega$-cover, 
		and $\ali$ wins otherwise.	That is, if $\g_1(\W,\Gamma) = (T,A)$, then 
		\[
		A = \set{R\in \Run(T): \exists k<\omega \left((R(2n+1))_{n\ge k}\notin \Gamma\right)}.
		\]
		
	\end{description}
	For these games, $\M(\g_1(\W,\W)) = \M(\g_1(\W,\Gamma)) = \W\cup \tau$.
\end{ex}

Examples \ref{EX_Tightness} and \ref{EX_Covering} will be further explored in Section \ref{SEC_TopGames} when we look at some topological games as functors. 

We note that our formulations of the games $\g_1(\W_x,\W_x)$, $\g_1(\W_x,\Gamma_x)$, $\g_1(\W,\W)$ and $\g_1(\W,\Gamma)$ as given in Examples \ref{EX_Tightness} and \ref{EX_Covering} deviate from those in the literature (as described in, {\em e.g.}, \cite{Scheepers1997}). In fact, ours are more general since they work also for types of spaces $X$ not admitted previously, such as finite spaces. However, game theoretically they are equivalent to presentations found in the literature, in terms of the existence of winning strategies.

Of course, any given games lead to new ones. For instance, forming subgames proceeds as follows.

\begin{defn}\label{DEF_SubGames}
	A game $G' = (T',A')$ is a \textit{subgame} of a game $G=(T,A)$, written as $G'\le G$, if: 
	\begin{itemize}
		\item[(a)] $T' \subseteq T$,
		\item[(b)] $A' = A \cap \Run(T')$.
	\end{itemize}
\end{defn}

Since $G'$ is assumed to be a game, condition (a) just says that $T'$ is a pruned subtree of $T$, and then condition (b) implies 
\[
(T',A')\le (T,A) \iff A'= A \cap \Run(T').
\]
Briefly, subgames of $G=(T,A)$ are \emph{determined} by pruned subtrees of $T$. Hence, we may write $T'\le G$, rather than $(T', A \cap \Run(T'))\le (T,A)$, whenever $T'$ is a pruned subtree of $T$. 

The next proposition follows trivially from Definition \ref{DEF_SubGames}.
\begin{prop}\label{PROP_UnionSubGame}
	For any family $\set{G_\alpha = (T_\alpha, A_\alpha): \alpha \in \kappa}$ of subgames of a game $G = (T,A)$ one has $\bigcup_{\alpha \in \kappa}T_\alpha\le G$. We write $\bigcup_{\alpha \in \kappa}G_\alpha$ for the subgame $(\bigcup_{\alpha \in \kappa}T_\alpha, \bigcup_{\alpha \in \kappa}A_\alpha)$ of $G$.
\end{prop}

Studying games is usually a matter of studying strategies for such games. In what follows we give two distinct formal definitions of strategies for $\ali$ within the setting of Definition \ref{DEF_infinite_game}, one in terms of a mapping which tells $\ali$ which move to make at each of her turns and the other, which implicitly extracts such a mapping from a subgame via a uniqueness condition. From a game-theoretic perspective, the two formulations are obviously equivalent, and therefore we will not distinguish them terminologically. For the sake of completeness, we formulate the (more often used) subgame version of a strategy also from $\bob$'s perspective.

\begin{defn}[$\ali$'s strategy mapping]\label{DEF_AStratMap}
	For a non-empty game $G = (T,A)$, a mapping $\gamma\colon S_\gamma\to \M(G)$ is a \textit{strategy for $\ali$} in $G$ if
	\begin{enumerate}
		\item[(a)] $S_\gamma\subseteq T$;
		\item[(b)] $\seq{\,}\in S_\gamma$;
		\item[(c)] every $t\in S_\gamma$ is $\ali$'s turn, and $\gamma(t)$ is a possible move for $\ali$ at the moment $t$;
		\item[(d)] for every $t\in S_\gamma$ one has $t^\smallfrown x^\smallfrown y\in S_\gamma$ if, and only if, $x=\gamma(t)$ and $y$ satisfies  $t^\smallfrown \gamma(t)^\smallfrown y\in T$.
	\end{enumerate} 
	We say that $\gamma$ is a \textit{winning strategy} if, furthermore, every increasing chain of finite sequences $\left(t_n:n<\omega\right)$ of $S_\gamma$
	satisfies $\lim\limits_{\longrightarrow} t_n\in A$, where $\lim\limits_{\longrightarrow} t_n$ is the infinite sequence which is the common extension of all the $t_n$s.
\end{defn}

So, a strategy $\gamma$ for $\ali$ in a game $G=(T,A)$ is a mapping with domain $S_\gamma\subseteq T$ which tells $\ali$ which moves to make at the moments in $S_\gamma$, subject to the following constraints: condition (b) ensures that $\gamma$ makes $\ali$ prepared to start the game, condition (c) ensures that $\gamma$ tells $\ali$ to respond with valid moves in $G$, and condition (d) ensures that $\gamma$ makes $\ali$ prepared to respond at every possible scenario that might come up when playing according to $\gamma$, and that $S_\gamma$ is minimal amongst the sets with properties (a--d).

\begin{defn}[$\ali$'s strategy subgame]\label{DEF_AStrat}
	A pair $\gamma = (T_\gamma, A_\gamma)$ is a \textit{strategy for $\ali$ in the game $G=(T,A)$} if 
	\begin{enumerate}
		\item[(a)] $\gamma$ is a subgame of $G$,
		\item [(b)] $T_\gamma\neq \emptyset$,
		\item[(c)] for all $t\in T_\gamma$, if $t$ is $\ali$'s turn, then there is a \textit{unique} $x\in \M(G)$ such that $t^\smallfrown x\in T_\gamma$,
		\item[(d)] for all $t\in T_\gamma$, if $t$ is $\bob$'s turn, then $t^\smallfrown x\in T_\gamma$ for all $x$ such that $t^\smallfrown x \in T$.
	\end{enumerate} 
	Furthermore, we say that $\gamma$ is a \textit{winning strategy (for $\ali$)} if $A_\gamma = \Run(T_\gamma)$.
\end{defn}

\begin{defn}[$\bob$'s strategy subgame]\label{DEF_BStrat}
	A pair $\sigma = (T_\sigma, A_\sigma)$ is a \textit{strategy for $\bob$ in the game $G=(T,A)$} if 
	\begin{enumerate}
		\item[(a)] $\sigma$ is a subgame of $G$,
		\item [(b)] $T_\sigma\neq \emptyset$,
		\item[(c)] for all $t\in T_\sigma$, if $t$ is $\bob$'s turn, then there is a \textit{unique} $x$ such that $t^\smallfrown x\in T_\sigma$,
		\item[(d)] for all $t\in T_\sigma$, if $t$ is $\ali$'s turn, then $t^\smallfrown x\in T_\sigma$ for all $x$ such that $t^\smallfrown x \in T$.
	\end{enumerate} 
	Furthermore, we say that $\sigma$ is a \textit{winning strategy (for $\bob$)} if $A_\sigma = \emptyset$.
\end{defn}

\medskip
We now turn to the remaining aspect of our upcoming categories: their morphisms.
Two distinct options interest us:

\begin{defn}[Game morphisms]\label{DEF_Amorph}
	Let $G_1=(T_1,A_1)$ and $G_2=(T_2,A_2)$ be games and $T_1\stackrel{f}{\to}T_2$ be chronological. Then 
	\begin{itemize}
		\item[(A)] $f$ is an $\A$-\textit{morphism} if $\overline{f}(R)\in A_2$ for every run $R\in A_1$, and
		\item[(B)] $f$ is a $\B$-\textit{morphism} if $\overline{f}(R)\in \Run(T_2)\setminus A_2$ for every run $R\in \Run(T_1)\setminus A_1$.
	\end{itemize}
	We indicate these properties respectively by $G_1\stackrel{f}{\to}_\A G_2$ and $G_1\stackrel{f}{\to}_\B G_2$, omitting the subscript to the arrow when the context provides sufficient clarity.		
\end{defn}

Intuitively, an $\A$-morphism between $G_1=(T_1, A_1)$ and $G_2=(T_2, A_2)$ is a mapping $f\colon T_1\to T_2$ that \textit{chronologically} respects the histories of the games and guarantees that, if $\ali$ wins a run $R$ of $G_1$, then $\ali$ also wins in the run $\overline{f}(R)$, and likewise for $\B$-morphisms. 

Of course:

\begin{prop}\label{PROP_morphcat}
	$\A$-morphisms are closed under composition, and so are $\B$-morphisms. The identity map of the game tree of a game is both an $\A$-morphism and a $\B$-morphism.
\end{prop}

In view of Proposition \ref{PROP_morphcat}, we have two emergent categories: 

\begin{itemize}
	\item $\Games_\A$: 
	\begin{itemize}
		\item objects are games,
		\item morphisms are $\A$-morphisms.
	\end{itemize}
	\item $\Games_\B$: 
	\begin{itemize}
		\item objects are games,
		\item morphisms are $\B$-morphisms.
	\end{itemize}  
\end{itemize} 

\comment{
	Our primary interest lies in studying the categories $\Games_{\A}$ and $\Games_{\B}$. Exploring their common ``base category'' $\Gmes$ assists in this pursuit, since its objects, considered as ``games with forgotten $A$\,s'', {\textit i.e.}, as ``games without their payoff sets'', appear in many distinct guises, which help us understand the overall structure of $\Gmes$.
}

\begin{remark}
	To better contextualize our approach with what has already been done in the literaure, we should point out that the category $\Gmes$ can easily be embedded in the category $\mathbf{Tree}$ of \cite{Streufert2018} (although such embedding is not full, since morphisms in Streufert's $\mathbf{Tree}$ are not required to map the root of its trees into other roots, while chronological mappings always associate $\seq{\,}$ to $\seq{\,}$). 
	
	We note that, although the objects of the category $\Games_{\A}$ (and $\Games_{\B}$) can naturally be seen as a subclass of objects in the category $\mathbf{Gm}$ of \cite{Streufert2021}, $\Games_{\A}$ (or $\Games_{\B}$) cannot be embedded in $\mathbf{Gm}$ by such association: it can be shown that $\mathbf{Gm}$ (restricted to our class of games) only allow for a morphism $f\colon (T,A)\to (T',A')$ with nonempty $A'$ and $\Run(T')\setminus A'$ to be \emph{both} an $\A$ and $\B$-morphism, while this is not be the case in $\Games_{A}$. Thus, the identity map $f\colon T\to T$ with
	\begin{gather*}
		T = \set{\seq{i:k<n}i\in \{0,1\}, \, n<\omega}
	\end{gather*} 
	defines an $\A$-morphism if we consider 
	\[f\colon (T,\emptyset)\to (T,\{\seq{0:n<\omega}\}),\] 
	but it does not define a game morphism in $\mathbf{Gm}$. And indeed, such distinction between Streufert's $\mathbf{Gm}$ and our $\Games_{\A}$ and $\Games_{\B}$ plays a crucial role in almost all of our results that follow (such as exploiting the functorial nature of prominent games as in Theorems  \ref{THM_ClassOmega}, \ref{THM_ClassGamma} in Section \ref{SEC_TopGames} and in Theorem \ref{THM_BMUni_Fun}).
\end{remark}

The following proposition characterizes some special types of morphisms in $\Games_{\A}$ and $\Games_{\B}$. We prove the first statement since it may come as a surprise that there are non-injective monomorphisms in these categories,  but we omit the elementary proofs of the other statements.

\begin{prop}\label{PROP_morphproperties}
	Let $G_1=(T_1, A_1)$ and $G_2=(T_2, A_2)$ be games. Then:
	\begin{itemize}
		\item[(a)] An $\A$-morphism $G_1\stackrel{f}{\to}_\A G_2$ is a monomorphism in $\Games_\A$ if, and only if, the mapping $\overline{f}\colon \Run(T_1)\to \Run(T_2)$ is injective. Likewise for $\B$-morphisms. 
		
		\item[(b)] An $\A$-morphism $G_1\stackrel{f}{\to}_\A G_2$ is an epimorphism in $\Games_\A$ if, and only if, the mapping $f\colon T_1\to T_2$ is surjective. Likewise for $\B$-morphisms. 
		
		\item[(c)]\label{PROP_IT_iso} A mapping $f\colon T_1\to T_2$ becomes an isomorphism $G_1\stackrel{f}{\to}_\A G_2$ in $\Games_{\A}$ if, and only if, $f$ is a bijection that is both an $\A$-morphism and a $\B$-morphism. These poroperties characterize also the isomorphisms in $\Games_\B$.			
		\item[(d)] If $G_1\stackrel{f}{\to} G_2$ and $T\le G_1$, then $f[T]\le G_2$.
		
		\item[(e)]\label{PROP_IT_embbeding} If $f\colon T_1\to T_2$ is injective and both an $\A$-morphism and a $\B$-morphism, then
		\[
		\Run(f[T_1])\cap A_2 = \overline{f}[A_1],
		\]
		and $\left(f[T_1],\overline{f}[A_1]\right)$ is a subgame of $G_2$ that is $A$-isomorphic to $G_1$.
	\end{itemize}
\end{prop}
\begin{proof}
	(a) With $G$ denoting the generating game $(\{*\}^{<\omega},\emptyset)$ of Example \ref{EX_boring}, one first observes that the A-morphisms $G\to G_1$ are in bijective correspondence to the runs of the game $G_1$. Hence, assuming $f$ to be monic in $\Games_{\A}$ and letting $\overline{f}(R)=\overline{f}(R')$ for $R,R'\in\mathrm{Run}(T_1)$, the corresponding A-morphisms $r,r':G\to G_1$ satisfy $f\circ r=f\circ r'$. Consequently, $r=r'$, and then $R=R'$ follows. 
	
	Conversely, assuming $\overline{f}$ to be injective, we consider any A-morphisms $g,h:G'\to G_1$ with $f\circ g=f\circ h$ for some game $G'=(T',A')$. Then $\overline{f}\circ\overline{g}=\overline{f}\circ\overline{h}$, and since any $t'\in T'$ may be ``extended" to a run $R'$ in $G'$ with $R'\restrict n=t'$ where $n=|t'|$, we obtain $\overline{g}(R')=\overline{h}(R')$ and then $g(t')=\overline{g}(R')\restrict n=\overline{h}(R')\restrict n=h(t')$. Therefore, $g=h$. 
\end{proof}

\begin{remark}\label{REM_generator} 
	{\em(1)} Similarly as in the above proof one shows that the generating game $G=(\{*\}^{<\omega},\emptyset)$ has the property that, for all distinct $\A$-morphisms $f,g:G_1\to G_2$, there is some $\A$-morphism $r:G\to G_1$ with $f\circ r\neq g\circ r$; that is: $G$ is indeed a {\em generator} (also called {\em separator} \cite{Adamek1990}) of the category $\Games_{\A}$.
	
	{\em(2)} For a chronological map $f:T_1\to T_2$ of game trees, injectivity of $f$ obviously implies injectivity of $\overline{f}:\mathrm{Run}(T_1)\to\mathrm{Run}(T_2)$, and it is not difficult to see that the converse implication generally fails. (For instance, let $T_1$ be the game tree $T$ of {\em Example \ref{EX_PreImgNotSub}} below and let $f$ identify only the two moments  
	of length $1$.)  One therefore has examples of monomorphisms in $\Games_{\A}$ which fail to be injective maps, even surjective ones. In particular, a morphism that is simultaneously monic and epic may not be an isomorphism in the category $\Games_{\A}$.
	
	{\em(3)} For a chronological map $f:T_1\to T_2$, the surjectivity of $\overline{f}$ implies the surjectivity of $f$, but the converse statement generally fails. In order to see this, consider for $T_2$ the game tree $T$ of {\rm Example \ref{EX_cogen}} and let
	\[
	T_1=\set{\seq{k:i<n}, k,n<\omega}
	\]
	and $(T_1,\emptyset)\stackrel{f}{\to} (T,\mathrm{Run}(T))$ be such that
	\[
	f(\seq{k:i<n}) = \begin{cases}
		\seq{1:i<k}^\smallfrown \seq{0:j<n-k}, \text{ if $n>k$},\\
		\seq{1:i<n}, \text{ otherwise.}
	\end{cases}
	\]
	Then $\overline{f}(R)\neq\seq{1:n<\omega}$ for any $R\in \Run(T_1)$.
	
	\emph{(4)} The game $G=(T,\Run(T))$ of {\rm Example \ref{EX_cogen}} is indeed a cogenerator (also coseparator \cite{Adamek1990}) in $\Games_{\rm A}$. In order to see this, suppose $f,g\colon G_1=(T_1,A_1)\to G_2=(T_2,A_2)$ are distinct and fix $s\in T_1$ such that $f(s)\neq g(s)$. Without loss of generality, let us assume that $f(s\restrict n)=g(s\restrict n)$ for every $n<|s|$. Define $h\colon T_2\to T$ as
	\[
	h(t) = \begin{cases}
		\seq{1:i<|s|-1}^\smallfrown \seq{0:j<|t|-(|s|-1)}, \text{ if $|t|>|s|$ and $t\restrict |s| = f(s)$},\\
		\seq{1:i<|t|}, \text{ otherwise}.
	\end{cases}
	\]
	Then $h\circ f(s) = \seq{1:i<|s|}^\smallfrown 0\neq \seq{1:i<|s|}^\smallfrown 1 = h\circ f(s)$. It remains to be shown that $h$ is chronological. Indeed, for any $t\in T_2$, if $|t|<|s|$ or  $t\restrict |s|\neq f(s)$, then 
	\[
	h(t\restrict n) = \seq{1:i<n} = \seq{1:i<|t|}\restrict n = h(t)\restrict n;
	\]
	otherwise,
	\[
	h(t\restrict n) = \begin{cases}
		\seq{1:i<|s|-1}^\smallfrown \seq{0:j<n-(|s|-1)} = h(t)\restrict n, \text{ if $n>|s|$},\\
		\seq{1:i<n} = h(t)\restrict n, \text{ else}.
	\end{cases}
	\]
	Hence, $h$ is chronological.
\end{remark}

Given games $G_1=(T_1, A_2)$, $G_2=(T_2, A_2)$, paraphrasing item (c) of Proposition \ref{PROP_morphproperties}, we have that a mapping $f\colon T_1\to T_2$ is an $\A$-\textit{isomorphism} if, and only if, it is a $\B$-\textit{isomorphism}, in which case we can just say that $f$ is an isomorphism.

Property (e) of Proposition \ref{PROP_morphproperties} deserves a special name:

\begin{defn}\label{DEF_Emb}
	Let $G_1=(T_1,A_1)$ and $G_2=(T_2,A_2)$ be games. We say that a chronological mapping $f\colon T_1\to T_2$ is an \textit{embedding} of $G_1$ into $G_2$ {(indicated by $G_1\stackrel{f}{\embed}G_2$)} if $f$ is injective and both an $\A$- and a $\B$-morphism.
\end{defn}

Game embeddings are characterized as the strong, and even regular, monomorphisms in the categories $\Games_{\A}$ and $\Games_\B$: see Corollary \ref{COR_ortfactsyst} below.

\medskip
Trivially, an isomorphism maps a winning strategy in the domain to a winning strategy in the range. Next we identify a much larger class of morphisms of games preserving winning strategies.

\begin{defn}
	For game trees $T_1$, $T_2$, we call a chronological map  $T_1\stackrel{f}{\to} T_2$ {\em locally surjective at $t\in T_1$} if, for every element $y$ with $f(t)^\smallfrown y\in T_2$, there is an element $x$ with  $t^\smallfrown x\in T_1$ and $f(t^\smallfrown x)=f(t)^\smallfrown y$. 
	The mapping $f$ is {\em locally surjective} if
	\begin{itemize}
		\item $T_1= \emptyset$ implies $T_2= \emptyset$, and
		\item $f$ is locally surjective at every $t\in T_1$.
	\end{itemize}
\end{defn}

A locally surjective $T_1\stackrel{f}{\to} T_2$ is necessarily surjective. (This is true by definition when $T_1= \emptyset$, and for $T_1\neq\emptyset$ note in particular that$f(\seq{\,})=\seq{\,}$ holds.)

The following proposition provides a categorical characterization of locally surjective morphisms. Recall that a family of morphisms $g_i:A_i\to B\;(i\in I)$ in a category $\mathbf{C}$ is \emph{collectively epic in $\mathbf{C}$} if, for any morphisms $f,f'\colon B\to C$ of $\mathbf{C}$ with $f\circ g_i= f'\circ g_i$ for all $i\in I$, one necessarily has $f=f'$.

\begin{prop}
	Let $T_1$, $T_2$ be game trees and $T_1\stackrel{f}{\to}T_2$ be chronological. Then $f$ is locally surjective if, and only if, there is a collectively epic family $(T_2\stackrel{g_i}{\to}T_1:i\in I)$ of morphisms in $\Gmes$ such that $f\circ g_i = \id_{T_2}$ for all $i\in I$.
\end{prop}
\begin{proof}
	The assertion is trivial if $T_1=\emptyset$, so we may assume $T_1\neq \emptyset$ throughout. Furthermore, one easily verifies that a family $(T_2\stackrel{g_i}{\to}T_1:i\in I)$ is collectively epic in $\Gmes$ if, and only if, $\bigcup_{i\in I}g_i[T_2]=T_1$. Now, given such a family with $f\circ g_i = \id_{T_2}$ for all $i\in I$, we must confirm that $f$ is locally surjective. Indeed,
	considering $t\in T_1$ and $y$ with $f(t)^\smallfrown y\in T_2$, we first find some $i\in I$ and $t'\in T_2$ with $t=g_i(t')$. Since $f\circ g_i=\id_{T_2}$ we get $f(t)=t'$ and then $g_i(f(t))=t$. Since $g_i$ is chronological, this forces $g_i(f(t)^\smallfrown y)=t^\smallfrown x\in T_1$ with some element $x$, and  $f(t^\smallfrown x)=f(t)^\smallfrown y$ follows.
	
	Conversely, for $f$ locally surjective, it suffices to construct a family $(T_2\stackrel{g_t}{\to}T_1:t\in T_1)$ with $f\circ g_t=\id_{T_2}$ and $t\in g_t[T_2]$, for all $t\in T_1$. So, fixing $t\in T_1$, we establish the chronological mapping $g_t\colon T_2\to T_1$ by recursively defining $g_t(s)\in T_1$ with $f(g_t(s))=s$, for all $s\in T_2=\bigcup_{n} T_2(n)$, as follows.
	\begin{itemize}
		\item Put $g_t(\seq{\,})=\seq{\,}$.
		\item Suppose that $g_t$ is defined on $T_2\restrict n$ as desired, and now consider $s\in T_2(n+1)$. If $s=f(t\restrict n+1)$, then simply put $g_t(s)=t\restrict n+1$. Otherwise, we write $s$ as $s=s'^\smallfrown y$ with $s'\in T_2(n)$, and we can assume that $g_t(s')\in T_1(n)$ is already defined, satisfying $f(g_t(s'))=s'$. Now the local surjectivity of $f$ gives us an $x$ with $g_t(s')^\smallfrown x\in T_1$ and $f(g_t(s')^\smallfrown x)=s'^\smallfrown y=s$. Hence, putting $g_t(s)=g_t(s')^\smallfrown x$ completes the induction step.
	\end{itemize}
	Clearly, the mapping $g_t$ is chronological and satisfies $g_t(f(t))=t$ and $f\circ g_t=\id_{T_2}$, as desired.
\end{proof}

If $G_1\stackrel{f}{\to}G_2$ and $\gamma=(T_\gamma,A_\gamma)$ is a strategy for $\ali$ in $G_1$, for $f[T_\gamma]$ to be a strategy for $\ali$ in $G_2$, it is necessary that $f$ be locally surjective at moments which are $\bob$'s turn: indeed, $f[T_\gamma]$ must account for every possible response $\bob$ might give to $\ali$ in $G_2$ in order to be a strategy in $G_2$. It should also be clear that this necessary condition is sufficient as well. Therefore:

\begin{cor}\label{LEMMA_StratImage}
	Let $G_1\stackrel{f}{\to}_\A G_2$ and $G_1\stackrel{g}{\to}_\B G_2$ be locally surjective morphisms of the games $G_1=(T_1,A_1)$ and $G_2=(T_2,A_2)$. Then:
	\begin{itemize}
		\item[(a)] If $\gamma=(T_\gamma,A_\gamma)$ is a winning strategy for $\ali$ in $G_1$ (so that $A_{\gamma}=\Run(T_{\gamma})$), then $(f[T_\gamma], \overline{f}[\Run(T_\gamma)])$ is a winning strategy for $\ali$ in $G_2$.
		\item[(b)] If $\sigma =(T_\sigma,A_\sigma)$ is a winning strategy for $\bob$ in $G_1$ (so that $A_{\gamma}=\emptyset$), then $(g[T_\sigma], \emptyset)$ is a winning strategy for $\bob$ in $G_2$. 
	\end{itemize}
\end{cor}

We saw in item (d) of Proposition \ref{PROP_morphproperties} that the image of a game under a chronological mapping is always a subgame of the codomain. So, it is only natural to ask whether pre-images of subgames under chronological mappings are subgames as well. The following example shows that this is not necessarily the case:

\begin{ex}\label{EX_PreImgNotSub}
	Consider the set
	\[
	T =  \left(\bigcup_{n<\omega}\{\seq{0:i< n}\}\right) \cup \left(\bigcup_{n<\omega}\{\seq{1:i< n}\}\right),\\
	\]
	of finite sequences that are constantly $0$ or constantly $1$, and let $T' = \{\seq{\,}\}\cup\set{\seq{0}^\smallfrown t: t\in  T}$. With $A=A'=\emptyset$
	we have the surjective $\A$-morphism $(T,A)\stackrel{f}{\to}_\A(T',A')$ defined by
	\[
	f(\seq{j:i< n})=\begin{cases}
		\seq{\,}, \text{ if $n=0$,}\\
		\seq{0}^\smallfrown\seq{j:i< n-1}, \text{ otherwise.}
	\end{cases}
	\]
	where $j\in\{0,1\}$. For the subgame
	\[
	T''=\{\seq{\,}\}\cup\set{\seq{0}^\smallfrown t: t\in  \bigcup_{n<\omega}\{\seq{1:i< n}\}},
	\]
	of $T'$ we have 
	\[
	f^{-1}(T'')=\{\seq{0}\}\cup \left(\bigcup_{n<\omega}\{\seq{1:i<n}\}\right),
	\]
	but this set is not pruned: although $\seq{0}\in f^{-1}(T'')$, there is no $x$ such that $\seq{0}^\smallfrown x \in f^{-1}(T'')$). Consequently, $f^{-1}(T'')\not\le G$.
\end{ex}

We note that $f^{-1}(T'')$ in Example \ref{EX_PreImgNotSub} is a subtree of $T$--the only obstruction for $f^{-1}(T'')\le G$ is the fact that $f^{-1}(T'')$ is not pruned. This observation may be easily generalized, as follows:

\begin{lemma}\label{PROP_SubChar}
	For games $G=(T,A)$ and $G'=(T',A')$, $G\stackrel{f}{\to}G'$ and $T''\le G'$, $f^{-1}(T'')\le G$ if, and only if, $f^{-1}(T'')$ is a pruned tree, that is: if for every $t\in f^{-1}(T'')$ there is an $x\in \M(G)$ such that $t^\smallfrown x\in f^{-1}(T'')$.
\end{lemma}

The Lemma can be usefully applied to locally surjective morphisms:

\begin{prop}
	Given games $G=(T,A)$, $G'=(T',A')$, $G\stackrel{f}{\to}G'$ and $G''=(T'',A'')\le G'$, if the morphism $G\stackrel{f}{\to}G'$ is locally surjective, then $f^{-1}(T'')\le G$.
\end{prop}
\begin{proof}
	We show that, for every $t\in f^{-1}(T'')$, there is an $x\in \M(G)$ with $f(t^\smallfrown x)\in T''$.
	Indeed, since $T''$ is pruned, from $f(t)\in T''$ one has $f(t)^\smallfrown y\in T''$ for some $y\in \M(G')$. Now, because $f$ is locally surjective, there is an $x\in \M(G)$ with $t^\smallfrown x\in T$ and $f(t^\smallfrown x)=f(t)^\smallfrown y\in T''$, as desired.
\end{proof}

\begin{cor}\label{LEMMA_StratPreimage}
	Let $G\stackrel{f}{\to}_\A G'$ and $G\stackrel{g}{\to}_\B G'$ be locally surjective morphisms of the games $G=(T,A)$ and $G'=(T',A')$. Then:
	\begin{itemize}
		\item[(a)] If $\gamma=(T_\gamma,A_\gamma)\le G'$ is a winning strategy for $\ali$, then $g^{-1}(T_\gamma)$ contains a winning strategy for $\ali$ in $G$;
		\item[(b)] If $\sigma =(T_\sigma,A_\sigma)\le G'$ is a winning strategy for $\bob$, then $f^{-1}(T_\sigma)$ contains a winning strategy for $\bob$ in $G$.
	\end{itemize}
\end{cor}
\begin{proof}
	By the previous corollary, $g^{-1}(T_\gamma)\le G$ and $f^{-1}(T_\sigma)\le G$. Moreover, since $g$ is a $\B$-morphism and $\gamma$ is a winning strategy for $\ali$, we have $A\cap \Run(g^{-1}(T_\gamma))=\Run(g^{-1}(T_\gamma))$. Likewise, since $f$ is an $\A$-morphism and $\sigma$ is a winning strategy for $\bob$, we have $\Run(f^{-1}(T_\sigma))=\emptyset$. So, in order to obtain the desired winning strategies, it suffices to consider \textit{any} strategy $\tilde{\gamma}=(T_{\tilde{\gamma}},A_{\tilde{\gamma}})$ for $\ali$ in $G_\A:=(g^{-1}(T_\gamma),\Run(g^{-1}(T_\gamma)))$, and \textit{any} strategy $\tilde{\sigma}=(T_{\tilde{\sigma}},A_{\tilde{\sigma}})$ for $\bob$ in $G_\B:=(f^{-1}(T_\sigma),\emptyset)$.
	
	Indeed, let $t\in T_{\tilde{\gamma}}$ be such that $t$ is $\bob$'s turn, and let $x$ be such that $t^\smallfrown x\in G$.
	Then, because $\gamma$ is a strategy for $\ali$ in $G'$, $g(t^\smallfrown x)\in T_\gamma$, so $t^\smallfrown x\in g^{-1}(T_\gamma)$. Now, because $\gamma_G$ is a strategy for $\ali$ in $G_\A$, we obtain $t^\smallfrown x\in T_{\gamma_G}$, as desired.
	
	Finally, we note that every run in $G_\A$ is won by $\ali$ and conclude that $\gamma_G$ is a winning strategy for $\ali$ in $G$.
	
	Of course, the proof is analogous for $\tilde{\sigma}$.
\end{proof}

Corollaries \ref{LEMMA_StratImage} and \ref{LEMMA_StratPreimage} will be useful later when we show how some classical theorems for topological games may be obtained using locally surjective mappings.

\subsection{Arboreal and functorial games}\label{SEC_SubTree}
Games are more than just rooted trees of ``choice nodes'', since they must also specify the winning criteria for the players. In the context of trees, this amounts to specifying which are the branches won by $\ali$. In this section we employ a well-known general categorical procedure (see \cite{Adamek1990}, 5.40), for adding the missing pay-off sets to the objects of the categories appearing in the diagram of Corollary \ref{COR_Tree_SetNop}, which we recall first.

\begin{defn}
	Let $\mathbf{C}$ be a category and $K\colon \mathbf{C}\to \Sets$ be a functor. Then the category $$\mathsf{Sub}_K(\mathbf C)$$ is defined to have as
	\begin{itemize}
		\item objects all pairs $(X,A)$ with $X\in \Obj{\mathbf{C}}$ and $A\subseteq KX$, and as
		\item morphisms $(X,A) \stackrel{f}{\to}(X',A')$ those morphisms $X\stackrel{f}{\to}X'$ in $\mathbf{C}$ with $Kf[A]\subseteq A'$.
	\end{itemize} 
	Briefly, $\mathsf{Sub}_K(\mathbf C)$ is the full subcategory of the comma category $(\Sets\downarrow K)$ given by subset inclusion maps $A\hookrightarrow KX$, also known as {\em subscones}.  We sometimes write just $\mathsf{Sub}(\mathbf C)$ for this category when the context makes it clear which functor $K$ we are considering.
\end{defn}

A known (see \cite{Adamek1990}, 21.8(2)) and easy, but useful, fact reads as follows:

\begin{prop}\label{PROP_TopFunctor}
	For every functor $K\colon \mathbf{C}\to \Sets$, the functor $U\colon \mathsf{Sub}_K(\mathbf C)\to\mathbf C$ which forgets the selected subsets, is topological. As a consequence, $U$ has both, a left adjoint and a right adjoint, and the category $\mathsf{Sub}_K(\mathbf C)$ is complete or cocomplete if (and only if) $\mathbf C$ has the respective property. 
\end{prop}

Indeed, the topologicity of $U$ follows from the easy fact that, for any family $f_i:X\to Y_i\;(i\in I)$ of $\mathbf C$-morphisms of a fixed object $X$ to objects $Y_i$ that are provided with subsets $B_i\subseteq KY_i$, there is a so-called $U$-{\em initial} structure on $X$ making every $f_i$ a morphism in $\mathsf{Sub}_K(\mathbf C)$, given by the set $A=\bigcap_{i\in I}(Kf_i)^{-1}[B_i]$. Dually, for any family $g_i;X_i\to Y\;(i\in I)$ of $\mathbf C$-morphisms of objects $X_i$ that are provided with subsets $A_i\subseteq KX_i$ to a fixed object $Y$, there is a so-called $U$-{\em final} structure on $Y$ making every $g_i$ a morphism in $\mathsf{Sub}_K(\mathbf C)$, namely $B=\bigcup_{i\in I}(Kg_i)[A_i]$.

An important result about topological functors is Wyler's {\em Taut Lift Theorem} (see \cite{Wyler1971a, Wyler1971b}, and \cite{Tholen1978} for an extension to a wider class of functors). We record here only a very simple instance of it which suffices for the application in the remainder of this section, as follows. At first, we let $H:\mathbf C\to \mathbf D$ and $K:\mathbf C\to\Sets,\;L:\mathbf D\to\Sets$ be functors with $L\circ H=K$ and consider the ``lifted'' functor $\widehat{H}$ that makes the diagram (in which $V$ is, like $U$, forgetful)
\begin{center}
	$\xymatrix{\mathsf{Sub}_K(\mathbf C)\ar[d]_U\ar[rr]^{\widehat{H}} 
		&& \mathsf{Sub}_L(\mathbf D)\ar[d]^V\\
		\mathbf C\ar[rr]^H\ar[rd]_K && \mathbf D\ar[ld]^L\\
		& \Sets & \\
	}$
\end{center}
commute; it sends an object $(X,A)$ in $\mathsf{Sub}_K(\mathbf C)$ to $(HX,A)$ in $\mathsf{Sub}_L(\mathbf D)$ and maps morphisms like $H$ does. Since $\widehat{H}$ leaves the subsets unchanged, trivially $\widehat{H}$ sends $U$-initial or U-final families to $V$-initial or $V$-final families, respectively. As this is the essential hypothesis of the Taut Lift Theorem, we can deduce the following corollary which, of course,  may also be easily checked in a direct manner.

\begin{cor}\label{COR_TautLift}
	If $H$ has a left-adjoint functor $F$, then also $\widehat{H}$ has a left-adjoint functor, $\tilde{F}$, which is a lifting of $F$ along $V$ and $U$, just like $\widehat{H}$ is a lifting of $H$ along $U$ and $V$,  i.e., $U\circ\tilde{F}=F\circ V$. The same statement holds when ``left'' is traded for ``right''; furthermore, if $H$ is faithful, full, or an equivalence, then $\widehat{H}$ has the respective property.
\end{cor}

\begin{proof}
	We indicate only the proof of the first statement. 
	For an object $(Y,B)$ in $\mathsf{Sub}_L(\mathbf D)$ and an $H$-universal arrow $\eta_Y:Y \to G(FY)$ for $Y$, with $FY$ in $\mathbf C$, one immediately shows that $\eta_Y$ serves also as an $\widehat{H}$-universal arrow $\eta_{(Y,B)}:(Y,B)\to \widehat{H}(\tilde{F}(Y,B))$ for $(Y,B)$, where $$\tilde{F}(Y,B):=(FY,(L\eta_Y)[B]).$$ Dually, in order to construct a right adjoint $\check{E}$ of $\widehat{H}$ from a right adjoint $E$ of $H$, one sees that an $H$-couniversal arrow $\varepsilon_Y:H(EY)\to Y$ for $Y$ serves also as an $\widehat{H}$-couniversal arrow $\varepsilon_{(Y,B)}:\widehat{H}(\check{E}(Y,B))\to (Y,B)$ for $(Y,B)$ when one puts $$\check{E}(Y,B)=(EY,(L\varepsilon_Y)^{-1}[B]).$$
	$$\xymatrix{\mathsf{Sub}_K(\mathbf C)\ar[rr]^{\bot}_{\bot}\ar[d]_U && \mathsf{Sub}_L(\mathbf D)\ar@/_0.8pc/[ll]_{\tilde{F}}\ar@/^0.8pc/[ll]^{\check{E}}\ar[d]^V\\
		\mathbf{C}\ar[rr]^{\bot}_{\bot} && \mathbf{D}\ar@/_0.8pc/[ll]_F\ar@/^0.8pc/[ll]^E\\
	}$$
\end{proof}

It is easy to see that the corollary still holds if the given functors $H,K,L$ satisfy only $L\circ H\cong K$, rather than the strict equality $L\circ H=K$. Indeed, with a slight adaptation of the definition of $\widehat{H}$ that utilizes this natural isomorphism, we still obtain a commutative rectangle of functors as above, and the assertions of the corollary remain intact. In this form, we now apply the above corollary to the following diagram of functors, of which $\mathrm I, \mathrm{Fun}$ and $\mathrm{Lim}$ have been defined in Section \ref{SEC_TreeTopos}, while $\mathrm{Run}$ and $\mathrm{Branch}$, defined on objects in Section \ref{SEC_Games}, allow for the (obvious and already mentioned) extension $f\mapsto\bar{f}$ to morphisms:

\begin{center}
	$\xymatrix{\Gmes\ar[r]^{\mathrm I}\ar[rd]_{\mathrm{Run}}^{\cong} & \Tree\ar[r]^{\mathrm{Fun}}\ar[d]|{\mathrm{Branch}} & \SetNop\ar[ld]^{\mathrm{Lim}}_{\cong}\\
		& \Sets & \\
	}$
\end{center}

We thus obtain the categories
\begin{itemize}
	\item $\mathsf{Sub}_{\mathrm{Run}}(\Gmes)$, coinciding exactly with $\Games_{\A},$
	\item $\mathsf{Sub}_{\mathrm{Branch}}(\Tree),$ written shortly as $\mathsf{Sub}(\Tree)$,
	\item $\mathsf{Sub}_{\mathrm{Lim}}(\SetNop),$ written shortly as $\mathsf{Sub}(\SetNop).$
\end{itemize}

By Proposition \ref{PROP_TopFunctor}, they come with topological functors to their respective base categories, which has some previously mentioned convenient consequences. For example, the forgetful functor $\mathrm{U}\colon\Sub{\Tree}\to\Tree$ being topological, any limit or colimit in $ \Sub{\Tree}$ can first be formed in $\Tree$ and then, with the help of the limit projections and colimit injections, provided with the obvious $\Sub{\Tree}$-structure. Furthermore, the functor $\mathrm{U}$ not only preserves all limits and colimits, it actually has both, a 
left adjoint, which specifies the empty set of branches, and a right adjoint, which specifies the set of all branches. 

Just like we formed the category $\Sub{\Tree}$ from $\Tree$ and $\mathsf{Sub}(\SetNop)$ from $\SetNop$, by restricting the functors $\mathrm{Branch}$ and $\mathrm{Lim}$ we obtain the full subcategories
\begin{itemize}
	\item $\TrGame:=\mathsf{Sub}_{\mathrm{Branch}_{|\PrTree}}(\PrTree).$
	\item $\FunGame:=\mathsf{Sub}_{\mathrm{Lim}_{|\mathsf{Epi}(\Sets)^{\omega^{\op}}}}(\mathsf{Epi}(\Sets)^{\omega^{\mathrm{op}}})$
\end{itemize}
of the categories $\Sub{\Tree}$ and $\mathsf{Sub}(\SetNop)$, respectively. Their objects are to be considered as games in their \textit{arboreal} and \textit{functorial} descriptions, as we demonstrate next.

Indeed, by Corollary \ref{COR_TautLift}, the functors forming the left triangle of the diagram of Corollary \ref{COR_Tree_SetNop} all have obvious ``$\mathsf{Sub}$-liftings'' and, thus, produce the (up to isomorphism) commutative triangle
\begin{center}
	$\xymatrix{\Games_\A\ar[r]^{\simeq\;}\ar[rd]_{\mathrm{I}} & \TrGame\ar@/_0.5pc/[d]_{\mathrm{inc}}^{\dashv}\\
		& \Sub{\Tree} \ar@/_0.5pc/[u]_{\mathrm{Pr}} \\
	}$
\end{center}
In particular, with Theorem
\ref{THM_GmeTree}, we obtain:

\begin{thm}\label{THM_Graphs&Games}
	The category $\Games_\A$ is equivalent to the full coreflective subcategory $\TrGame$ of $\Sub{\Tree}$.
\end{thm}

In fact, we are ready to state the easily verified claim that, by Corollary \ref{COR_TautLift}, the entire diagram of Corollary  \ref{COR_Tree_SetNop} has a ``$\mathsf{Sub}$-lifting"', as in

\begin{center}
	$\xymatrix{\Games_{\A}\ar[r]^{\simeq\;}\ar[rd]_{\mathrm{I}} & \mathbf{ArbGame}\ar@/_0.5pc/[d]^{\dashv}_{\mathrm{inc}}\ar@/^0.4pc/[rr]^{\mathrm{Fun}} && \mathbf{FunGame}\ar@/_0.5pc/[d]^{\dashv}_{\mathrm{inc}}\ar@/^0.4pc/[ll]_{\simeq}^{\mathrm{Tr}} & (T^*,A)\\
		&\Sub{ \Tree}\ar@/_0.5pc/[u]_{\mathrm{Pr}}\ar@/^0.4pc/[rr]^{\mathrm{Fun}} && \Sub{\SetNop}\ar@/_0.5pc/[u]\ar@/^0.4pc/[ll]_{\simeq}^{\mathrm{Tr}} & (T,A)\ar@{|->}[u] \\	
	}$
\end{center}

We therefore have:

\begin{thm}\label{THM_IsoSubSetsNop}
	The above diagram commutes in an obvious sense, up to natural isomorphisms. In particular, the category $\Sub{\Tree}$ is equivalent to $\Sub{\SetNop}$, and the category $\Games_\A$ is equivalent to $\mathbf{FunGame}$, which is a full coreflective subcategory of the topological category $\Sub{\SetNop}$ over the topos $\SetNop$.
\end{thm}

Let us finally point out that, of course, the diagram of Corollary \ref{COR_Branch}, in both the unpruned and the pruned form, also has a ``$\mathsf{Sub}$-lifting'', which then rests on the topological category
$$\Sub{\Sets} \coloneqq \mathsf{Sub}_{\mathrm{Id_{\Sets}}}(\Sets)$$
over $\Sets$, whose objects are pairs of sets $(X,S)$, with $S\subseteq X$, and whose morphisms $(X,S) \stackrel{f}{\to}(X',S')$ are mappings $f\colon X\to X'$ with $f[S]\subseteq S'$.
Now, 
considering the pruned form of the lifted diagram, and replacing its upper left category $\Sub{\PrTree}\cong\mathbf{ArbGame}$ equivalently by $\Games_\A$, for the record we may describe the  functors on the left side of the emerging (up to isomorphism) commutative diagram
\begin{center}
	$\xymatrix{\Games_\A\ar@/^0.4pc/[rr]^{\mathrm{}}_{\simeq}\ar@/_0.5pc/[ddr]_{\mathrm{Forget}}^{\bot} && \FunGame\ar@/^0.4pc/[ll]^{\mathrm{}}\ar@/^0.5pc/[ldd]^{\mathrm{Lim}}_{\bot}\\
		&&\\
		&\Sub{\Sets}\,,\ar@/_0.5pc/[luu]_{\mathrm{Free}}\ar@/^0.5pc/[uur]^{\Delta}&\\
	}$
\end{center}
as follows: the functor
$\mathrm{Forget}\colon\Games_\A\to\Sub{\Sets}$ sends

\begin{itemize}
	\item  a game $G=(T,A)$ to $(\Run(T),A)$, and
	\item an $\A$-morphism $(T,A)\stackrel{f}{\to} (T',A')$ to $\overline{f}\colon\Run(T)\to\Run(T')$;
\end{itemize}	
and the functor $\mathrm{Free}$ sends
\begin{itemize}
	\item a pair $(X,S)$ with $S\subseteq X$ to $( \bigcup_{n<\omega}\set{\seq{x: i<n}: x\in X}, \set{\seq{x: n<\omega}: x\in S})$ and \\
	\item extends a  morphism $f:(X,S)\to (X',S')$ by applying $f$ componentwise to sequences.
\end{itemize}

\subsection{Metric games}\label{SEC_MetGames}
In the following sections we will see how the resulting equivalent description of games in terms of ultrametrics helps us establish properties of the category $\Games_{\A}$.

Considering the commutative triangle
\begin{center}
	$\xymatrix{\Gmes\ar[rr]^{\mathrm{Run}}\ar[rd]_{\mathrm{Run}} && \CUMet\ar[ld]^{\mathrm{Forget}} \\
		& \Sets  \\
	}$
\end{center}
we may apply the ``$\mathsf{Sub}$-machinery''
developed in the previous section and, with Theorem \ref{THM_GmeUMet} and Corollary \ref{COR_TautLift},
obtain the adjunction
$$\xymatrix{\qquad\qquad\Games_{\A}\ar@/^1.0pc/[rr]^{\mathrm{Run}} &\bot& \mathsf{Sub}(\CUMet)\ar@/^1.0pc/[ll]^{\mathrm{Ball}}}.$$ 
Here $\mathrm{Run}$ sends an $\A$-game $(T,A)$ to $(\mathrm{Run}(T),d_T,A\subseteq\mathrm{Run}(T))$, and $\mathrm{Ball}$ sends an arbitrary object $(X,d,S)$ in $\mathsf{Sub}(\CUMet)=\mathsf{Sub}_{\mathrm{Forget}}(\CUMet) $ to $(\mathrm{Ball}(X,d),A_S)$ with $A_S=\varepsilon_{(X,d)}^{-1}[S]$, so that $A_S$ contains precisely those runs $R=\seq{\overline{B}\left(x,\frac{1}{n+2}\right):n<\omega} $ in $\mathrm{Ball}(X,d)$ for which the unique center $x$ lies in $S$.
Clearly, if we now put 
\begin{itemize}
	\item $\mathbf{MetGame}:=\mathsf{Sub}_{\mathrm{Forget}_{|\mathbf{SeqSpa}}}(\CUMet)$,
\end{itemize}
applying Theorem \ref{THM_GmeUMet} and Corollary \ref{COR_TautLift} once again, we obtain that $\mathbf{MetGame}$ gives us an appropriate metric description of games, as follows:

\begin{thm}\label{THM_UMetEqv}
	The category $\Games_\A$ is equivalent to $\MetGame$, which is a full coreflective subcategory of $\Sub{\CUMet}$.
\end{thm}

We conclude this section similarly to the previous: with the diagrams of Figure \ref{diagram_tikzcd}, we summarize the results showed in this section (following the same notation explained for Figure \ref{diagram_Gme_tikzcd}). We emphasize agin that the diagram commute only up to isomorphism, but  the downward composition of functor from $\Games_\A$ to $\Sets$ in \ref{diagram_tikzcd} on the left side commutes strictly with the functor $\Run$.

\begin{figure}
	\begin{tikzcd}[row sep=5em, column sep=6em]
		&  {\Games}_{\A}  \arrow[d,leftrightarrow]\arrow[dr,leftrightarrow]\arrow[dl,leftrightarrow]	& 	 \\ 
		\MetGame\arrow[r,leftrightarrow] \arrow[d,hook,"\dashv"{auto,xshift=0.5ex}]\arrow[d, leftarrow, bend left, shift left=2, "\mathrm{Run}\circ \mathrm{Ball}"] & \TrGame\arrow[r, leftrightarrow] \arrow[d,hook,"\dashv"{auto,xshift=0.5ex}]\arrow[d, leftarrow, bend left, shift left=2, "\Pr"]	&    \FunGame \arrow[d,hook,"\dashv"{auto,xshift=0.5ex}]\arrow[d, leftarrow, bend left, shift left=2, "(-)^*"]\\
		\Sub{\CUMet} \arrow[d, leftarrow, bend right, shift right=2, "\dashv"{xshift=0.5ex}]\arrow[d, leftarrow, bend left, shift left=2] \arrow[d, "\dashv"{xshift=0.5ex}]	 & \Sub{\Tree} \arrow[r,leftrightarrow] \arrow[d, bend right, leftarrow, shift right=2, "\dashv"{xshift=0.5ex}]\arrow[d, leftarrow, bend left, shift left=2] \arrow[d, "\dashv"{xshift=0.5ex}]  &    \Sub{\SetNop} \arrow[d, bend right, leftarrow, shift right=2, "\dashv"{xshift=0.5ex}]\arrow[d, bend left, leftarrow, shift left=2] \arrow[d, "\dashv"{xshift=0.5ex}]\\
		\CUMet  \arrow[dr, "\mathrm{Forget}"description] & \Tree \arrow[r,leftrightarrow] \arrow[d,  "\Br"]&    \SetNop  \arrow[dl, "\mathrm{Lim}"description]\\
		&  \Sets \arrow[ul, bend left, "\tang"{sloped, auto,yshift=0.75ex}, "\msmall{\mathrm{Discrete}}"{sloped, auto, swap}] \arrow[ur, bend right, "\tang"{sloped, auto,yshift=0.75ex}, "\msmall{\mathrm{\Delta}}"{sloped, auto, swap}] \arrow[u,  bend left, shift left=2, "\msmall{\mathrm{Free}}"{auto}, , "\dashv"{auto,swap,yshift=0.0ex, xshift=0.5ex}]
	\end{tikzcd}
	\caption{Equivalent categories of games and their underlying categories}\label{diagram_tikzcd}	
\end{figure}

\section{Topological games as functors}\label{SEC_TopGames}
It is high time for us to point out a trivial fact:
\begin{prop}\label{PROP_EquivCat}
	The categories $\Games_\A$ and $\Games_\B$ are isomorphic.
\end{prop}
\begin{proof}
	The functor $\Games_\A \stackrel{}{\to}\Games_\B$  sending the game $(T,A)$	to $(T,\Run(T)\setminus A)$ makes every A-morphism a B-morphism and is trivially inverse to itself.	
\end{proof}

Despite being isomorphic, the distinction of our two game categories is useful. Indeed, we will see next that some topological games are naturally described as functors over only one of the two categories.

\begin{ex}
	Let $\CTopp$ denote the category of pointed topological spaces. Hence, an object $(X,x)$ is a topological space $X$ with a distinguished point $x$, and a morphism $f:(X,x)\to(Y,y)$ is a continuous map with $f(x)=y$. 
	The tightness games naturally lead to the following functors:
	\begin{description}
		\item[$\g_1(\W_x,\W_x)$] The functor \begin{tikzcd}[column sep = 2em]
			\CTopp \arrow[r, "\Tight^\W_1"] &	\Games_\B
		\end{tikzcd} sends
		\begin{itemize}
			\item an object $(X,x)$ to the game $\g_1(\W_x,\W_x)$ over $X$, and
			\item a morphism $f\colon (X,x)\to (Y,y)$ to the $\B$-morphism 
			\begin{center}
				\begin{tikzcd}[row sep = 0em]
					\Tight^\W_1(X,x) \arrow[r, "\Tight^\W_1(f)"] &	\Tight^\W_1(Y,y) \\
					\seq{A_0, a_0, \dotsc, A_n, a_n}  \arrow[r, mapsto] &   \seq{f[A_0], f(a_0), \dotsc, f[A_n], f(a_n)}.
				\end{tikzcd}
			\end{center}
			(Indeed, for every $n<\omega$, $x\in\overline{A_n}$ implies $f(x)=y\in\overline{f[A_n]}$ and, for the same reason, $x\in \overline{\set{a_n:n\ge k}}$ implies $f(x)=y\in\overline{\set{f(a_n):n\ge k}}$ for all $k<\omega$.)
		\end{itemize} 
		\item[$\g_1(\W_x,\Gamma_x)$] The functor \begin{tikzcd}[column sep = 2em]
			\CTopp \arrow[r, "\Tight^\Gamma_1"] &	\Games_\B
		\end{tikzcd} sends
		\begin{itemize}
			\item an object $(X,x)$ to $\g_1(\W_x,\Gamma_x)$ over $X$, and 
			\item a morphism $f\colon (X,x)\to (Y,y)$ to the $\B$-morphism 
			\begin{center}
				\begin{tikzcd}[row sep = 0em]
					\Tight^\Gamma_1(X,x) \arrow[r, "\Tight^\Gamma_1(f)"] &	\Tight^\Gamma_1(Y,y) \\
					\seq{A_0, a_0, \dotsc, A_n, a_n}  \arrow[r, mapsto] &   \seq{f[A_0], f(a_0), \dotsc, f[A_n], f(a_n)}.
				\end{tikzcd}
			\end{center}
		\end{itemize} 
	\end{description}
\end{ex}

\begin{ex}
	Also the covering games can naturally be seen as (contravariant) functors:
	\begin{itemize}
		\item $\g_1(\W,\W)$: Here, \begin{tikzcd}[column sep = 2em]
			\CTop^{\op} \arrow[r, "\Cover^\W_1"] &	\Games_\B
		\end{tikzcd} is defined
		\begin{itemize}
			\item on objects by $\Cover^\W_1X=\g_1(\W,\W)$ over $X$, and
			\item on morphisms $f\colon X\to Y$ by
			\begin{center}
				\begin{tikzcd}[row sep = 0em]
					\Cover^\W_1Y \arrow[r, "\Cover^\W_1f"] &	\Cover^\W_1X \\
					\seq{\mU_0, U_0, \dotsc, \mU_n, U_n}  \arrow[r, mapsto] &   \seq{f^{-1}[\mU_0], f^{-1}(U_0), \dotsc, f^{-1}[\mU_n], f^{-1}(U_n)},
				\end{tikzcd}
			\end{center}
			where $f^{-1}[\mU_k]=\set{f^{-1}(U):U\in \mU_k}$. (Indeed, if  $n<\omega$ and $F\subseteq X$ finite, for the finite set $f[F]\subseteq Y$ one has some $U\in \mU_n$ with $f[F]\subseteq U$ and, hence, $F\subseteq f^{-1}(U)\in f^{-1}[\mU_n]$. By the same argument it is clear that, if $\set{U_n:n\ge k}$ is an $\omega$-cover, then $\set{f^{-1}(U_n):n\ge k}$ is also an $\omega$-cover.)
		\end{itemize} 
		\item $\g_1(\W,\Gamma)$: The functor \begin{tikzcd}[column sep = 2em]
			\CTop^{\op} \arrow[r, "\Cover^\Gamma_1"] &	\Games_\B
		\end{tikzcd} is defined
		\begin{itemize}
			\item on objects by $\Cover^\Gamma_1X=\g_1(\W,\W)$ over $X$. 
			\item on morphisms $f\colon X\to Y$ by				
			\begin{center}
				\begin{tikzcd}[row sep = 0em]
					\Cover^\Gamma_1Y \arrow[r, "\Cover^\Gamma_1f"] &	\Cover^\Gamma_1X \\
					\seq{\mU_0, U_0, \dotsc, \mU_n, U_n}  \arrow[r, mapsto] &   \seq{f^{-1}[\mU_0], f^{-1}(U_0), \dotsc, f^{-1}[\mU_n], f^{-1}(U_n)}.
				\end{tikzcd}
			\end{center}
			(Indeed, if $n<\omega$ and $F\subseteq X$ finite, for the finite set $f[F]\subseteq Y$ one finds $U\in \mU_n$ with $f[F]\subseteq U$ and, hence, $F\subseteq f^{-1}(U)\in f^{-1}[\mU_n]$. Similarly, given $S\subseteq \omega$ infinite and an $\omega$-cover $\set{U_n:n\in S}$ is an $\omega$, also $\set{f^{-1}(U_n):n\in S}$ is an $\omega$-cover.)
		\end{itemize} 
	\end{itemize}
\end{ex}

We now proceed to provide some new categorical depth to Theorems \ref{THM_ClassOmega} and \ref{THM_ClassGamma} as stated in the Introduction, in each case facilitated by the fact that there are two natural transformations linking some of the functors just presented. To this end, we also need the auxiliary ``lifted'' hom-functor 
\[
{\Cp}_*:\CTop^{\op}\to \CTopp	
\]	
\begin{itemize}
	\item sending a space $X$ to the function space $\Cp(X)$ with its topology of pointwise convergence, pointed by the function  $\bar{0}:X\to\RR$ with constant value $0$, \textit{i.e.}, {${\Cp}_*X = (\Cp(X), \bar{0})$}, and
	\item a continuous mapping $f\colon X\to Y$ to the $\CTopp$-morphism
	\begin{center}
		\begin{tikzcd}[row sep = 0em]
			(\Cp(Y), \bar{0}) \arrow[r, "{\Cp}_*f"] &	(\Cp(X), \bar{0}) \\
			\varphi  \arrow[r, mapsto] &   \varphi\circ f.
		\end{tikzcd}
	\end{center}
\end{itemize}

Each of the proofs of the two logical directions of the equivalence statement in Theorem  \ref{THM_ClassOmega} will be based on a proposition establishing the required natural transformation, followed by two lemmata. The proof of Theorem \ref{THM_ClassGamma} follows the same scheme, so that should suffice for us to give just an adaptation of a shortened replica of the proof of Theorem \ref{THM_ClassOmega}. 

We  recall some aspects of the original proof of Theorem \ref{THM_ClassOmega} in \cite{Scheepers1997}, as follows.
For every space $X$ and $A\subseteq \Cp(X)$, with $I_0$  denoting the open interval $(-1,1)$, one considers
\[
\mU_0(A)=\set{\varphi^{-1}\left(I_0\right): \varphi\in A}.
\]
(Here we carry the index $0$ only in order to stay consistent with a notation used lateron.)
Trivially, if $X=\emptyset$, then $\mU_0(A)$ is an $\omega$-cover.  If $X\neq\emptyset$ and $\bar{0}\in \overline{A}$, then $\mU_0(A)$ is an $\omega$-cover as well. Indeed, for a finite set $F\subseteq X$, since $\bar{0}\in \overline{A}$, there is some $\varphi\in A$ such that $\varphi(x)\in I_0$ for every $x\in F$, so $F\subseteq \varphi^{-1}\left(I_0\right)\in \mU_0(A)$.

\begin{prop}\label{PROP_theta}
	For every space $X$ one has the $\B$-morphism	
	\begin{center}
		\begin{tikzcd}[row sep = 0em]
			\Tight^\Omega_1( {\Cp}_*X )\arrow[r, "\theta_X"] &	\Cover^\Omega_1X \\
			\seq{A_0, \varphi_0, \dotsc, A_n, \varphi_n}  \arrow[r, mapsto] &   \seq{\mU_0(A_0), \varphi_0^{-1}(I_0), \dotsc,\mU_0(A_n), \varphi_n^{-1}(I_0)},
		\end{tikzcd}
	\end{center}
	defining a natural transformation $\theta\colon \Tight^\Omega_1\circ {\Cp}_* \longrightarrow \Cover^\Omega_1$.
	
	\[
	\xymatrix{ \CTop^{\op}\ar[rr]^{{\Cp}_*}_\theta\ar[rd]_{\Cover^\Omega_1}^{\qquad\;\;\,\Longleftarrow}			 && \CTopp\ar[ld]^{\Tight^\Omega_1}		\\
		& \Games_\B & \\
	}
	\]
\end{prop}
\begin{proof}
	To see that $\theta_X$ is a $\B$-morphism, we first note that $\theta_X$ is clearly chronological. Suppose now that $\seq{A_0, \varphi_0, \dotsc, A_n, \varphi_n,\dotsc}$ is a run in $\Tight^\Omega_1({\Cp}_*X)$ in which $\bob$ wins. Showing that $\bob$ wins the run $\seq{\mU_0(A_0), \varphi_0^{-1}(I_0), \dotsc,\mU_0(A_n), \varphi_n^{-1}(I_0)}$ is then analogous to the proof that $\mU_0(A)$ is an $\omega$-cover if $\overline{0}\in \overline{A}$.
	
	In order to confirm that $\theta$ is a natural transformation, for every continuous map $f:X\to Y$ we need to show the commutativity of the diagram
	\[
	\xymatrix{\Tight^\Omega_1({\Cp}_*Y)\ar[rr]^{\Tight^\Omega_1({\Cp}_*f)}\ar[d]_{\theta_Y}	 && \Tight^\Omega_1({\Cp}_*X)\ar[d]^{\theta_X} \\
		\Cover^\Omega_1Y\ar[rr]^{\Cover^\Omega_1f}	&& \Cover^\Omega_1X\,.	\\	
	}
	\]	
	Here the mapping $\Tight^\Omega_1({\Cp}_*f)$ sends every sequence $\seq{A_0, \varphi_0, \dotsc, A_n, \varphi_n}$ of the domain to the sequence
	$\seq{A_0\circ f, \varphi_0\circ f, \dotsc, A_n\circ f, \varphi_n\circ f}$, with the abbreviation $A\circ f = \set{\varphi\circ f:\varphi\in A}$.
	Therefore, 
	\begin{align*}
		& \theta_X(\Tight^\Omega_1( \Cp f(\seq{A_0, \varphi_0, \dotsc, A_n, \varphi_n}))\\
		=\; & \seq{\mU_0(A_0\circ f), (\varphi_0\circ f)^{-1}(I_0), \dotsc,\mU_0(A_n\circ f), (\varphi_n\circ f)^{-1}(I_0)}  \\
		=\; & \seq{\mU_0(A_0\circ f), f^{-1}(\varphi_0^{-1}(I_0)), \dotsc,\mU_0(A_n\circ f), f^{-1}(\varphi_n^{-1}(I_0))}\\
		=\; & \Cover^\Omega_1f(\theta_Y(\seq{A_0, \varphi_0, \dotsc, A_n, \varphi_n})).\\
	\end{align*}	
\end{proof}

\begin{lemma}\label{LEMMA_thetaLocalSurj}
	For every space $X$, the codomain restriction of $\Tight^\Omega_1({\Cp}_*X)\stackrel{\theta_X}{\to} \Cover^\W_1X$ to its image is locally surjective.
\end{lemma}
\begin{proof}
	The local surjectivity at a moment $t$ in $\Tight^\Omega_1({\Cp}_*X)$ that is $\bob$'s turn should be clear; in fact, in this case we even have  local surjectivity of the unrestricted mapping $\theta_X$. 
	So, suppose $t$ is a moment in $\Tight^\Omega_1({\Cp}_*X)$ that is $\ali$'s turn, \textit{i.e.}, $t=\seq{A_0, \varphi_0, \dotsc, A_n, \varphi_n}$. Then $\theta_X(t)^\smallfrown y$ is a moment in the image of $\theta_X$ if, and only if, there is an $A_{n+1}\subseteq \Cp(X)$ such that $\bar{0}\in \overline{A_{n+1}}$ and $y=\mU_0(A_{n+1})$. But note that $t^\smallfrown A_{n+1}$ is a moment in $\Tight^\Omega_1({\Cp}_*X)$ and $\theta_X(t^\smallfrown A_{n+1})=\theta_X(t)^\smallfrown \mU_0(A_{n+1}) = \theta_X(t)^\smallfrown y$, which concludes the proof.
\end{proof}

Recall that two games $G,G'$ are \emph{equivalent} if $\ali$ has a winning strategy for $G$ precisely when $\ali$ has a winning stragey for $G'$, and likewise for $\bob$.

\begin{lemma}\label{LEMMA_thetaImgEqv}
	For every $T_{3\frac{1}{2}}$-space $X$, the image of $\theta_X$ is equivalent to $\Cover^\Omega_1X$. 
\end{lemma}
\begin{proof}
	It suffices to show that, for every $\omega$-cover $\mW$, there is a refinement $\mU$ of $\mW$ of the form $\mU=\mU_0(A)$, for some $A\subseteq \Cp(X)$ with $\bar{0}\in \overline{A}$. Then the equivalence follows from a standard argument for covering games.
	
	Indeed, given $\mW$, for every finite set $F\subseteq X$ we can choose a set $W_F\in \mW$ with $F\subseteq W_F$. Then, since $X$ is  a $T_{3\frac{1}{2}}$-space,		there is a continuous function $\varphi_F\colon X \to [0,1]$, such that $\varphi_F[F]=\{0\}$ and $\varphi_F[X\setminus W_F]=\{1\}$.
	Now let
	\[
	A(\mW)=\set{\varphi_F: F\subseteq X \text{ finite}}.
	\]
	Then we clearly have 
	$\bar{0}\in \overline{A(\mW)}$. Moreover, for every finite set $F\subseteq X$, the set $\varphi_F^{-1}(I_0)$ is contained in $W_F$, so that $\mU_0(A(\mW))$ refines $\mW$ and serves as the desired $\omega$-cover.
\end{proof}

The previous two lemmata in conjunction with Corollaries \ref{LEMMA_StratImage}(b) and \ref{LEMMA_StratPreimage}(a) prove the first half of Theorem  \ref{THM_ClassOmega}:	
\begin{cor}
	Let $X$ be a $T_{3\frac{1}{2}}$-space. Then:
	\begin{itemize}
		\item If $\ali$ wins $\g_1(\W,\W)$ over $X$, then $\ali$ wins also $ \g_1(\W_{\bar{0}},\W_{\bar{0}})$ over $\Cp(X)$.
		\item If $\bob$ wins $\g_1(\W_{\bar{0}},\W_{\bar{0}})$ over $\Cp(X)$, then $\bob$ wins also $\g_1(\W,\W)$ over $X$.
	\end{itemize}
\end{cor}

To obtain the other half of Theorem  \ref{THM_ClassOmega} in a similar fashion, we modify the definition of $\theta_X$ and consider the mapping

\begin{center}
	\begin{tikzcd}[row sep = 0em]
		\Tight^\Omega_1({\Cp}_*X)\arrow[r, "\eta_X"] &	{\Cover^\W_1X} \\
		\seq{A_0, \varphi_0, \dotsc, A_n, \varphi_n} \arrow[r, mapsto] &   \seq{\mU_0(A_0), \varphi_0^{-1}(I_0), \dotsc,\mU_n(A_n), \varphi_n^{-1}(I_n)}.
	\end{tikzcd}
\end{center}
Here, for a space $X$ and $A\subseteq \Cp(X)$, consistently with the earlier notations $I_0$ and  $\mU_0(A)$, we have put $I_n=(\frac{-1}{n+1},\frac{1}{n+1})$ and 
\[
\mU_n(A)=\set{\varphi^{-1}\left(I_n\right): \varphi\in A.}
\]
The proof that the mapping $\eta_X$ is well-defined is completely analogous to that of $\theta_X$: 
whenever $\bar{0}\in\overline{A}$, then $\mU_n(A)$ is an $\omega$-cover for every $n<\omega$.

Furthermore, the proof that the naturality diagram of a continuous map $f:X\to Y$ for $\eta_X$ is commutative proceeds just like the one for $\theta_X$. However, as we will show next, unlike $\theta_X$, the mapping $\eta_X$ is an $\A$-morphism, but generally fails to be a $\B$-morphism, \textit{i.e.}, generally $\eta_X$ fails to live in the target category $\Games_\B$ of the functors $\Tight_1^{\Omega}$ and $\Cover_1^{\Omega}$ and therefore does {\em not} lead to a natural transformation $\Tight^\Omega_1\circ {\Cp}_*\longrightarrow\Cover^\Omega_1$.
Still, employing the forgetful functor $\mathrm{U}\colon \Games_{\B}\to \Gmes$	, which disregards the payoff set of a game, we can state:

\begin{prop}
	For every space $X$, the mapping $\eta_X$ is an $\A$-morphism, defining the natural transformation $\eta\colon \mathrm{U}\circ\Tight^\Omega_1\circ {\Cp}_* \longrightarrow \mathrm{U}\circ\Cover^\Omega_1$.
	
	\[
	\xymatrix{ \CTop^{\op}\ar[rr]^{{\Cp}_*}_\eta\ar[rd]_{\mathrm{U}\circ\Cover^\Omega_1}^{\qquad\;\;\,\Longleftarrow}			 && \CTopp\ar[ld]^{\mathrm{U}\circ\Tight^\Omega_1}		\\
		& \Gmes & \\
	}
	\]
	
\end{prop}
\begin{proof}
	As indicated above, it remains to be shown that $\eta_X$ is an $\A$-morphism, for every space $X$. Hence, consider a run $\seq{A_0, \varphi_0, \dotsc, A_n, \varphi_n,\dotsc}$ in $\Tight^\Omega_1({\Cp}_*X)$ in which $\ali$ wins. Then there is some $k<\omega$, an $\varepsilon>0$ and a finite set $F\subseteq X$ such that, for every $n\ge k$, $\varphi_n[F]\not\subseteq]-\varepsilon,\varepsilon[$. Let $N\ge k$ be such that $\frac{1}{N+1}<\varepsilon$. Then we also have that $\varphi_n[F]\not\subseteq I_n$ for every $n\ge N$ and, hence, that $F\not\subseteq \varphi_n^{-1}(I_n)$ for every $n\ge N$. 
\end{proof}

As we indicate next, there are plenty of examples of spaces $X$ for which $\theta_X$ fails to be an $\A$-morphism and $\eta_X$ fails to be a $\B$-morphism. Once again, this underlines the necessity of distinguishing between the two types of game morphisms, despite the fact that the respective categories are isomorphic.

\begin{ex}
	For $\theta_X$ failing to be an $\A$-morphism, we may consider any non-empty space $X$. We just have to note that, as the preimages are always  taken of $I_0$, we can deduce from $\bob$ winning a run $\seq{\mU_0(A_0), \varphi_0^{-1}(I_0), \dotsc, \mU_0(A_n),\varphi_n^{-1}(I_0), \dotsc}$ in $\g_1(\W,\W)$ over $X$ only that the condition for $\bar{0}\in \overline{\set{\varphi_n:n<\omega}}$ is met for $\varepsilon=1$. This means that  the functions $\varphi_n$ with constant value $ \frac{1}{2}$ for all $n<\omega$ provide a counter-example, regardless of what the non-empty space $X$ may be.
	
	As for $\eta_X$ failing to be a $\B$-morphism, the singleton space $X=\{0\}=1$ provides	 a simple counterexample. Since  $\Cp(X)\cong\RR$, we can consider $\varphi_n(0)= \frac{2}{n+1}$ for every $n<\omega$, so that $\varphi_n^{-1}(I_n)=\emptyset$, despite $\bar{0}\in \overline{\set{\varphi_n:n<\omega}}$.   
\end{ex}

The proofs of the following two statements are analogous to those of Lemma \ref{LEMMA_thetaLocalSurj} and \ref{LEMMA_thetaImgEqv}.

\begin{lemma}
	For every space $X$, the codomain restriction of $\Tight^\Omega_1({\Cp}_*X)\stackrel{\eta_X}{\to} {\Cover^\W_1}X$ to its image is locally surjective.
\end{lemma}

\begin{lemma}
	For every $T_{3\frac{1}{2}}$-space $X$, the image of $\eta_X$ is equivalent to ${\Cover^\Omega_1}X$.
\end{lemma}

Hence, with these two lemmata and Corollaries \ref{LEMMA_StratImage}(a) and \ref{LEMMA_StratPreimage}(b) we have completed the proof of the second half of  Theorem \ref{THM_ClassOmega}, which is: 	
\begin{cor}
	Suppose $X$ is a $T_{3\frac{1}{2}}$-space. Then:
	\begin{itemize}
		\item If $\bob$ wins $\g_1(\W,\W)$ over $X$, then $\bob$ wins also $\g_1(\W_{\bar{0}},\W_{\bar{0}})$ over $\Cp(X)$.
		\item If $\ali$ wins $\g_1(\W_{\bar{0}},\W_{\bar{0}})$ over $\Cp(X)$, then $\ali $ wins also $\g_1(\W,\W)$ over $X$.
	\end{itemize}
\end{cor}

As for Theorem \ref{THM_ClassGamma}, we consider again the natural mappings $\theta_X$ and $\eta_X$, but change the superscript $\Omega$ to $\Gamma$ in both, the domain and codomain, thus obtaining the mappings
\begin{center}
	\begin{tikzcd}[row sep = 0em]
		\Tight^\Gamma_1({\Cp}_*X) \arrow[r, "\tilde{\theta}_X"] &	\Cover^\Gamma_1X \\
		\seq{A_0, \varphi_0, \dotsc, A_n, \varphi_n} \arrow[r, mapsto] &   \seq{\mU_0(A_0), \varphi_0^{-1}(I_0), \dotsc,\mU_0(A_n), \varphi_n^{-1}(I_0)},
	\end{tikzcd}
\end{center}
\begin{center}
	\begin{tikzcd}[row sep = 0em]
		\Tight^\Gamma_1( {\Cp}_*X) \arrow[r, "\tilde{\eta}_X"] &	{\Cover^\Gamma_1}X \\
		\seq{A_0, \varphi_0, \dotsc, A_n, \varphi_n} \arrow[r, mapsto] &   \seq{\mU_0(A_0), \varphi_0^{-1}(I_0), \dotsc,\mU_n(A_n), \varphi_n^{-1}(I_n)}.
	\end{tikzcd}
\end{center}
Now, all proofs we have given in the ``$\W$-case'' remain intact also in the ``$\Gamma$-case'' (that is: when we trade $\theta_X$ and $\eta_X$ for $\tilde{\theta}_X$ and $\tilde{\eta}_X$, respectively), {\em except} for the proofs that $\theta_X$ is a $\B$-morphism and $\eta_X$ is an $\A$-morphism. Hence, let us state and prove these explicitly:

\begin{prop}
	For every space $X$, the mapping $\tilde{\theta}_X$ is a $\B$-morphism, and the mapping $\tilde{\eta}_X$ is an $\A$-morphism.
\end{prop}
\begin{proof}
	First let $\seq{A_0, \varphi_0, \dotsc, A_n, \varphi_n,\dotsc}$ be a run in $\Tight^\Gamma_1( {\Cp}_*X)$ in which $\bob$ wins, and let $S\subseteq\omega$ be infinite and $F\subseteq X$ be finite. We fix an increasing enumeration $S=\set{n_k:k<\omega}$ and have $\lim_{k\to\infty}\varphi_{n_k}=\bar{0}$. So, by taking $\varepsilon=1$ we obtain some $K<\omega$ with $\varphi_{n_K}[F]\subseteq I_0$ for every $k\le K$ and, hence, $F\subseteq \varphi_{n_K}^{-1}(I_0)$. Therefore $\seq{\mU_0(A_0), \varphi_0^{-1}(I_0), \dotsc,\mU_0(A_n), \varphi_n^{-1}(I_0),\dotsc}$ in $\Cover^\Gamma_1X$ is a winning run for $\bob$, which proves that $\theta_X$ is a $\B$-morphism. 
	
	Now suppose $\seq{A_0, \varphi_0, \dotsc, A_n, \varphi_n,\dotsc}$ is a run in $\Tight^\Gamma_1( {\Cp}_*X)$ in which $\ali$ wins. Then there is an $\varepsilon>0$ and $x\in X$ such that, for every $k<\omega$, there is an $n_k\ge k$ with $\varphi_{n_k}(x)\not\in ]-\varepsilon,\varepsilon[$. In particular, given $K<\omega$ with $\frac{1}{n_K+1}<\varepsilon$, one has $\varphi_{n_k}(x)\not\in I_{n_k}$ for all $k\ge K$. So, $S=\set{n_k:k\ge K}$ attests that $\set{\varphi_n^{-1}(I_n):n<\omega}$ is not a $\gamma$-cover and, hence, that $\ali$ wins the run $\seq{\mU_0(A_0), \varphi_0^{-1}(I_0), \dotsc,\mU_n(A_n), \varphi_n^{-1}(I_n),\dotsc}$ in $\Cover^\Gamma_1(X)$. This shows $\eta_X$ is an $\A$-morphism. 
\end{proof}	

Now the proof of Theorem	 \ref{THM_ClassGamma} may be finished exactly as our proof of Theorem \ref	{THM_ClassOmega}.

\section{The Universality of the Banach-Mazur game}\label{SEC_BMUniversal}

Many classical results concern the existence of ``universal'' objects in a given category. For example, since by Cayley's Theorem every group is isomorphic to 
a subgroup of a symmetric group, symmetric groups may be considered universal in the category of groups. In this section we show that the Banach-Mazur games are universal in this sense in the category of games. More precisely:

\begin{thm}[Universality of the Banach-Mazur Game]\label{THM_BMUniversal}
	For every game $G$ there is an ultrametric space $\tilde{\K}(G)$ such that $G$ is isomorphic to a subgame of the Banach-Mazur game over $\tilde{\K}(G)$.
\end{thm} 

In order to prove Theorem \ref{THM_BMUniversal}, let us first consider the following functorial construction:

\begin{defn}
	Let ${\D}_{\B}\colon \Games_{\B}\to \Games_{\B}$ be the functor such that,
	\begin{itemize}
		\item on objects, for a game $G = (T,A)$, with some element $0_G\not\in \M(G)$ define $\D_{\B}G= (\D_{\B} T,A)$ by
		\begin{equation}\label{eq_BMUni}
			\D_\B T = \begin{cases}
				T\cup \set{t^\smallfrown \seq{0_G:i\le n}: t\in T \text{, $t$ is $\ali$'s turn and $n<\omega$} }, \text{ if $T\neq\emptyset$}\\
				\set{\seq{0_G:i\le n}: n<\omega }, \text{ otherwise.}
			\end{cases}
		\end{equation}
		\item on morphisms, for a $\B$-morphism $(T,A)\stackrel{f}{\to}(T',A')$, let $(\D_{\B}T,A)\stackrel{\D_{\B}f}{\to}(\D_{\B}T',A')$ be such that
		\[
		(\D_{\B}f)(t) = \begin{cases}
			f(t), \text{ if $t\in T$},\\
			f(s)^\smallfrown \seq{0_{G'}:i\le n}, \text{ if $t=s^\smallfrown \seq{0_{G}:i\le n}$ for some $s\in T$.}
		\end{cases}
		\]
	\end{itemize}
\end{defn}

Intuitively, given a non-empty game $G$, $\D_{\B}G$ is the game that is identical to $G$, except that we additionally give $\ali$ the option to ``quit'' the game at every moment $t$ that is her turn. (Note that the runs of $\D_{\B}G$ in which $\ali$ wins are the same as in $G$, so that $\ali$ loses the run $t^\smallfrown \seq{0_G:n<\omega}$ for every $t$ that is her turn - and once she chooses $0_G$, both players have no choice but to keep choosing $0_G$, so that $\bob$ is determined to win from that moment onward.) The notation ``$\D_{\B}$''  refers to the fact that $\D_{\B}G$ makes sure that the runs in which $\bob$ wins are \textit{dense} in the space $(\Run(\D_{\B}T),d_{D_{\B}T})$.

From a game-theoretical point of view, $\D_{\B}G$ is clearly equivalent to $G$ (by construction, the existence of winning strategies for any of the players is preserved from $G$ to $\D_{\B}G$, and vice-versa) -- in spite of this, its importance will soon become clear.

We should note that $\D_{\B}$ could be defined in the same way as a functor from $\Games_{\A}$ to $\Games_{\A}$, however our search for a candidate for $\tilde{\K}(G)$ in Theorem \ref{THM_BMUniversal} leads us naturally to using $\Games_{\B}$, as follows.

\begin{defn}\label{DEF_KromSpace}
	The \textit{Krom space} $\mathsf KG$  of a game $G=(T,A)$ is the subspace $\Run(T)\setminus A$ of the ultrametric space $(\Run(T),d_T)$.
	
	With  $\mathbf{UltMet_1}$ denoting the category of ultrametric spaces of diameter at most $1$ and their $1$-Lipschitz mappings, the object assignment $G\mapsto \mathsf KG$ gives a functor 
	$$\K\colon \Games_{\B}\to \mathbf{UltMet_1}$$
	which sends a $\B$-morphism $(T,A)\stackrel{f}{\to}_\B (T',A')$ to the restriction $\Run(T)\setminus A\to \Run(T')\setminus A'$ of the map $\overline{f}$.
\end{defn}

Definition \ref{DEF_KromSpace} of $\K(G)$ is inspired by the work in \cite{Krom1974} which, for every topological space $X$, constructs an ultrametric space, now known as the \textit{Krom space of $X$}. In fact, one easily shows that such space is \textit{equal} (not just isometric) to $\K(G)$ when $G$ is the Banach-Mazur game over the space $X$, so that Definition \ref{DEF_KromSpace} is a straight generalization of its original.

We consider the composite functor  $\tilde{\K}=\K\circ\D_{\B}\colon \Games_{\B}\to \mathbf{UltMet_1}$ and note that, for a game $G=(T,A)$, the set $\set{[t]:t\in T}$ with $[t] := \set{R\in \Run(T)\setminus A: R\text{ restricts to } t}$ is a topological basis for the space $(\Run(\D_{\B}T),d_{\D_{\B}T})$ and, hence, for $\tilde{\K} G$ as well. 

Consequently, a promising candidate for an injective $\A$- and $\B$-morphism  $G\to\BM(\tilde{\K}G)$ is the mapping

\begin{center}
	\begin{tikzcd}[row sep = 0em]
		T \arrow[r, "\eta_G"] & \Run(\D_{\B}T) \\
		t  \arrow[r, mapsto] &   \seq{[t\upharpoonright 1],  \dotsc, [t]}.
	\end{tikzcd}
\end{center}
We note that taking $\D_{\B}G$ before taking the Krom space is crucial here. Otherwise, $[t]\cap \Run(T)\setminus A$ might be empty, hence an invalid move in the Banach-Mazur game, while this is certainly not the case in $\Run(\D_{\B}T)\setminus A$, which is a dense subspace of $(\Run(\D_{\B}T), d_{\D_{\B}T})$. 

We proceed with:

\begin{myproof}{Theorem}{THM_BMUniversal}
	Since  $\eta_G$ is an injective chronological mapping, 
	by item (e) of Proposition \ref{PROP_morphproperties} it suffices to show that $\eta_G$ is both an $\A$- and $\B$-morphism.
	Indeed, given a run $R\in \Run(T)$, one has that $\overline{\eta_G}(R) = \seq{[R\upharpoonright n]:n>0}$ and
	\[
	\bigcap_{n<\omega}[R\upharpoonright n]=\{R\},
	\]
	so that the intersection $\bigcap_{n>0}[R\upharpoonright n]$ taken in $\tilde{\K}G$ is non-empty if, and only if, $R\not\in A$. Therefore, $\ali$ wins the run $\overline{\eta_G}(R)$ in $\BM(\tilde{\K}G)$ \emph{if, and only if,} $R\in A$, as required.
\end{myproof}

We end this section with an important observation: the Banach-Mazur game leads to a functor $\BM\colon \CTop_{\mathrm{open}}\to \Games_{\B}$, where $\CTop_{\mathrm{open}}$ is the category of topological spaces with open mappings as morphisms. 
\begin{itemize}
	\item On objects, for a space $X$, let $\BM X$ be the Banach-Mazur game played over $X$. (If $X$ is empty, also $\BM X$ is taken to be empty: $\BM \emptyset=(\emptyset,\emptyset)$.)
	\item On morphisms, given an open mapping $X\stackrel{f}{\to}Y$, let \begin{tikzcd}[column sep = 2em]
		\BM X \arrow[r, "\BM f"] &	\BM Y
	\end{tikzcd} be such that 
	\[(\BM f)(\seq{U_0, V_0, \dotsc, U_n, V_n}) = \seq{f[U_0], f[V_0], \dotsc, f[U_n], f[V_n]}.\]
	(For $X$ empty, $\BM f$ is necessarily the unique $\B$-morphism from $(\emptyset,\emptyset)$ to $\BM Y$).
\end{itemize} 
However, in this sense, the embedding $\eta_G:G\to\mathsf{BM}(\tilde{\mathsf K}G)$ generally fails to be natural in $G$. This is due to the fact that we were able to functorially extend $\mathsf{BM}$ only for open maps; but, for a $\B$-morphism $f$,
the map $(\K( \D_{\B}f)$ is not necessarily open.

Fortunately, we can ameliorate the situation and ``make'' the embedding natural, by considering the (non-full) subcategory $\Games_{\mathrm{ls}\B}$  of $\Games_{\B}$ with the same objects, but whose morphisms are locally surjective. Since $\D_{\B}$ preserves the local surjectivity of $\B$-morphisms, and since $\K$ maps locally surjective $\B$-morphisms to open maps, we can restrict our functor $\tilde{\mathsf K}$ to $\Games_{\mathrm{ls}\B}$ and consider its values (via the metric-induced topology) as lying in the domain of the functor $\BM:\CTop_{\mathrm{open}}\to\Games_{\B}$. Then, with the inclusion functor $\mathsf J: \Games_{\mathrm{ls}\B}\to \Games_\B$, and without renaming the restriction of $\tilde{\mathsf K}$ as just described, we obtain:

\begin{thm}\label{THM_BMUni_Fun}
	Under the restriction to locally surjective $\B$-morphisms, the game embedding $\eta_G:G\to \BM(\tilde{\mathsf K}G)$ becomes natural in $G$, so that one has a natural transformation
	$\mathsf{J}\stackrel{\eta}{\to}\BM\circ \tilde{\K}$. 
\end{thm}  

\section{Categorical constructions}\label{SEC_CatProp}
In this section we establish some important properties of our category $\Games_\A\cong\Games_\B$. 

\subsection{Completeness and Cocompleteness}
From the equivalent presentation of $\Games_{\A}$ as a full coreflective subcategory of a topological category over a presheaf topos as given in Theorem \ref{THM_IsoSubSetsNop}, one immediately gets:

\begin{cor}
	The category $\Games_\A$ is complete and cocomplete.
\end{cor}

In arboreal terms, by the equivalence given in Theorem \ref{THM_Graphs&Games} it is clear how to obtain limits and colimits in $\Games_{\A}$: the construction of colimits in the coreflective subcategory $\TrGame$ of $ \Sub{\Tree}$, which is equivalent to $\Games_\A$, proceeds in $\Sub{\Tree}$, whereas for constructing limits in $\TrGame$ one must apply the pruning functor to the limit formed in $\Sub{\Tree}$.

In what follows we concentrate on the game-theoretic construction and interpretation of products and equalizers and their duals in $\Games_\A$.  As only some of these have appeared in the literature on infinite games, we give a summary of them, before turning to properties that are aided by, but do not automatically follow from, our various presentations of $\Games_{\A}$. 

\subsubsection{Products}\label{SUBSECT_products}
We start by describing products of games in $\Games_{\A}$ (see, for instance, \cite{Galvin2016}).
For a family $\mathcal{G}=\left(G^\alpha=(T^\alpha,A^\alpha):\alpha\in \kappa\right)$ of games, we let $\Multi_\A(\mathcal{G})=(T,A)$ denote the game with
\begin{gather*}
	T=\bigcup_{n<\omega}\set{t=\seq{(x_i^\alpha)_{\alpha\in\kappa}:i\le n}:  \seq{x_i^\alpha:i\le n}\in T^\alpha \text{ for every $\alpha\in \kappa$}},\\
	A=\set{t=\seq{(x_n^\alpha)_{\alpha\in\kappa}: n<\omega}:  \seq{x_n^\alpha:n<\omega}\in A^\alpha \text{ for every $\alpha\in \kappa$}}.
\end{gather*}
Of course, $\M(\Multi_\A((G^\alpha:\alpha\in \kappa))\subseteq \prod_{\alpha\in\kappa}\M(G^\alpha)$.
Note that $\kappa =0$ (empty families) is permitted, in which case we obtain the terminal game $\mathbf 1$, as already described in Example \ref{EX_boring}.

Intuitively, $\Multi_\A(\mathcal{G})$ is the game in which $\ali$ and $\bob$ play all the games of $\mathcal G$ \textit{at the same time}, and its winning criteria is that $\ali$ must win the runs of every game in $\mG$. Analogously, $\Multi_\B(\mathcal{G})$ is the same game, but the winning criteria reverses, in the sense that now $\bob$ must win the run of every game in $\mG$ (the games $\Multi_\A(\mathcal{G})$ and $\Multi_\B(\mathcal{G})$ are often called \emph{multiboard games} in the literature). 

Now, with the projections
\begin{center}
	\begin{tikzcd}[row sep = 0em]
		\Multi_\A(\mathcal{G}) \arrow[r, "\pi^\alpha"] &	G^\alpha \\
		\seq{(x_i^\alpha)_{\alpha\in\kappa}:i\le n}  \arrow[r, mapsto] &   \seq{x_i^\alpha:i\le n},
	\end{tikzcd}
\end{center}
one easily confirms that  $$\Multi_\A(\mathcal{G})\cong\prod_{\alpha\in\kappa} G^{\alpha}$$ serves as a product in $\Games_\A$.
(We note that, in what follows, we often write a finite sequence $\seq{(x_i^\alpha)_{\alpha\in\kappa}:i\le n} $ of $\kappa$-indexed families equivalently as a $\kappa$-indexed family 
\begin{equation}\label{EQ_multi}
	(\seq{x_i^\alpha:i\le n})_{\alpha\in\kappa}
\end{equation}
of finite sequences, even though, formally, for $\kappa$ infinite, the latter element can never be a moment of any game.)

\subsubsection{Coproducts}\label{SUBSECT_coproducts} Coproducts of games are even easier to describe than products.
For a family $\left(G^{\alpha}=(T^{\alpha},A^{\alpha}):\alpha\in \kappa\right)$ of games, one defines $\bigsqcup_{\alpha\in \kappa}G^{\alpha}=(T,A)$ by
\begin{gather*}
	T:=\{\seq{\,}\}\cup\left(\bigsqcup_{\alpha\in \kappa}\left(T^{\alpha}\setminus\{\seq{\,}\}\right)\right),\quad
	A:=\bigsqcup_{\alpha\in \kappa}A^{\alpha}.
\end{gather*}

Intuitively, $\bigsqcup_{\alpha\in \kappa}G^{\alpha}$ is the game in which $\ali$ also chooses at her first move which of the games $G^{\alpha}$ she and $\bob$ will be playing from that moment onward. With the canonical injections defined (on non-empty moments) as in $\Sets$, it is clear that $(T,A)$ serves as the coproduct $\coprod_{\alpha\in\kappa} G^{\alpha}$ of the given family of games, not only in $\Games_\A$, but also in $\Games_\B$.

\subsubsection{Equalizers and pullbacks}\label{SUBSECeqpb} The construction of the equalizer in $\Games_{\A}$ of a pair of morphisms $f,g\colon G\to G'$ with $G=(T,A)$ and $G'=(T',A')$ is, in functorial terms, predicated by the coreflector as described before Corollary  \ref{COR_Tree_SetNop}. Transcribed in game-theoretic notation, we simply must consider the subgame $G''=(T'',A'')\leq G$ defined by
\[T''(n):=\{ R\upharpoonright n: R\in\mathrm{Run}(T), \overline{f}(R)=\overline{g}(R)\}\quad\text{and}\quad A'':=A\cap T''\]
for all $n<\omega$. Indeed, one easily sees that the inclusion map $G''\to G$ serves as the desired equalizer.

The subgame $G''$ may be hard to play and win: at every moment of the game $G''$ one has to be sure that there is a run in $G$ which, once mapped by both,  $\overline{f}$ and $\overline{g}$, results in the same run in $G'$, and $\ali$ wins such run in $G''$ if it is a winning run in $G$.

Of course, given $\A$-morphisms $G_1\stackrel{f}{\to}G\stackrel{g}{\leftarrow}G_2$, their pullback is formed by taking the equalizer 
of the pair $f\circ \pi_1, g\circ \pi_2$,
with the projections 
$\Multi_{\A}(G_1,G_2)\stackrel{\pi_i}{\to} G_,\,i=1,2$. 
Explicitly, in the pullback span $G_1\stackrel{p_1}{\leftarrow}G_1\times_G  G_2 \stackrel{p_2}{\to} G_2$, the subgame $G_1\times_G  G_2$ of $\Multi_{\A}(G_1,G_2)$ is
obtained by pruning the tree 
\[
\set{(t_1,t_2)\in \Multi_{\A}(G_1,G_2): f(t_1)=g(t_2)},
\]
and $p_i$ is the restriction of the projection $\pi_i$. Equivalently, one forms the pullback of the cospan
$T_1\stackrel{f}{\to}T\stackrel{g}{\leftarrow}T_2$ in $\Gmes$ and then provides $T_1\times_T T_2$ with the payoff set $\overline{p_1}^{-1}[A_1]\cap \overline{p_2}^{-1}[A_2]$ in order to obtain the game $G_1\times_G  G_2$.

\subsubsection{Coequalizers and pushouts}\label{SUBSEC_Coeq_Pushouts}
In order to form the coequalizer of a pair of $\A$-morphisms $f,g\colon G\to G'$ with $G=(T,A)$ and $G'=(T',A')$, we let $\sim$ be the smallest equivalence relation on $T$ such that $f(t)\sim g(t)$ for every $t\in T$. Then, whenever $t\sim s$, one obviously has
$$ |t|=|s|\quad\text{and}\quad t\restrict i\sim s\restrict i\;\text{ for every } i\le |t|=|s|,$$
and with $[t]$ denoting the $\sim$-equivalence class of $t$, we define
\begin{gather*}
	T'' =\set{\seq{[t\restrict i]:0<i\le|t|}: t\in T},\qquad
	A'' = \set{ \seq{[R\restrict n]:0<n<\omega}: R\in A}.
\end{gather*}
Then the projection $G'\stackrel{q}{\to}_\A G''=(T'',A'')$ with $q(t)=\seq{[t\restrict i]:0<i\le|t|}$ serves  as a coequalizer of $f,g\colon G\to G'$ in $\Games_\A$.

We may think of $G''$ as the game arising from $G'$ in which all $\sim$-equivalent moves are declared equal, and for $\ali$ to have a winning run in $G''$ it suffices to have a winning $\sim$-representative in $G'$ for that run. The difficulty lies in determining all equivalent moves at every moment of the game $G''$.

In $\Games_\B$, the coequalizer $G'' =(T'',A'')$ of the $\B$-morphisms $f,g:G\to G'$ may be constructed in the same way, except that the pay-off set has to be changed to
\[
A'' = \Run(T'')\setminus \set{ \seq{[R\restrict n]:0< n<\omega}: R\in \Run(T)\setminus A}.
\]

The pushout $G_1\sqcup_GG_2$ of $\A$-morphisms $G_1\stackrel{f}{\leftarrow}G\stackrel{g}{\to}G_2$ is formed by taking the coequalizer 
of the pair $i_1\circ f, i_2\circ g$, 
with the coproduct injections  
$G_j\stackrel{i_j}{\to} G_1\sqcup G_2,\,j=1,2$.

\subsection{Infinitary extensiveness}

Knowing how to construct coproducts and pullbacks makes it easy to prove that the category $\Games_\A$ has an important ``space-like'' property, shared with the category of topological spaces as well as with all categories of $\Sets$-valued sheaves.
Recall that a finitely complete category $\mathbf C$ with coproducts is {\em infinitary extensive}  \cite{Carboni1993, Centazzo2004}  if, for every family of objects $A_j\,(j\in J)$, a morphism into the coproduct $\coprod_{j\in J}A_j$ is determined by the family of morphisms given by its pullbacks along every coproduct injection; more precisely, in terms of the ``slices'' of the category $\mathbf C$, if the canonical functor
$$\prod_{j\in J}\mathbf C/A_j\longrightarrow\mathbf C/\coprod_{j\in J}A_j$$
is an equivalence of categories.

\begin{thm}
	The category $\Games_\A$ is infinitary extensive.
\end{thm}
\begin{proof}
	Let
	\begin{center}
		\begin{tikzcd}[column sep=15mm, row sep=15mm]
			G'_{j} \arrow[r, "f_j"] \arrow[d, "g_j"]&	G \arrow[d,"h"]\\
			G_j \arrow[r, hook, "i_j"] & \bigsqcup_{j\in J}G_j
		\end{tikzcd}
	\end{center}
	be a $J$-indexed family of commutative diagrams in $\Games_\A$, with $G=(T,A)$, $G_j=(T_j,A_j)$ and $G'_j=(T'_j,A'_j)$ for all  $j\in J$. We need to show that  every square is a   pullback diagram if, and only if, $(G'_{j}\stackrel{f_j}{\to}G)_{j\in J}$ is a coproduct diagram.
	
	First, note that, because $\set{T_j\setminus \{\seq{\,}\}:j\in J}$ disjointly covers the tree of $\bigsqcup_{j\in J}G_j$, one has
	\begin{equation}\label{EQ_extensiveness}
		T\setminus\{\seq{\,}\}=\bigsqcup_{j\in J}(h^{-1}(T_j)\setminus \{\seq{\,}\}).
	\end{equation}
	Then, every square is a pullback diagram if, and only if, $G'_j\le G$ where, as a pullback of the inclusion map $i_j$, also $f_j$ may be taken to be an inclusion map  
	and $T'_j$ to be obtained by pruning $h^{-1}(T_j)$.
	
	Now, given $j\in J$, let $t\in h^{-1}(T_j)$. Then $h(t^\smallfrown x)\in T_j$ for every $x$ such that $t^\smallfrown x\in T$ and, thus, $h^{-1}(T_j)$ is a pruned subtree of $T$. Hence, every square is a pullback diagram if, and only if, $G'_j\le G$ (with $f_j$ being an inclusion map)
	and $T'_j=h^{-1}(T_j)$ -- which, together with (\ref{EQ_extensiveness}), concludes the proof.
\end{proof}

We note the following (generally valid) consequence of the Theorem:

\begin{cor}
	The category $\Games_\A$ is infinitary distributive, i.e., the canonical morphism $\bigsqcup_{j\in J}G\times G_j\to G\times\bigsqcup_{j\in J}G_j$ is an isomorphism, for all games $G, G_j\, (j\in J)$.
\end{cor}
\begin{proof}
	Just consider the pullback diagrams
	\begin{center}
		\begin{tikzcd}[column sep=15mm, row sep=15mm]
			G\times G_{j} \arrow[r, hook, "\mathrm{id}_G\times i_j"] \arrow[d, "p_2^j"]&	G\times \bigsqcup_{j\in J}G_j \arrow[d,"p_2"]\\
			G_j \arrow[r, hook, "i_j"] & \bigsqcup_{j\in J}G_j\;,
		\end{tikzcd}
	\end{center}
	whose top arrows must represent a coproduct.
\end{proof}

\subsection{Orthogonal factorization systems}
Recall that  an orthogonal factorization system $(\mathcal E,\mathcal M)$ in a category $\mathcal C$ is given by two classes of morphisms, both closed under composition with isomorphisms, such that every morphism $f$ factors as $f=m\cdot e$ with $e\in\mathcal E, m\in\mathcal M$, and the diagonalization (or lifting) property holds: whenever $v\cdot e'=m'\cdot u$ with $e'\in\mathcal E, m'\in\mathcal M$, then $t\cdot e'=u$ and $m'\cdot t=v$ for a unique morphism $t$ (in which case one says that $e'$ is orthogonal to $m'$). Such systems are completely determined by just one of the two classes since, when fixing one of them, one obtains the other as containing precisely those morphisms that are orthogonal to all morphisms in the class considered first. The system $(\mathcal E,\mathcal M)$ is  proper if $\mathcal E$ is a class epimorphisms and $\mathcal M$ is a class of monomorphisms in the category $\mathcal C$.

Here we will present four proper orthogonal factorization systems in $\Games_{\A}$, taking advantage of the fact that the forgetful functor $$\mathsf{Sub}_{\mathrm{Run}}(\Gmes)=\Games_{\mathrm A}\longrightarrow\Gmes$$ is topological (Proposition \ref{PROP_TopFunctor}). In fact, it is well known (see 21.14 in \cite{Adamek1990} or, more generally, II.5.7.1 in \cite{Hofmann2014}) that for any topological functor $U:\mathcal T\to \mathcal C$ and any (proper) orthogonal factorization system $(\mathcal E,\mathcal M)$ in $\mathcal C$, in the category $\mathcal T$ one then has the (proper) orthogonal factorization systems
$$(U^{-1}[\mathcal E],\,  U^{-1}[\mathcal M]\cap\{U\text{-initial morphisms}\})\quad\text{and}\quad(U^{-1}[\mathcal E]\cap\{U\text{-final morphisms}\},\, U^{-1}[\mathcal M])$$
which, neglecting the functor $U$, we will allow ourselves to respectively denote simply by
$$(\mathcal E,\mathcal M_*)\quad\text{and}\quad (\mathcal E^*,\mathcal M),$$
thus denoting the classes $U^{-1}[\mathcal E]$ and $U^{-1}[\mathcal M]$ in $\mathcal T$ again by $\mathcal E$ and $\mathcal M$, respectively.

Hence, for the forgetful functor $U:\mathcal T=\mathsf{Sub}_{\mathrm{Run}}(\Gmes)\longrightarrow\mathcal C=\Gmes$ one has the following Lemma which the reader may also verify easily in a direct manner.

\begin{lemma}\label{LEMMA_SUBOrtFactSyst}
	For any (proper) orthogonal factorization system $(\mathcal E,\mathcal M)$ in $\Gmes$ one has the (proper) orthogonal factorization systems $(\mathcal E,\mathcal M_*)$ and $(\mathcal E^*,\mathcal M)$ in $\Games_{\A}$, with
	\begin{align*}
		\mathcal M_* &= \set{(T,A)\stackrel{f}{\to}(T',A'): (T\stackrel{f}{\to}T' )	\in\mathcal M, \,A = \overline{f}^{-1}[A']},\\
		\mathcal E^* &= \set{(T,A)\stackrel{f}{\to}(T',A'): (T\stackrel{f}{\to}T' )\in\mathcal E,\, A' = \overline{f}[A]}.\\
	\end{align*}
\end{lemma}

We will now present two proper orthogonal factorization systems in $\Gmes$, each of  which giving two systems in $\Games_{\A}$ by Lemma \ref{LEMMA_SUBOrtFactSyst}, leading us to four distinct proper orthogonal factorization systems in $\Games_{\A}$.

\begin{prop}\label{PROP_Gmes_RegMon}
	The category $\Gmes$ has the proper orthogonal factorization system 
	$$(\mathcal E,\mathcal M):=(\{\text{surjective morphs}\},\{\text{injective morphs}\})=(\{\text{epis}\},\{\text{regular monos}\}).$$
\end{prop}
\begin{proof}
	Every chronological mapping $f\colon T\to T'$ factors into a surjective chronological mapping followed by an injective one: simply consider the obvious factorization through the inclusion of the subgame $f[T]\leq T'$. The surjections are precisely the epimorphisms in $\Gmes$---in Proposition \ref{PROP_morphcat}(b) we already noted the corresponding statement to hold in $\Games_{\A}$. It now suffices to show that the injective chronological maps are precisely the regular monomorphisms ({\em i.e.}, the equalizers of parallel pairs of morphisms) in $\Gmes$, since their regularity makes them {\em a fortiori} strong ({\em i.e.}, orthogonal to all epimorphisms).

	Presented as equalizers, regular monomorphisms in $\Gmes$ are clearly injective. Conversely, having an injective chronological mapping, we may assume that, up to isomorphism, it is the inclusion map $f:T\hookrightarrow T'$ of a subtree $T\leq T'$. Forming the cokernelpair of $f$ ({\em i.e.}, the pushout of two copies of $f$), one easily confirms that one has the equalizer diagram
	\[
	\begin{tikzcd}[column sep=normal,row sep=huge]
		T \arrow[r, "f"] & T' \arrow[r,shift left=.75ex,"q_1"] \arrow[r,shift right=.75ex,"q_2"{below}]  &
		(T'\setminus T)\sqcup T\sqcup(T'\setminus T),
	\end{tikzcd}
	\]
	where the two pushout injections $q_i$ map the moments in $T$ identically but separate the other moments into distinct summands $T'\setminus T$ of their codomain.
\end{proof}

Having shown that the class of epimorphisms belongs to an orthogonal factorization system, we now deduce from general categorical facts that also the class of monomorphisms in $\Gmes$ (characterized like in Proposition \ref{PROP_morphcat}(a) for $\Games_{\A}$) belongs to an orthogonal factorization system in $\Gmes$.

\begin{prop}\label{PROP_GmesRegMonFact}	
	The category $\Gmes$ has the proper orthogonal factorization system $(\mathcal E',\mathcal M')$ with $$\mathcal M':=\{T_1\stackrel{f}{\to}T_2:\mathrm{Run}(T_1)\stackrel{\overline{f}}{\to}\mathrm{Run}(T_2) \text{  injective}\}=\{\text{monomorphisms}\}$$ and (necessarily) $\mathcal E':=\{\text{strong epimorphisms}\}$.
\end{prop}

\begin{proof}
	The morphisms of any complete and well-powered category have $(\{\text{strong epis}\},\{\text{monos}\})$-factorizations. (Recall that well-powered means that, for every given object, the isomorphism classes of monomorphisms into it may be labelled by a set.) For a proof, see for example II.5.3.3 of \cite{Hofmann2014}. Since in $\Gmes$, for any monomorphism $T_1\to T_2$, the cardinality of $T_1$ is bounded by the cardinality of $\mathrm{Run}(T_2)$, the complete category $\Gmes$ is well-powered, and therefore has the desired factorizations.
\end{proof}

Keeping the notation of the two propositions, with Lemma \ref{LEMMA_SUBOrtFactSyst} we obtain easily most of the following claims:
\begin{cor}\label{COR_ortfactsyst}
	The pairs $(\mathcal E,\mathcal M_*),\,(\mathcal E^*,\mathcal M),\, (\mathcal E',(\mathcal M' )_*)$, and $((\mathcal E')^*,\mathcal M')$ constitute four distinct proper orthogonal factorization systems in $\Games_\A$, ordered by
	\begin{center}
		$
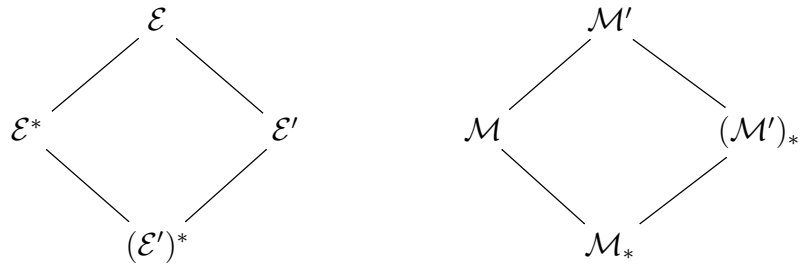
\xymatrix{& \mathcal E &&&& \mathcal M' & \\
			\mathcal E^*\ar@{-}[ur]\ar@{-}[dr] && \mathcal E'\ar@{-}[lu]\ar@{-}[ld] && \mathcal M\ar@{-}[ru]\ar@{-}[rd] && (\mathcal M')_*\ar@{-}[lu]\ar@{-}[ld]\\
			& (\mathcal E')^* &&&& \mathcal M_*\\\
		}$
	\end{center}
	Here the systems
	\begin{align*}
		(\mathcal E,\mathcal M_*)&=(\{\text{surjective morphs}\},\{\text{game embeddings}\})=(\{\text{epis}\},\{\text{regular monos}\})\\
		((\mathcal E')^*,\mathcal M')&=(\{\text{strong epis}\},\{f:\overline{f} \text{ injective}\})=(\{\text{strong epis}\},\{\text{monos}\})	\\
	\end{align*}
	assume the roles of the least and largest proper orthogonal factorization systems in $\Games_{\A}$.	
\end{cor}

\begin{proof}
	Only the claims about the classes of (regular/strong) monomorphisms and epimorphisms in $\Games_{\A}$ vis-\'{a}-vis the corresponding classes in $\Gmes$ may warrant further explanation. But, just like Lemma \ref{LEMMA_SUBOrtFactSyst}, they are all easy consequences of the topologicity of the forgetful functor $\mathsf{Sub}_{\mathrm{Run}}(\Gmes)=\Games_{\mathrm A}\longrightarrow\Gmes$, as recorded in generality in the following remark.
\end{proof}

\begin{remark}\label{REM_topfun_RegEpi_StrEpi}
	A topological functor $U:\mathcal T\to\mathcal C$ preserves and reflects monomorphisms and epimorphisms. Furthermore, for every morphism $f$ in $\mathbf C$, one has (\cite{Tholen1974}, Prop. 9.7) that
	\begin{itemize}
		\item $f$ is regularly monic in $\mathbf C$ if, and only if, $f$ is $P$-initial and $Pf$ is regularly monic in $\mathbf D$;
		\item $f$ is regularly epic in $\mathbf C$ if, and only if, $f$ is $P$-final and $Pf$ is regularly epic in $\mathbf D$.	
	\end{itemize}
	Likewise, when ``regularly'' is replaced by ``strongly''.
\end{remark}

\subsection{Regular versus strong epimorphisms}
Proposition \ref{PROP_Gmes_RegMon} and Corollary \ref{COR_ortfactsyst} imply that the classes of strong monomorphisms and regular monomorphisms coincide, in both, $\Gmes$ and $\Games_{\A}$. It is therefore natural to ask the corresponding question for strong and regular epimorphisms. Next we will characterize the regular epimorphisms and then construct two consecutive regular epimorphisms whose composite fails to be regular, even though it is trivially strong.  

In order to characterize the regular epimorphisms in the category $\Gmes$, we recall from Section \ref{SUBSECeqpb} that the kernel pair of a morphism $T\stackrel{f}{\to} T'$ in  $\Gmes$, {\em i.e.}, the pullback of two copies of $f$, 
is described by the projections
$\begin{tikzcd}
	T\times_{T'} T \arrow[r,shift left=.75ex,"\pi_1"] \arrow[r,shift right=.75ex,"\pi_2"{below}]  &	T
\end{tikzcd}$ of 
\[
T\times_{T'} T = \set{(t,s):|t|=|s| \text{ and }\exists\, R,S\in \Run(T)\, (R\restrict|t| = t, \, S\restrict |s| = s \text{ and }\overline{f}(R)=\overline{f}(S))}\,.
\]
We will show that the characteristic property of the coequalizer of $\pi_1,\pi_2$ is captured by the following definition:

\begin{defn}
	A surjective morphism $T\stackrel{q}{\to}Q$ in $\Gmes$ is a {\em quotient map} if, for all $N<\omega$ and $t,s\in T(N)$, one has $q(t)=q(s)$ if, and only if, there are runs $\seq{R_0, \dotsc, R_n}$ in $T$ such that
	\[
	t=R_0\restrict N, \; s = R_n\restrict N, \; \Delta(R_{2i}, R_{2i+1})\ge N \text{  and  } \overline{q}(R_{2i+1}) = \overline{q}(R_{2i+2})
	\]
	for every $i < n/2$.
\end{defn}

\begin{prop}\label{PROP_Gme_RegEpis}
	A morphism $T\stackrel{f}{\to}T'$ in $\Gmes$ is regularly epic if, and only if, it is a quotient map.
\end{prop}
\begin{proof}
	Let $\begin{tikzcd}
		T\times_{T'} T \arrow[r,shift left=.75ex,"\pi_1"] \arrow[r,shift right=.75ex,"\pi_2"{below}]  &
		T
	\end{tikzcd}$  be the kernel pair of $T\stackrel{f}{\to} T'$ in $\Gmes$. 
	As shown in Section \ref{SUBSEC_Coeq_Pushouts}, the coequalizer 
	of  $\pi_1,\pi_2$ is formed with 
	the least equivalence relation $\sim$ on $T$ satisfying $t\sim s$ for all $(t,s)\in T\times_{T'} T$. 
	Hence, on one hand, for the surjective morphism $f$ to be regularly epic, $\sim$ must coincide with the equivalence relation induced by $f$.	
	On the other hand, for $f$ to qualify as a quotient map, the defining property tells us that the equivalence relation induced by $f$ should coincide with the transitive closure $\approx$ of the relation $\mathcal R$ on $T$, defined by $t\mathcal R s$ if, and only if, there are $R,S\in \Run(T)$ extending $t$ and $s$, respectively, and satisfying $\overline{f}(R)=\overline{f}(S)$. 
	
	One obviously has that the relations $\sim$ and $\approx$ coincide and may routinely finish the proof.
\end{proof}

Now, for the construction of two consecutive quotient maps whose composite fails to be a quotient map, first, for every $n<\omega$, we consider the game tree
\[
T_n = \set{\seq{*_n:j<k}:k<n} \cup \set{\seq{*_n:i<n}^\smallfrown \seq{x:j<k}:k<\omega,\, x\in \{0,1\} }.
\]
and let $T = \bigsqcup_{n<\omega} T_n$. We write an element $t\in T$ in the form $t_n$ to indicate the location of $t$ within the coproduct.
With the cogenerating game tree
\[
Q = \set{\seq{\ast:i<n}^\smallfrown \seq{0:j<k}: n,k<\omega}
\]
we define the map $q\colon T\to Q$ by
\[
q(t_n) = \begin{cases}
	\seq{\ast: i<n+t_n(n)}^\smallfrown\seq{0: j<|t|-(n+t_n(n))}, \text{ if $|t|>n$},\\
	\seq{\ast:i<|t|}, \text{ otherwise.}
\end{cases}
\]

It is clear that $q$ is chronological. Furthermore:

\begin{lemma}\label{LEMMA_q_regular}
	The morphism $T\stackrel{q}{\to}Q$ in $\Gmes$ is a quotient map.
\end{lemma}
\begin{proof}
	Suppose that $q(t_n)=q(s_m)$ for $t_n,s_m\in T$. 
	
	In case $n=m$, we note that $q(t_n)=q(s_n)$ equivalently means that $t_n=s_n$. Indeed, this is clear for $|t_n|=|s_n| \le n$, and if $|t_n|=|s_n| > n$ and $t_n(n) \neq s_n(n)$, then
	\begin{align*}
		q(t_n) &= \seq{\ast: i<n+t_n(n)}^\smallfrown\seq{0: j<|t|-(n+t_n(n))} \\
		&\neq \seq{\ast: i<n+s_m(n)}^\smallfrown\seq{0: j<|t|-(n+s_n(n))} = q(s_n),
	\end{align*}
	contradicting our hypothesis. Thus $t_n(n) = s_n(n)$ and $t_n = s_n$.
	
	If $n<m$, we will show the existence of a sequence of runs of $T$ which attests the quotient property of $q$ at $t_n,s_m\in T$, by considering the following two possible cases:
	\begin{itemize}
		\item $|t_n|\le n$: Then $t_n = \seq{\ast_n: i<|t_n|}$, in which case the sequence $\seq{R_0, \dotsc, R_k}$ with 
		\begin{gather*}
			R_0 = \seq{\ast_n: i<n}^\smallfrown\seq{1:j<\omega},\\
			R_1 = \seq{\ast_{n+1}: i<n+1}^\smallfrown\seq{0:j<\omega}\\
			R_2 = \seq{\ast_{n+1}: i<n+1}^\smallfrown\seq{1:j<\omega},\\
			\vdots\\
			R_k = \seq{\ast_{m}: i<m}^\smallfrown\seq{0:j<\omega},
		\end{gather*}
		satisfies the desired properties.
		
		\item $|t_n|> n$: We claim that $t_n(n) =  1$. Indeed, suppose that $t_n(n) = 0$, in which case 
		\begin{gather*}
			t_n = \seq{\ast_n: i<n}^\smallfrown \seq{0:j<|t_n|-n}\\
			\text{and } q(t_n) = \seq{\ast: i<n}^\smallfrown \seq{0:j<|t_n|-n}.
		\end{gather*}
		However, 
		\begin{equation}\label{EQ_Gme_Strong_not_Reg}
			\forall k>n \, \forall t_k\in T_k \left(|t_k|>n\implies q(t_k)(n)=\ast \right).
		\end{equation}
		Thus, in particular,
		\[
		q(s_m)(n)=\ast\neq 0 = q(t_n)(n)
		\] 
		and $q(t_n)\neq q(s_m)$, contradicting our hypothesis. 
		
		Hence, 
		\begin{gather*}
			t_n = \seq{\ast_n: i<n}^\smallfrown \seq{1:j<|t_n|-n}\\
			\text{and } q(t_n) = \seq{\ast: i<n+1}^\smallfrown \seq{0:j<|t_n|-n-1}.
		\end{gather*}
		
		We now further separate the rest of the proof into two subcases:
		
		\begin{itemize}
			\item $|t_n|>n+1$: Then $q(t_n)(n+1) =q(s_m)(n+1)= 0$ and, by \eqref{EQ_Gme_Strong_not_Reg}, $m=n+1$. In this case, $\seq{R_0,R_1}$ with
			\begin{gather*}
				R_0=\seq{\ast_n:i<n}^\smallfrown \seq{1:j<\omega}\\
				R_1 = \seq{\ast_{n+1}:i<n+1}^\smallfrown\seq{0:j<\omega}.
			\end{gather*} 
			is such that $R_0\restrict (n+1) = t_n$, $R_1\restrict (n+1) = s_m$,  $\overline{q}(R_0)=\overline{q}(R_1)$, and we are done.
			\item $|t_n|=n+1$: Then the sequence $\seq{R_0, \dotsc, R_k}$ with 
			\begin{gather*}
				R_0 = \seq{\ast_n: i<n}^\smallfrown\seq{1:j<\omega},\\
				R_1 = \seq{\ast_{n+1}: i<n+1}^\smallfrown\seq{0:j<\omega}\\
				R_2 = \seq{\ast_{n+1}: i<n+1}^\smallfrown\seq{1:j<\omega},\\
				\vdots\\
				R_k = \seq{\ast_{m}: i<m}^\smallfrown\seq{x:j<\omega},
			\end{gather*}
			where $x=s_m(m)$ if $|s_m|>m$ and $x=0$ otherwise, concludes the proof that $q$ is a quotient map.
		\end{itemize} 
	\end{itemize}
\end{proof}

Since the coproduct of regular epimorphisms in any category (with coproducts) is again a regular epimorphism, we conclude:

\begin{cor}\label{PROP_Gme_QotientCoprods}
	The morphism $q\sqcup q \colon T\sqcup T\to Q\sqcup Q$ in $\Gmes$ is a quotient map.	
\end{cor}

For ease of notation, let us write $T\sqcup T = T^0\sqcup T^1$ and $Q\sqcup Q = Q^0\sqcup Q^1$ with $T^k = T$ and $Q^k = Q$ ($k\in \{0,1\}$), and for an element $t\in T^0\sqcup T^1$ we use the notation $t_n^k$ with $n<\omega$ and $k\in \{0,1\}$ to indicate its location within the coproduct: $t_n^k\in T_n\subseteq T^k$.

Now, for the game tree 
\[
C = \set{\seq{*:i<n}^\smallfrown \seq{x:j<k}: n,k<\omega, \, x\in \{0,1\}}
\]
we define $c\colon Q^0\sqcup Q^1 \to C$ by assigning to $t^k=\seq{\ast^k:i<n}^\smallfrown \seq{0:j<m}\in Q^k$ the moment
\[
c(t^k) = \seq{\ast:i<n}^\smallfrown \seq{k:j<m}.
\]
Clearly, $c$ is a chronological surjection. Furthermore,
for distinct $t,s\in  Q^0\sqcup Q^1$ one has $c(t)=c(s)$ if, and only if, $t=\seq{\ast^{k}:i<n}\in Q^k$ and $s=\seq{\ast^{1-k}:i<n}\in Q^{1-k}$ for some $n<\omega$, in which case
\[
R_0 = \seq{\ast^k:i<\omega} \text{ and }
R_1 = \seq{\ast^{1-k}:i<\omega}
\]
are runs in $Q^0\sqcup Q^1$ such that $R_0\restrict n = t$, $R_1\restrict n = s$ and $\overline{c}(R_0) = \overline{c}(R_1) = \seq{\ast: i<\omega}$. So, $c$ is a quotient map.

However:
\begin{lemma}\label{LEMMA_e_not_RegEpic}
	Whereas the morphisms $q\sqcup q$ and $c$ are quotient maps, the composite morphism $e = c\circ (q\sqcup q)\colon T^0\sqcup T^1\to C$ in $\Gmes$ is not.
\end{lemma}
\begin{proof}
	We first note that any $R,R'\in \Run(T^0\sqcup T^1)$ with $\overline{e}(R) = \overline{e}(R')$ must  lie in $ \Run(T^k)$, for the same $k\in\{0,1\}$. Indeed, if we had $R\in \Run(T^0)$ and $R'\in \Run(T^1)$, then $\overline{e}(R)(n)=0$ and $\overline{e}(R')(n)=1$ for sufficiently large $n<\omega$.
	
	Now, the singleton sequences $t^0_1 = \seq{\ast_1^0}\in T_1\subset T^0$ and $t^1_1 = \seq{\ast_1^1}\in T_1\subset T^1$ satisfy $e(t^0_1)=e(t^1_1) = \seq{\ast}$. However, since  (in the notation of Lemma \ref{PROP_ChronDeltaChar}) any  $R\in T^0$ and $R'\in T^1$ satisfy
	\[
	\Delta(R,R')= 0 < 1 = |t^0_1| = |t^1_1|,
	\]
	it follows that there can be no sequence $\seq{R_0, \dotsc, R_n}$ of runs in $\Run(T^0\sqcup T^1)$ attesting the quotient map property of $e$ at $t^0_1, t^1_1$.
\end{proof}

\begin{thm}\label{THEOREM_non-regular}
	In both categories, $\Gmes$ and $\Games_{\A}$, the composite of regular epimorphisms may fail to be a regular epimorphism. As a consequence, these categories contain strong epimorphisms that fail to be regular, and their morphisms generally lack $(\{\text{regular epis}\},\{\text{monos}\})$-factorizations. 
\end{thm}

\begin{proof}
	A ``folklore'' fact of category theory says that, in any category with kernel pairs and their coequalizers, the following properties are equivalent:
	\begin{itemize}
		\item[(i)] Strong epimorphisms are regular.
		\item[(ii)] The class of regular epimorphisms is closed under composition.
		\item[(iii)] $(\{\text{regular epis}\},\{\text{monos}\})$ is an (orthogonal) factorization system.
	\end{itemize}
	Since a direct proof of this fact may not be found easily, let us sketch one here. Indeed, regular epimorphisms are always strong, and the class of strong  epimorphisms is always closed under composition. This shows (i)$\Rightarrow$(ii). For (ii)$\Rightarrow$(iii), given a morphism $f$, one factors $f=m\cdot e$ with $e$ the coequalizer of the kernel pair of $f$ and then factors the unique morphism $m=n\cdot d$ in the same manner. Then, since by hypothesis (ii) $d\cdot e$ is a regular epimorphism, its kernel pair factors through the kernel pair of $f$. This makes the coequalizer $d\cdot e$ factor through $e$, so that $d$ must be an isomorphism. Hence, $d$ has a trivial kernel pair, which makes $m$ a monomorphism. Finally, assuming (iii), one sees that even all extremal epimorphisms ({\em i.e.}, those, which admit only isomorphisms as a second monic factor) are regular. Since strong epimorphisms are trivially extremal, one obtains (i).
	
	Equipped with this general fact, since Proposition \ref{PROP_Gme_RegEpis} and Lemma \ref{LEMMA_e_not_RegEpic} tell us that condition (i) fails in $\Gmes$, we obtain all assertions of the Theorem for that category. Providing the principal witness
	$e = c\circ (q\sqcup q)\colon T^0\sqcup T^1\to C$ with empty payoff sets, we conclude (see Remark \ref{REM_topfun_RegEpi_StrEpi}) that $e$ maintains this role also in $\Games_{\A}$.
\end{proof}

\subsection{Coregularity, non-regularity, and descent} We will show that the categories $\Gmes$ and $\Games_{A}$ are coregular ({\em i.e.}, their duals are regular), so that, in addition to the existence of  the (epi, regular mono)-factorizations, in these categories not only the class of epimorphisms, but also the class of regular monomorphisms is stable under pushout. By contrast, these categories fail badly to be regular: not only are they missing the needed factorizations (by Theorem \ref{THEOREM_non-regular}) but, as we will show, their classes of regular epimorphisms also fail to be stable under pullback.

\begin{thm}
	$\Gmes$ and $\Games_{A}$ are coregular categories.
\end{thm}
\begin{proof}
	We carry out the proof for $\Games_{\A}$. The claim about $\Gmes$ then follows by ignoring the payoff sets from games.
	Since the cocomplete category  $\Games_\A$ has (epi, regular mono)-factorizations, with the regular monomorphism given by game embeddings, it remains to be shown that the pushout of an embedding along any A-morphism is again an embedding.
	
	For games $G=(T,A)$, $G_1 = (T_1,A_1)$ and $G_2 = (T_2,A_2)$, let $G_{1} \stackrel{f}{\leftarrow} G\stackrel{m}{\hookrightarrow} G_2$ be morphisms,  with $m$ a game embedding. As explained at the end of Subsection \ref{SUBSEC_Coeq_Pushouts}, the pushout of such diagram can be expressed as $G_1 \stackrel{q\circ i_1}{\longrightarrow}G_1{\sqcup}_{G} G_2\stackrel{q\circ i_2}{\longleftarrow}G_2$, where $G_1\sqcup G_2\stackrel{q}{\to}G_1{\sqcup}_{G} G_2$ is the coequalizer of $\begin{tikzcd}[column sep=normal,row sep=huge]
		G \arrow[r,shift left=.75ex,"i_1\circ m"] \arrow[r,shift right=.75ex,"i_2\circ f"{below}]  &
		G_1\sqcup G_2,
	\end{tikzcd}$
	with $G_1\stackrel{i_1}{\to}G_1\sqcup G_2$ and $G_2\stackrel{i_2}{\to}G_1\sqcup G_2$ the coproduct injections.
	The equivalence relation $\sim$ on $T_1\sqcup T_2$ satisfying the properties that
	\begin{itemize}
		\item if $t\in T_1$, then the $\sim$-equivalence class of $t$ is $[t] = \{t\}\sqcup m[f^{-1}(t)]$,
		\item and otherwise, if $t\in T_2 \setminus m[f^{-1}(T_1)]$, then $[t] = \{t\}$,
	\end{itemize}
	is the minimal equivalence relation on $T_1\sqcup T_2$ such that $f(t)\sim m(t)$ for every $t\in T$. So, as described in Section \ref{SUBSEC_Coeq_Pushouts}, we can take $G_1\sqcup_G G_2= (T_1\sqcup T_2/_\sim, A_1\sqcup A_2/_\sim )$ and $q\colon G_1\sqcup G_2\to G_1\sqcup_G G_2$ to be given by $q(t)=\seq{[t\restrict i]:i<|t|}$.
	
	By the definition of $\sim$, if $t,s\in T_1$, then $q(t)=[t]\neq [s]=q(s)$, so $i_1\circ q$ is injective. Thus, it only remains to be shown that $i_1\circ q$ is an $\B$-morphism.
	Let $R\in \Run(T_1)$ and suppose that $\overline{i_1\circ q}(R)\in A_1\sqcup A_2/_\sim$. Then, by definition of $A_1\sqcup A_2/_\sim$, there is an $S\in \Run(A_1\sqcup A_2)$ such that $\overline{q}(S)=\overline{i_1\circ q}(R_1)$. If $S\in A_1$, then $R=S$ (because $q$ is injective when restricted to $T_1$). On the other hand, note that $R\restrict n\in T_1$ and $S\restrict n\in T_2$ are $\sim$ equivalent for every $n<\omega$. According to the definition of $\sim$, this means that for every $n<\omega$, there is $t_n\in T$ such that $f(t_n)=R\restrict n$ and $m(t_n)=S\restrict n$. Now, because $m$ is injective, each $t_{n+1}$ extends its predecessor $t_n$, so there is an infinite sequence $S'\in \Run(T)$ extending every $t_n$. Since $m$ is an embedding and $\overline{m}(S') = S\in A_2$, we have $S'\in A$. But $\overline{f}(S')=R$, so $R\in A_1$. This completes the proof for $\Games_{A}$.
\end{proof}

In any category (with pullbacks), a pullback-stable regular epimorphism (so that its pullback along any morphism with the same codomain is again a regular epimorphism) is often called a {\em descent morphism}. In pursuit of the characterization of descent morphisms in the categories $\Gmes$ and $\Games_{\A}$, we define:  

\begin{defn}
	A morphism $T\stackrel{q}{\to}Q$ in $\Gmes$ is a {\em  strict quotient map} 
	if $\overline{q}\colon \Run(T)\to \Run(Q)$ is surjective and, for all distinct $R,R'\in \Run(Q)$, there are $S,S'\in \Run(T)$ such that $\overline{q}(S)=R$, $\overline{q}(S') = R'$ and $\Delta(S,S') = \Delta(R,R')$.	
\end{defn}

Just invoking the relevant definitions, one immediately sees that strict quotient maps are in fact quotient maps. However, there are quotient maps which fail to be strict, as we show next.

\begin{ex}\label{EX_RegEpic_not_Porj}
	For  $k\in \{0,1\}$, we let
	\[
	T_k = \set{\seq{\ast_k: i<n}:n<\omega}\cup\set{\seq{\ast_k}^\smallfrown\seq{k:i<n}:n<\omega}
	\]
	and then put $T = T_0\sqcup T_1$. Furthermore, we let
	\[
	C = \set{\seq{\ast: i<n}:n<\omega}\cup\set{\seq{\ast}^\smallfrown\seq{k:i<n}:n<\omega, k\in \{0,1\}}
	\]
	and define $c\colon T\to C$ by
	\[
	c(t) = \begin{cases}
		\seq{\ast:i<n} \text{ if $t = \seq{\ast_k:i<n}$},\\
		\seq{\ast}^\smallfrown\seq{k:i<n} \text{ if $t = \seq{\ast_k}^\smallfrown\seq{k:i<n}$}.
	\end{cases}
	\]
	This is clearly a surjective chronological map. In order to show that $c$ is a quotient map, we consider distinct  elements $t,s\in T$ with $c(t)=c(s)$. Then, for some $n<\omega$, we must have  $t = \seq{\ast_k:i<n}$ and $s = \seq{\ast_{1-k}:i<n}$, and the runs
	\[
	R_0 = \seq{\ast_k:i<\omega} \text{ and }
	R_1 = \seq{\ast_{1-k}:i<\omega}
	\]
	are such that $R_0\restrict n = t$, $R_1\restrict n = s$ and $\overline{c}(R_0)=\seq{\ast:i<\omega}=\overline{c}(R_1)$. This confirms that $c$ is a quotient map.
	
	However, for the runs $R = \seq{\ast}^\smallfrown \seq{0:n<\omega}$  and $R' =  \seq{\ast}^\smallfrown \seq{1:n<\omega}$ in $C$ one has $\overline{c}^{-1}(R) = \{S\}$ and $\overline{c}^{-1}(R') = \{S'\}$, where 
	\[
	S = \seq{\ast_0}^\smallfrown \seq{0:n<\omega} \text{ and }
	S' = \seq{\ast_1}^\smallfrown \seq{1:n<\omega}.
	\]
	Since $\Delta(S,S') = 0 < 1 = \Delta(R,R')$, we conclude that $c$ is not a strict quotient map..
\end{ex}

\begin{thm}\label{THM_Proj_PullbackStable}
	The strict quotient maps are precisely the descent morphisms of the category $\Gmes$. 
\end{thm}
\begin{proof}
	For a strict quotient map $T_1\stackrel{e}{\to} T$ and any chronological map $T_2\stackrel{f}{\to} T$, according to Proposition \ref{PROP_Gme_RegEpis} we should show that that the pullback projection $p_2: P=T_1\times_T T_2\to T_2$ is a (strict) quotient map, in order to conclude that $e$ is a descent morphism. Here we present the pullback $T_1\stackrel{p_1}{\leftarrow}P \stackrel{p_2}{\to} T_2$  of $T_1\stackrel{e}{\to}T\stackrel{f}{\leftarrow}T_2$ as the equalizer $P \stackrel{i}{\to} \Multi_\A(G_1,G_2)$ of the pair $f\circ\pi_1, e\circ \pi_2$, putting $p_1=\pi_1\circ i, p_2=\pi_2\circ i$ with the product projections $\pi_1,\pi_2$.
	
	Now let $R_2\in \Run(T_2)$. Then there is $R_1\in \Run(T_1)$ such that $\overline{e}(R_1)=\overline{f}(R_2)$ (because $\overline{e}$ is surjective), so $(R_1,R_2)\in \Run(P)$ and $\overline{p_2}(R_1,R_2)=R_2$. Hence, $\overline{p_2}$ is surjective. 
	
	Consider distinct $R_2,S_2\in \Run(T_2)$. If $\overline{f}(R_2) = \overline{f}(S_2)=R$, then there is  $R_1\in T_1$ such that $\overline{e}(R_1)=R$ and, thus, $(R_1, R_2), (R_1,S_2)\in \Run(T')$ satisfy $\Delta((R_1,R_2),(R_1,S_2)) = \Delta(R_2,S_2)$. So, suppose that $R = \overline{f}(R_2) \neq \overline{f}(S_2) = S$. Then, because $e$ is a descent morphism, there are $R_1,S_1\in \Run(T_1)$ such that
	\begin{gather*}
		\overline{e}(R_1) = R,\;\;
		\overline{e}(S_1) = S,\\
		\Delta(R_1,S_1) = \Delta(R,S) = \Delta(\overline{f}(R_2), \overline{f}(S_2)) \ge \Delta(R_2,S_2).
	\end{gather*}
	Therefore, for $(R_1, R_2), (S_1,S_2)\in \Run(P)$ one has
	\[
	\Delta(\overline{p_2}(R_1,R_2), \overline{p_2}(S_1,S_2)) = \Delta(R_2,S_2)  = \min\{\Delta(R_1,S_1),\Delta(R_2,S_2)\} = \Delta((R_1,R_2),(S_1,S_2)),
	\]
	and we conclude that $p_2$ is a descent morphism.
	
	Conversely, in order to show that every descent morphism $T\stackrel{q}{\to}Q$ in $\Gmes$ is a  strict quotient map, by Proposition \ref{PROP_Gme_RegEpis} it suffices that, assuming its failure to be a strict quotient map, we present a pullback of $q$ which is not a quotient map. First, if $\overline{q}$ is not surjective, with $R\in \Run(Q)$ not in the image of $\overline{q}$, with the terminal game $\mathbf 1$ (of Example \ref{EX_boring}) one may consider $f\colon \mathbf{1}\to Q$ with $\overline{f}(\seq{\ast:n<\omega}) = R$. Then  $T\stackrel{p_1}{\leftarrow}\emptyset \stackrel{p_2}{\to} \mathbf{1}$ is the pullback of $T\stackrel{q}{\to}Q\stackrel{f}{\leftarrow}\mathbf{1}$ and it is clear that $p_2$ is not even an epimorphism.
	
	If $\overline{q}$ is surjective, its failure to be a strict quotient map lets us fix $R_0,R_1\in \Run(Q)$ such that 
	\begin{equation}\label{EQ_RegEpi_not_Proj}
		N:= \Delta(R_0,R_1) > \Delta(S_0,S_1)
	\end{equation}
	for all $S_0,S_1\in \Run(T)$ with $\overline{q}(S_k)=R_k$. For the game tree 
	\[
	X = \set{\seq{\ast:i<N}^\smallfrown \seq{x:j<n}:n<\omega, x\in \{0,1\}}\cup \set{\seq{\ast:i<m}:m<N}
	\]
	we now define $f\colon X\to Q$ by
	\[
	f(t) = \begin{cases}
		R_0\restrict |t|, \text{ if $|t|>N$ and $t(N)=0$,}\\
		R_1\restrict |t|, \text{ otherwise.}
	\end{cases}
	\]
	In this case, $f$ is chronological and the pullback $T\stackrel{p_1}{\leftarrow}P \stackrel{p_2}{\to} X$  of $T\stackrel{q}{\to}Q\stackrel{f}{\leftarrow}X$ can be expressed as
	\[
	P = \set{S\restrict n: n<\omega, S\in \Run(T) \left( \overline{q}(S) = R_k \text{ for some } k\in\{0,1\}\right)},
	\]
	with $p_1$ the inclusion of $X$ into $T$ and $p_2$ defined by
	\[
	\overline{p_2}(S) = 
	\seq{\ast:i<N}^\smallfrown \seq{k:n<\omega}, \text{ whenever $\overline{q}(S)=R_k$,}
	\]
	
	Indeed, suppose that $T\stackrel{p_1'}{\leftarrow}Z \stackrel{p_2'}{\to} X$ such that $q\circ p_1' = f\circ p_2'$ is given. Note that for all $t\in Z$ and for all $S\in \Run(Z)$ extending $t$, $\overline{q}(\overline{p_1'}(S)) = R_k$ for some $k\in\{0,1\}$. In this case, the mapping $h\colon Z\to P$ such that $h(t) = p_1'(t)$ for all $t\in Z$ is well defined (and it is clear that such $h$ is chronological and the unique mapping such that $p_1' =p_1\circ h$).
	
	Furthermore, if $S\in \Run(Z)$, then
	\[
	\overline{q}(\overline{h}(S)) = \overline{q}(\overline{p_1'}(S)) = \overline{f}(\overline{p_2'}(S)) = R_k.
	\]
	Thus, $\overline{p_2'}(S) = \overline{p_2\circ h}(S) = \seq{\ast:i<N}^\smallfrown \seq{k:n<\omega}$, so $p_2' = p_2\circ h$ and $T\stackrel{p_1}{\leftarrow}P \stackrel{p_2}{\to} X$ is the pullback of $T\stackrel{q}{\to}Q\stackrel{f}{\leftarrow}X$.
	
	We will show that $p_2$ is not a quotient map (which, in view of Proposition \ref{PROP_Gme_RegEpis}, concludes the proof). Let $S,S'\in \Run(X)\subseteq \Run(T)$ be such that $\overline{q}(S) = R$ and $\overline{q}(S')=R'$. Note that $p_2(S\restrict N) = \seq{\ast:i<N} = p_2(S\restrict N)$, and, by \eqref{EQ_RegEpi_not_Proj}, $S\restrict N \neq S'\restrict N$. But, by \eqref{EQ_RegEpi_not_Proj} again, for all $S_0,S_1\in \Run(X)$, $\Delta(S_0,S_1)\ge N$ implies that 
	\[
	\overline{p_2}(S_0) = \overline{q}(S_0) = \overline{q}(S_1) = \overline{p_2}(S_1),
	\]
	Thus a sequence $\seq{S_0, \dotsc, S_k}$ of runs of $X$ such that 
	\[
	S\restrict N = S_0\restrict N, \; \Delta(S_{2i}, S_{2i+1})\ge N \text{  and  } \overline{p_2}(S_{2i+1}) = \overline{P_2}(S_{2i+2})
	\]
	for every $i < k/2$ will inevitably be such that $\overline{p_2}(S_j) = \overline{p_2}(S) = \seq{\ast:i<N}^\smallfrown \seq{0:n<\omega}$ for all $j\le k$. In particular, by \eqref{EQ_RegEpi_not_Proj} one more time, $S_k\restrict N \neq S'\restrict N$, so $p_2$ is not a quotient.
\end{proof}

With $\mathcal S$ denoting the class of strict quotient maps in $\Gmes$, in the notation of Lemma \ref{LEMMA_SUBOrtFactSyst} we now consider the class
\[
\mathcal{S}^* = \set{(T,A)\stackrel{p}{\to}_\A (T',A'): p\in\mathcal S,  A' = \overline{p}[A]}
\]
in $\Games_{\A}$ and prove:

\begin{cor}\label{THM_Game_Proj_PullbackStable}
	The class $\mathcal S^*$ is the class of descent morphisms in the category $\Games_{\A}$. 
\end{cor}
\begin{proof}
	Let the game morphism $G_1=(T_1,A_1)\stackrel{e}{\to}_\A (T,A)$ be in $\mathcal{S}^*$. For $G_2 = (T_2,A_2)\stackrel{f}{\to} (T,A)$, let $(T_1,A_1)\stackrel{p_1}{\leftarrow}(T',A') \stackrel{p_2}{\to} (T_2,A_2)$  be the pullback of $(T_1, A_1)\stackrel{e}{\to}(T,A)\stackrel{f}{\leftarrow}(T_2,A_2)$ in $\Games_{\A}$. Then  
	$T_1\stackrel{p_1}{\leftarrow}T' \stackrel{p_2}{\to} T_2$ is the pullback of $T_1\stackrel{e}{\to}T\stackrel{f}{\leftarrow}T_2$ in $\Gmes$, and it follows from Theorem \ref{THM_Proj_PullbackStable} that with $e$ also $p_2$ is a strict quotient map. In fact,  one has $p_2\in \mathcal{S}^*$ since, given $R\in A_2$ one has $S\in A_1$ with $\overline{e}(S) = \overline{f}(R)\in A$, so that  $(S,R)\in A'$ satisfies $\overline{p_2}(S,R) = R\in\overline{p_2}[A']$.
	
	Similarlyly, considering conversely a descent  morphism $(T,A)\stackrel{p}{\to}_\A(T',A')$ in $\Games_\A$, since the forgetful topological functor $\Games_{\A}\to \Gmes$ with its right and left adjoints must preserve not only regular epimorphisms and pullbacks, but also descent morphisms, $T\stackrel{p}{\to}T'$ is a strict quotient map by Theorem \ref{THM_Proj_PullbackStable}	satisfying $A' = \overline{p}[A]$ (by Remark \ref{REM_topfun_RegEpi_StrEpi}). This confirms $p\in\mathcal S^*$.	
\end{proof}

\begin{cor}\label{COR_NotRegular}
	In the categories $\Gmes$ and $\Games_{A}$, regular epimorphisms may fail to be stable under pullback.
\end{cor}
\begin{proof}
	Example \ref{EX_RegEpic_not_Porj} combined with Theorem \ref{THM_Proj_PullbackStable} settles the claim in $\Gmes$; the claim for $\Games_{\A}$ follows by providing the counterexample constructed in $\Gmes$ with empty pay-off sets. 
\end{proof}

\subsection{Exponentiation}
Since by Theorem \ref{THM_GmeUMet} the category  $\Gmes$ is equivalent to
the subcategory $\mathbf{SeqSpace}$ of $\CUMet$, in this section we will use the cartesian closedness of the latter category to immediately conclude the same property for the former and then extend it to the category $\Games_{\A}$, giving a full interpretation of how to play, and win, ``exponential games'', {\em i.e.,} internal hom objects in $\Games_{\A}$. 

First, describing the internal hom  for objects $X$ and $Y$ in $\mathbf{SeqSpa}$, we equip the set
\[ 
[X,Y] = \set{f\in Y^X: f \text{ is $1$-Lipschitz}}
\]
with the sup-metric of the product $Y^X$ and consider the evaluation map 
\begin{center}
	\begin{tikzcd}[row sep = 0em]
		[X,Y]\times X \arrow[r, "\mathrm{ev}"] &	Y \\
		(f,x) \arrow[r, mapsto] &   f(x).
	\end{tikzcd}
\end{center}
Clearly, the space $[X,Y]$  is ultrametric, with its distance function (like that of $Y$) ranging in the set $\{0\}\cup\set{\frac{1}{n+1}:n<\omega}$ and, as a closed subspace of the space $Y^X$ (which is complete, because completeness of $Y$ alone already implies that every Cauchy sequence of mappings from $X$ to $Y$ is pointwise convergent, and pointwise convergence, together with the Cauchy property, implies uniform convergence), it is also complete. Furthermore, one easily checks that the evaluation map presented above is $1$-Lipschitz and satisfies the required universal property: every 1-Lipschitz map $h: Z\times X\to Y$ (with $Z$ in $\mathbf{SeqSpa}$) factors through $\mathrm{ev}$ as $h=\mathrm{ev}\circ (g\times\mathrm{id}_X)$, for a unique 1-Lipschitz map $g:Z\to [X,Y]$. Therefore:

\begin{prop}
	The category $\Gmes$ is cartesian closed.
\end{prop}

With its exponentiation built on top of that of $\mathbf{SeqSpace}$, one routinely checks that also the category $\mathbf{MetGame}$ is cartesian closed. Indeed, given $(X,A_X)$ and $(Y,A_Y)$ in $\MetGame$, one puts
\[
A_{[X,Y]} = \set{f\in [X,Y]: f[A_X]\subset A_Y},
\]
so that $\mathrm{ev}(f,x)=f(x)\in A_Y$ for every $(f,x)\in A_{[X,Y]}\times A_X$. This makes not only the map $\mathrm{ev}\colon [X,Y]\times X\to Y$ a morphism in $\MetGame$, but also the map $g$ induced by $h$ as above, as long as $h$ lives in $\mathbf{MetGame}$.

So, in view of Theorem \ref{THM_UMetEqv}, we obtain:

\begin{cor}
	The category $\Games_\A$ is cartesian closed, and the payoff-forgetting functor $\mathrm{U}\colon \Games_{\A}\to \Gmes$ preserves exponentiation.
\end{cor}

Beyond its mere existence, in what follows we should give a direct game-theoretic description of the internal hom of any given games $G_1$ and $G_2$, denoted by $\Multi_\A^{G_1}(G_2)$. Roughly, this is the game in which $\ali$ and $\bob$ take turns constructing a chronological mapping from the tree of $G_1$ to the tree of $G_2$, so that the mapping is completely defined by the end of a run of the game. The winning criteria, then, is that $\ali$ wins the run if the constructed chronological mapping is an $\A$-morphism (and $\bob$ wins otherwise). More precisely:

\begin{defn}
	For games $G_1=(T_1,A_1)$ and $G_2=(T_2,A_2)$, the $\A$-\textit{exponential} game $\Multi_\A^{G_1}(G_2)$) proceeds as follows:
	\begin{itemize}
		\item In the first inning,
		\begin{itemize}
			\item $\ali$ chooses a mapping $f_0\colon T_1(1)\to T_2(1)$;
			\item $\bob$ responds by choosing a mapping $f_1\colon T_1(2)\to T_2(2)$ such that $f_1(t)\restrict 1=f_0(t\restrict 1)$ for every $t\in T_1(2)$.
		\end{itemize}
		\item In the n-th inning $n\ge 1$,
		\begin{itemize}
			\item $\ali$ chooses a mapping $f_{2n}\colon T_1(2n+1)\to T_2(2n+1)$ such that $f_{2n}(t)\restrict (2n)=f_{2n-1}(t\restrict (2n))$ for every $t\in T_1(2n+1)$;
			\item $\bob$ responds by choosing a mapping $f_{2n+1}\colon T_1(2n+2)\to T_2(2n+2)$ such that $f_{2n+1}(t)\restrict (2n+1)=f_{2n}(t\restrict (2n+1))$ for every $t\in T_1(2n+2)$.
		\end{itemize}
	\end{itemize}
	That is, a moment $s$ of the game $\Multi_\A^{G_1}(G_2)$ is a sequence of mappings $\seq{f_0, \dotsc, f_n}$ such that each mapping $f_k$ has $T_1 (k+1)$ as its domain and $f_j(t)\restrict i+1 = f_i(t\restrict i+1)$ for all $i\le j\le n$ and $t\in T(j+1)$. 
	$\ali$ is then said to have won the run $\seq{f_n:n<\omega}$ if the resulting chronological mapping $\bigcup_{n<\omega}f_n=f\colon T_1\to T_2$ (in set-theoretic notation) defined throughout the run is an $\A$-morphism.
	
	The \textit{evaluation morphism}
	\begin{center}
		\begin{tikzcd}[row sep = 0em]
			\Multi_\A^{G_1}(G_2)\times G_1 \arrow[r, "\mathrm{ev}"] &	G_2 \\
			(\seq{f_i:i\le n},t) \arrow[r, mapsto] &   f_n(t),
		\end{tikzcd}
	\end{center}
	is well-defined since $|t|=|\seq{f_i:i\le n}|=n+1$ (so that $t\in \dom(f_n)$). It is clearly chronological since the $f_k$s are constructed \textit{chronologically} throughout $\Multi_\A^{G_1}(G_2)$, and it is an $\A$-morphism since a run $(\seq{f_n:n<\omega}, R)$ in $\Multi_\A^{G_1}(G_2)\times G_1$ is won by $\ali$ in $\Multi_\A^{G_1}(G_2)\times G_1$ if, and only if, $f=\bigcup_{n<\omega}f_n$ is an $\A$-morphism and $R$ is won by $\ali$ in $G_1$, therefore $\overline{ev}(\seq{f_n:n<\omega}, R)=\overline{f}(R)\in A_2$.
\end{defn}
Let us indicate that the required universal property is satisfied and consider a game morphism
$G\times G_1\stackrel{h}{\to}G_2$, with $G=(T,A)$. Then, with 
\begin{center}
	\begin{tikzcd}[row sep = 0em]
		T_1 (|t|) \arrow[r, "f_t"] &	T_2 (|t|)\\
		s \arrow[r, mapsto] &   h(t ,s).
	\end{tikzcd}
\end{center}
for every $t\in T$, we obtain the moment
$\seq{f_{t\restrict 1},\dotsc, f_t}$ in $\Multi_\A^{G_1}(G_2)$. One checks routinely that
\begin{center}
	\begin{tikzcd}[row sep = 0em]
		G \arrow[r, "g"] &	\Multi_\A^{G_1}(G_2) \\
		t \arrow[r, mapsto] &   \seq{f_{t\restrict 1},\dotsc, f_t}.
	\end{tikzcd}
\end{center}
is in fact the desired unique $\A$-morphism $G\stackrel{g}{\to}\Multi_\A^{G_1}(G_2)$ making the following diagram commute:

\begin{center}
	\begin{tikzcd}[column sep=20mm, row sep=20mm]
		\Multi_\A^{G_1}(G_2)\times G_1 \arrow[r, "\mathrm{ev}"] &	G_2 \\
		G\times G_1 \arrow[u, "g\times \mathrm{id}_{G_1}"]\arrow[ur,"h"{auto,swap}]&
	\end{tikzcd}
\end{center}

We should give some justification for denoting the internal hom by $\Multi_\A^{G_1}(G_2)$, thus using a notation reminiscent of product games. 
Indeed, just like in $\Sets$, where the internal hom is given by the power $X^Y$ of $Y$-many copies of $X$, also in $\Games_{\A}$ the internal hom is closely related to a product game.  

Indeed, for the given games $G_1=(T_1,A_1)$ and $G_2 = (T_2,A_2)$, let $T'$ be the tree of the game $\Multi_\A((G_2)_{R\in \Run(T_1)})$ and then consider the game $(T,A)$ with
\begin{gather*}
	T=\set{t=\seq{(x_i^R)_{R\in\Run(T_1)}:i\le n}\in T': \seq{x_i^{R}:i\le n}=\seq{x_i^{R'}:i\le n} \text{ if $R\restrict n = R'\restrict n$}},\\
	A=\set{t=\seq{(x_n^R)_{R\in\Run(T_1)}: n<\omega}:  \seq{x_n^R:n<\omega}\in A_2 \text{ for every $R\in A_1$}}.
\end{gather*}
Now, denoting the exponential game $\Multi_\A^{G_1}(G_2)$ by $(\tilde{T},\tilde{A})$, we define the map $f\colon T\to \tilde{T}$ by
\[
f(\seq{(x_i^R)_{R\in\Run(T_1)}:i\le n})=\seq{f_i: i\le n},
\]
where $f_i(R\restrict (i+1))=\seq{x_j^{R}:j\le i}$ for every $i\le n$; here, the condition imposed on the moments of $T$ makes sure that the $f_i$s are well defined and chronological.
It is easy to check that $f$ is an isomorphism between games, with its inverse $g\colon \tilde{T}\to T$ given
\[
g(\seq{f_i: i\le n})=\seq{(x_i^R)_{R\in\Run(T_1)}:i\le n},
\]
where $\seq{x_j^{R}:j\le i}=f_i(R\restrict (i+1))$.

So, intuitively, $\Multi_\A^{G_1}(G_2)$ may also be seen as the game in which $\ali$ and $\bob$ play the game $G_2$ simultaneously on $\Run(T_1)$-many boards, in such a way that, if $R\restrict n = R'\restrict n $, then, up to $n$, they must make the same moves on the boards corresponding to $R$ and $R'$. Moreover, $\ali$ must win in every board relating to runs $R\in A_1$ in order to win in $\Multi_\A^{G_1}(G_2)$.

Other than asserting an important categorical property of our ludic categories, 
the exponential game $\Multi_\A^{G_1}(G_2)$ is interesting also from a game-theoretical point of view:

\begin{prop}
	Let $G_1=(T_1,A_1)$ and $G_2=(T_2,A_2)$ be games, with $A_1\neq\emptyset$. Then the following is true:
	\begin{itemize}
		\item[(a)] $\ali$ has a winning strategy in $\Multi_\A^{G_1}(G_2)$ if, and only if, $\ali$ has a winning strategy in $G_2$.
		\item[(b)] If $\bob$ has a winning strategy in $G_2$, then $\bob$ has a winning strategy in $\Multi_\A^{G_1}(G_2)$.
	\end{itemize}
\end{prop}
\begin{proof}
	Using the ``product' identification" of $\Multi_\A^{G_1}(G_2)$ as described above, we
	note that, if $\gamma$ is a winning strategy for $\ali$ in $G_2$, then $\ali$ can use $\gamma$ on every board of the game $\Multi_\A^{G_1}(G_2)$, so that she will win on every board (not just on the boards corresponding to $A_1$). In particular, doing so provides her with a winning strategy for $\Multi_\A^{G_1}(G_2)$. This proves the ``if''-part of (a), and  the proof of (b) is analogous, with the condition $A_1\neq \emptyset$ being necessary, as $\Multi_\A^{G_1}(G_2)$ would be trivial for $\ali$ otherwise.
	
	Now suppose that $\ali$ has a winning strategy $\gamma$ in $\Multi_\A^{G_1}(G_2)$ and fix $R\in A_1$. Then $\ali$ can use the responses which $\gamma$ instructs her to use on the $R$-th board as a strategy in $G_2$ (filling anything on the other boards for $\bob$), so that the fact that $\gamma$ wins at every board $R'\in A_1$ implies that doing so makes $\ali$ win in $G_2$.
\end{proof}

So the cases in which $G_2$ is determined are trivial for $\Multi_\A^{G_1}(G_2)$ (in a game-theoretical sense), and $\Multi_\A^{G_1}(G_2)$ is equivalent to $G_2$ for $\ali$ -- much like $\Multi_\A((G_2)_{R\in \Run(T_1)})$. Therefore, $\Multi_\A^{G_1}(G_2)$ can possibly be seen as a way to make the undetermined game $G_2$ \textit{easier} for $\bob$, so that he may possibly then have a winning strategy.

Hence, we propose a natural question:

\begin{prob}
	Which (non-trivial) conditions on an undetermined game $G_2$ and/or $G_1$ guarantee the existence of a winning strategy for $\bob$ in the game $\Multi_\A^{G_1}(G_2)$?
\end{prob}

The category $\Sets$ and, in fact, every (quasi)topos (as considered in the next section), enjoy a stronger property than cartesian closedness, called local cartesian closednes. What about $\Games_{\A}$? We give a negative answer, but first recall:
\begin{defn}\label{DEF_LocalCC}
	A finitely complete category $\mathbf{C}$ is \emph{locally cartesian closed} if its comma categories $\mathbf{C}/Z$ are cartesian closed, for all objects $Z$ in $\mathbf{C}$. 
	
	This equivalently means \cite{Dyckhoff1987} that, given any morphism $X\stackrel{f}{\to} Z$ and any object $Y$ in $\mathbf{C}$, there are morphisms $P\stackrel{p}{\to}Z$ and $P\times_Z X\stackrel{\varepsilon}{\to} Y$ such that, for all $Q\stackrel{q}{\to}Z$ and $Q\times_Z X\stackrel{g}{\to} Y$, there is a unique $Q\stackrel{h}{\to}P$ such that $p\circ h = q$ and $\varepsilon\circ(h\times_Z \id_X)=g$. 
\end{defn}

Considering for $Z$ the terminal object  in $\mathbf C$ one sees that locally cartesian closedness implies cartesian closedness, but the converse implication fails in many categories, including our ludic categories. We first show: 

\begin{prop}\label{PROP_GMEnotLCC}
	$\Gmes$ is not locally cartesian closed.
\end{prop}
\begin{proof}
	By  Theorem \ref{THM_GmeUMet} it suffices to show that $\mathbf{SeqSpa}$ is not locally cartesian closed. 
	We consider any object\footnote{Such object is described explicitly in the proof of Proposition \ref{PROP_noClassifiers}.} 
	$Z=\set{z_n:n<\omega}\sqcup\{\bar{z}\}$ in $\mathbf{SeqSpa}$ given by a non-trivially converging sequence $z_n\stackrel{n\to\infty}{\longrightarrow}\bar{z}$, and let $X$ be the discrete metric space with the same underlying set as $Z$. 
	Now suppose that for the discrete metric space $Y=\{0,1\}$ and the (1-Lipschitz) identity map 
	$X\stackrel{f}{\to}Z$ we would have $P, p, \varepsilon$ satisfying the universal property described in the Definition.
	
	First, considering for $Q$ the terminal object in $\mathbf{SeqSpace}$ 
	we see that, for every $z\in Z$ and $i\in \{0,1\}$, there is a unique $x_z^i\in P$ such that $p(x_z^i)=z$ and $\varepsilon(x_z^i,z)=i$. 
	Next, considering for every $n<\omega$ the (two-point) subspace $Q_n=\{z_n,\bar{z}\}$ of $Z$ and the inclusion map $q_n\colon Q_n\to Z$, one obtains the discrete subspace $Q_n\times_Z X \cong \{z_n,\bar{z}\}$ of $X$. 
	Then, by the universal property, the 1-Lipschitz map $g_n\colon Q\times_Z X\to Y$ with $g_n(z_n)=0$ and $g_n(\bar{z})=0$ induces the 1-Lipschitz map $h_n\colon Q_n\to P$ which, by the uniqueness property stated in our previous choice for $Q$, must satisfy 
	$h_n(z_n)=x_{z_n}^0$ and $h_n(\bar{z})=x_{\bar{z}}^0$, so that 
	$d_P(x_{z_n}^0, x_{\bar{z}}^ 0)\le d_{Q_n}(z_n,\bar{z})=d_Z(z_n,\bar{z})$ follows.
	
	Likewise, instead of $g_n$ we may also consider the map $g_n'\colon Q\times_Z X\to Y$ with  $g_n'(z_n)=0$ and $g_n'(\bar{z})=1$ and, following the same argumentation as before, conclude that also $d_P(x_{z_n}^0, x_{\bar{z}}^1) \le d_Z(z_n,\bar{z})$ holds. Consequently, the sequence
	$(x_{z_n}^0)_{n<\omega}$ must converge in $P$ to both $x_{\bar{z}}^0$ and $x_{\bar{z}}^1$. But this is impossible since $\varepsilon(x^0_{\bar{z}},\bar{z})=0\neq 1=\varepsilon(x^1_{\bar{z}},\bar{z})$ forces $x_{\bar{z}}^0\neq x_{\bar{z}}^1$.
\end{proof}

\begin{prop}
	$\Games_{\A}$ is not locally cartesian closed.
\end{prop}
\begin{proof}
	Equipping the objects $Z,X,Y\in\mathbf{SeqSpa}$ as in the proof of Proposition \ref{PROP_GMEnotLCC} with their maximal ``pay-off sets'', so that $(Z,Z), (X,X),(Y,Y)\in\MetGame$, one sees immediately that $\MetGame$ cannot be locally cartesian closed either. By 
	Proposition \ref{PROP_GMEnotLCC}, this completes the proof also here. 
\end{proof}

\subsection{Classifiers}
Recall that a \emph{quasitopos} is a locally cartesian closed category with a classifier for strong subobjects, as defined below.
Not being locally cartesion closed, neither $\Gmes$ nor $\Games_{\A}$ can be a quasitopos. But what about the other defining property of a quasitopos? We will show that, although these categories do not have a classifier for strong subobjects either, each one does possess an object satisfying the classifying property {\em weakly}, whereby this property is stripped of its uniqueness requirement. 
In fact, more generally, we will see that they have weak classifiers for all strong partial morphisms, as defined below.
\begin{defn}\label{DEF_StrSubClass}
	A \emph{classifier for strong subobjects} in a finitely complete category $\mathbf{C}$ is an object $\Omega$ equipped with a \emph{truth morphism} $1\stackrel{\tau}{\to}\Omega$ from $\mathbf{C}$'s terminal object such that, for every strong monomorphism $A\stackrel{m}{\to}X$, there is a unique morphism $X\stackrel{\chi}{\to}\Omega$ such that 
	\begin{center}
		\begin{tikzcd}[column sep=15mm, row sep=15mm]
			A\arrow[r, "!"] \arrow[d, "m"]&	1 \arrow[d,"\tau"]\\
			X \arrow[r, "\chi"] & \Omega.
		\end{tikzcd}
	\end{center}
	is a pullback diagram.
\end{defn}

\begin{prop}\label{PROP_noClassifiers}
	The categories $\Gmes$ and $\Games_{\A}$ do not have a classifier for strong subobjects.
\end{prop}
\begin{proof}
	By Theorems \ref{THM_GmeUMet} and \ref{THM_UMetEqv}, it suffices to show the absence of a classifier for strong subobjects in the categories $\mathbf{SeqSpa}$ and $\MetGame$. 
	Suppose we had such classifier $\{\ast\}\stackrel{\tau}{\to}\Omega$ in $\mathbf{SeqSpa}$. We metricize the set $X=\set{x_n:n<\omega}\sqcup\{\bar{x}\}$ of pairwise distinct elements by $d_X(\bar{x},x_n) = \frac{1}{n+1}$ and $d_X(x_n,x_m) = \max\{\frac{1}{n+1}, \frac{1}{m+1}\}$ for all $m,n<\omega$, 
	to obtain the object $X\in\mathbf{SeqSpa}$. 
	Its strong subobject $\{\bar{x}\}\hookrightarrow X$ then comes with the``characteristic morphism'' $X\stackrel{\chi}{\to}\Omega$ producing the pullback diagram
	\begin{center}
		\begin{tikzcd}[column sep=15mm, row sep=15mm]
			\{\bar{x}\}\arrow[r, "!"] \arrow[d, ""]&	\{\ast\} \arrow[d,"\tau"]\\
			X \arrow[r, "\chi"] & \Omega .
		\end{tikzcd}
	\end{center}
	Hence, $\chi^{-1}(\chi(\bar{x}))=\{\bar{x}\}$, which implies $\chi(x_n)\neq\chi(\bar{x})$ for all $n<\omega$. Furthermore, since $x_n\stackrel{n\to\infty}{\longrightarrow}\bar{x}$ in $X$, one has $\chi(x_n)\stackrel{n\to\infty}{\longrightarrow}\chi(\bar{x})$.

	But there is another morphism, $\chi'$, which fits into the above pullback diagram in lieu of $\chi$. Indeed, the map
	\begin{center}
		\begin{tikzcd}[row sep = 0em]
			X \arrow[r, "\chi'"] &	\Omega\\
			x_n \arrow[r, mapsto] &   \chi(x_{n+1}),\\
			\bar{x} \arrow[r, mapsto] &   \chi(\bar{x}).
		\end{tikzcd}
	\end{center}
	is clearly 1-Lipschitz and satisfies $(\chi')^{-1}(\chi'(\bar{x}))=\{\bar{x}\}$, but it is distinct from $\chi$ since, otherwise, $\chi$ would be constant on $\{x_n:n<\omega\}$ and thus preventing $(\chi(x_n))_{n<\omega}$ to converge to $\chi(\bar{x})$ in $\Omega$.
	
	This concludes the non-existence proof for $\mathbf{SeqSpa}$. The claim for $\MetGame$ follows (just as in the proof of Proposition \ref{PROP_GMEnotLCC}) when we choose the ``pay-off sets'' to be maximal. 
\end{proof}

Note that, in the above proof, we relied exclusively on the non-uniqueness of the characteristic morphism $\chi$. Discarding that condition in Definition \ref{DEF_StrSubClass} and thus arriving at the notion of {\em weak classifier for strong subobjects} would indeed allow us to show the existence of such objects in $\Gmes$ and $\Games$, as we will prove more generally now.

\begin{defn}\label{DEF_weakClassifier}
	A \emph{weak classifier for strong partial maps into $B$} in a finitely complete category $\mathbf{C}$ with $B\in \Obj{\mathbf{C}}$ is a morphism $B\stackrel{i_B}{\to}B_{\perp}$ such that, for every span $X\stackrel{m}{\leftarrow} A\stackrel{f}{\to}B$ with a strong monomorphism $m$, there is a morphism $X\stackrel{f_\perp}{\to}B_\perp$ such that the following diagram is a pullback:
	\begin{center}
		\begin{tikzcd}[column sep=15mm, row sep=15mm]
			A\arrow[r, "f"] \arrow[d, "m"]&	B \arrow[d,"i_B"]\\
			X \arrow[r, "f_\perp"] & B_\perp.
		\end{tikzcd}\label{DIAG_weakClassifier}
	\end{center}
	(The notion of weak classifier for strong subobjects emerges when we specialize $B$ to be the terminal object $1$, thus setting $\Omega=1_{\bot}$.) We say that $\mathbf{C}$ {\em has weak classifiers for strong partial maps} if a weak classifier for strong partial maps into $B$ exists for all objects $B$ in $\mathbf{C}$.
\end{defn}

\begin{thm}\label{THM_GamesClassifier}
	The categories $\Gmes$ and $\Games_{\A}$ have weak classifiers for strong partial maps.
\end{thm}
\begin{proof}
	For every game tree $T$ in $\Gmes$ we let $\ast_T$ be an element outside $\M(T)$ and put
	\[
	T_\perp = \set{t^\smallfrown \seq{\ast_T : i<n}:t\in T,n<\omega}.
	\]
	Being injective, the inclusion map $i_T\colon T\to T_{\perp}$ is trivially a strong monomorphism in $\Gmes$ (by Proposition \ref{PROP_Gmes_RegMon}).
	Given a span $\tilde{S}\stackrel{m}{\leftarrow} S\stackrel{f}{\to}T$ with a strong monomorphism $m$, 
	for every $s\in \tilde{S}$ we let $n_s=\max\{n\le |s|: s\restrict n\in m[S]\}$ and consider the map
	\begin{center}
		\begin{tikzcd}[row sep = 0em]
			\tilde{S} \arrow[r, "f_\perp"] &	T_\perp\\
			s \arrow[r, mapsto] &  	f(s\restrict n_s)^\smallfrown \seq{\ast_T:i<|s|-n_s }.
		\end{tikzcd}
	\end{center}
	Clearly, $f_\perp$ is chronological. Moreover, $f_\perp(m(s))=i_T(f(s))$ for every $s\in S$ and $m[S]=f_\perp^{-1}(i_T[T])$, so we obtain the pullback diagram
	\begin{center}
		\begin{tikzcd}[column sep=15mm, row sep=15mm]
			S\arrow[r, "f"] \arrow[d, "m"]&	T \arrow[d,"i_T"]\\
			\tilde{S} \arrow[r, "f_\perp"] & T_\perp.
		\end{tikzcd}
	\end{center}
	
	This completes the proof of the claim for $\Gmes$. Augmenting it to the level of games, for $G = (T,A)$ in $\Games_{\A}$ we consider $G_\perp = (T_\perp, A_\perp)$ with
	\[
	A_\perp = A\cup \set{\Run(T_\perp)\setminus \Run(T)}.
	\]
	In this way $i_T\colon T\to T_{\perp}$ becomes a game embedding, {\em i.e.}, a strong monomorphism in $\Games_{\A}$.
	Considering any span
	$(\tilde{S},\tilde{A})\stackrel{m}{\leftarrow} (S,\overline{m}^{-1}[\tilde{A}])\stackrel{f}{\to}(T,A)$ in $\Games_{\A}$ with $m$ injective, 
	the map $f_\perp: \tilde{S}\to T_\perp$ becomes an $\A$-morphism. 
	Indeed, given any $R\in \tilde{A}$, if $R\notin\Run(m[S])$, then $\overline{f_\perp}(R)\in A_\perp$, and if $R\in \Run(m[S])$, since $m$ is a game embedding, we find  $R'\in \Run(S)$ with  $R=\overline{m}(R')$. Because $f$ and $i_T$ are $\A$-morphisms, in this case $\overline{f_\perp}(R) = \overline{i_T}(\overline{f}(R'))\in A_\perp$ follows, which makes 
	$f_\perp$ an $\A$-morphism. One concludes the argument for $\Games_{\A}$ just like for $\Gmes$.
\end{proof}


\begin{thebibliography}{10}
		
		\bibitem{Adamek1990}
		J.~Ad\'{a}mek, H.~Herrlich, G.E.~Strecker.
		\newblock {\em Abstract and Concrete Categories}.
		\newblock {John Wiley \& Sons, New York, 1990.}
		
		\bibitem{Brandenburg2018}
		M.~Brandenburg.
		\newblock {Algebraic games—playing with groups and rings}.
		\newblock {\em International Journal of Game Theory}, 47(2):417--450, 2018.
		
		\bibitem{Birkedal2012}
		L.~Birkedal, R.E~Møgelberg, J.~Schwinghammer, K.~Støvring. 
		\newblock{First steps in synthetic guarded domain theory: step-indexing in the topos of trees.}
		\newblock{\em  Logical Methods in Computer Science} 8(4:1):1--45, 2012.  
		
		\bibitem{Bourn2004}
		D.~Bourn, M.~Gran.
		\newblock{Regular, protomodular, and abelian categories.}
		\newblock{In: M.C.~Pedicchio and W.Tholen (editors), {\em Categorical Foundations}}, pages 165--211, Cambridge University Press, Cambridge, 2004.
		
		\bibitem{Carboni1993}
		A.~Carboni, S.~Lack, R.F.C.~Walters.
		\newblock{Introduction to extensive and distributive categories.}
		\newblock{\em Journal of Pure and Applied Algebra,} 84:145--158, 1993.
		
		\bibitem{Centazzo2004}
		C.~Centazzo, E.~Vitale.
		\newblock{Sheaf Theory.}
		\newblock{In: M.C.~Pedicchio and W.Tholen (editors), {\em Categorical Foundations}}, pages 311--357, Cambridge University Press, Cambridge, 2004.
		
		\bibitem{Dyckhoff1987}
		R.~Dyckhoff, W.~Tholen.
		\newblock{Exponentiable morphisms, partial products, and pullback complements}
		\newblock{\em Journal of Pure and Applied Algebra} 49:103--116, 1987.
		
		\bibitem{Ferenczi2009}
		V.~Ferenczi and C.~Rosendal. 
		\newblock {Banach spaces without minimal subspaces}.
		\newblock {\em Journal of Functional Analysis} 257(1):149--193, 2009.
		
		\bibitem{Galvin2016}
		F.~Galvin and M.~Scheepers.
		\newblock {Baire spaces and infinite games}.
		\newblock {\em Arch. Math. Logic}, 55(1-2):85--104, 2016.
		
		\bibitem{Goubault2008}
		J.~Goubault-Larrecq, S.~Lasota, D.~Nowak.
		\newblock{Logical relations for monadic types.}
		\newblock{\em Mathematical Structures in Computer Science},18(6):1169--1217, 2008.
		
		
		\bibitem{Hofmann2014} 
		D.~Hofmann, G.J.~Seal and W.~Tholen (editors). 
		\newblock{\em Monoidal Topology. A Categorical Approach to Order, Metric, and Topology}.
		\newblock Cambridge University Press, Cambridge, 2014.
		
		\bibitem{Krom1974}
		M.~R. Krom.
		\newblock {Cartesian Products of Metric Baire Spaces}.
		\newblock {\em Proc. Amer. Math. Soc.}, 42(2):588--594, 1974.
		
		\bibitem{Lucatelli2023}
		F.~Lucatelli-Nunes, M.~V\'{a}k\'{a}r.	
		\newblock{CHAD for expressive total languages.}
		\newblock{\em Mathematical Structures in Computer Science}, to appear (published online 14 July 2023).	
		
		\bibitem{MacLane1992}
		S.~Mac Lane, I.~ Moerdijk.
		\newblock{\em Sheaves in Geometry and Logic. A First Introduction to Topos Theory.}
		\newblock{Springer-Verlag, NewYork, 1992.}
		
		\bibitem{Oxtoby1957}
		J.~C. Oxtoby.
		\newblock {The Banach-Mazur game and Banach category theorem}.
		\newblock In: {\em Contributions to the Theory of Games. Vol. III}, pages
		159--163. Princeton University Press, 1957.
		
		\bibitem{Pitz2023}
		M.~Pitz.
		\newblock {Characterising path-, ray- and branch spaces of order trees, and end spaces of infinite graphs}.
		\newblock {\em arXiv}, pages 1--17, 2023.
		
		\bibitem{Scheepers1997}
		M.~Scheepers.
		\newblock {Combinatorics of open covers (III): Games, $C_p(X)$}.
		\newblock {\em Fundamenta Mathematica}, 152(3):231--254, 1997.
		
		\bibitem{Scheepers2014}
		M.~Scheepers.
		\newblock {Remarks on countable tightness}.
		\newblock {\em Topology and its Applications}, 161(1):407--432, 2014.
		
		\bibitem{Streufert2018}
		P.~A. Streufert.
		\newblock {The category of node-and-choice preforms for extensive-form games}.
		\newblock {\em Studia Logica}, 106(5):1001--1064, 2018.
		
		\bibitem{Streufert2020}
		P.~A. Streufert.
		\newblock {The category of node-and-choice extensive-form games}.
		\newblock {\em arXiv}, pages 1--49, 2020.
		
		\bibitem{Streufert2021}
		P.~A. Streufert.
		\newblock {The category for extensive-form games}.
		\newblock {\em arXiv}, pages 1--60, 2021.
		
		\bibitem{Telgarsky1987}
		R.~Telg{\'{a}}rsky.
		\newblock {Topological games: on the 50th anniversary of the Banach-Mazur
			game}.
		\newblock {\em Rocky Mountain Journal of Mathematics}, 17(2):227--276, 1987.
		
		\bibitem{Tholen1974}
		W.~Tholen.
		\newblock{\em Relative Bildzerlegungen und algebraische Kategorien}.
		\newblock Doctoral thesis, Westf\"{a}lische Wilhelms-Universit\"{a}t, M\"{u}nster, 1974.
		
		\bibitem{Tholen1978}
		W.~Tholen.
		\newblock{On Wyler's Taut Lift Theorem.}
		\newblock{\em General  Topology and its Applications,} 8:197--206, 1978.
		
		\bibitem{Wyler1971a}
		O.~Wyler.
		\newblock{ On the categories of general topology and topological algebra.}
		\newblock{\em Archiv der Mathematik,} 22:7--17, 1971.
		
		\bibitem{Wyler1971b}
		O.~Wyler.
		\newblock{Top categories and categorical topology.}
		\newblock{\em General Topology and its Applications,} 1:17--28, 1971.
		
		
\end{thebibliography}
\end{document}